\documentclass{article}
\usepackage{amsmath, amsthm}

\newcommand{\Dm}{{\rm Dom}}
\newcommand{\R}{{\mathbf R}}
\newcommand{\Z}{{\mathbf Z}}
\newcommand{\N}{{\mathbf N}}
\newcommand{\T}{{\mathbf T}}
\newcommand{\lb}{{\mathbf l}}
\newcommand{\hbf}{{\mathbf h}}
\newcommand{\bxi}{{\mathbf \xi}}

\newcommand{\Dc}{{\mathcal D}}
\newcommand{\Gc}{{\mathcal G}}
\newcommand{\Qc}{{\mathcal Q}}

\newcommand{\Oc}{{\mathcal O}}

    \newcommand    {\e}{{\mathbf x}}
    \newcommand    {\by}{{\mathbf y}}
    \newcommand    {\bq}{{\mathbf q}}
    
\newcommand    {\br}{{\mathbf r}}
    \newcommand    {\bs}{{\mathbf s}}
    
    \newcommand    {\C}{{\mathbf C}}
    \newcommand    {\B}{{\mathcal H}}
    \newcommand    {\Bc}{{\mathcal B}}

    \newcommand    {\J}{{\mathbf p}}
\newcommand    {\hb}{{\mathbf h}}

    \newcommand    {\w}{{\mathbf k}}
    
    \newcommand    {\Rs}{{\mathcal R}}
    \newcommand    {\Ec}{{\mathcal E}}
    \newcommand    {\Fc}{{\mathcal F}}

    \newcommand    {\ifr}{{\rm if}}
    
    \newcommand    {\fr}{{\rm for}}

    \newcommand    {\f}{{\mathcal B}}

\newcommand    {\tb}{{\mathbf t}}
    \newcommand    {\mb}{{\bf m}}

     \newcommand    {\Nc}{{\mathcal N}}

\newtheorem{theorem}{Theorem}[section]
\newtheorem{proposition}[theorem]{Proposition}
\newtheorem{lemma}[theorem]{Lemma}
\newtheorem{corollary}[theorem]{Corollary}
\theoremstyle{definition}
\newtheorem{definition}[theorem]{Definition}
\theoremstyle{remark}
\newtheorem{remark}[theorem]{Remark}
\numberwithin{equation}{section}

    \renewcommand{\theequation}{\thesection.\arabic{equation}}

\title{\Large\bf
Virtual bound levels in a gap of the essential spectrum of the
Schr\"odinger operator with a weakly perturbed periodic potential}

\author
{\large  Leonid Zelenko} \vskip 0.8truecm

\date{}

\begin{document}
\maketitle

 \begin{abstract}
In the space $L_2(\R^d)$ we consider the Schr\"odinger operator
$H_\gamma=-\Delta+ V(\e)\cdot+\gamma W(\e)\cdot$, where
$V(\e)=V(x_1,x_2,\dots,x_d)$ is a periodic function with respect to
all the variables, $\gamma$ is a small real coupling constant and
the perturbation $W(\e)$ tends to zero sufficiently fast as
$|\e|\rightarrow\infty$. We study so called virtual bound levels of
the operator $H_\gamma$, that is those eigenvalues of $H_\gamma$
which are born at the moment $\gamma=0$ in a gap
$(\lambda_-,\,\lambda_+)$ of the spectrum  of the unperturbed
operator $H_0=-\Delta+ V(\e)\cdot$ from an edge of this gap while
$\gamma$ increases or decreases. For a definite perturbation
$(W(\e)\ge 0)$ we investigate the number of such levels and an
asymptotic behavior of them and of the corresponding eigenfunctions
as $\gamma\rightarrow 0$ in two cases: for the case where the
dispersion function of $H_0$, branching from an edge of
$(\lambda_-,\lambda_+)$, is non-degenerate in the Morse sense at its
extremal set and for the case where it has there a non-localized
degeneration of the Morse-Bott type. In the first case in the gap
there is a finite number of virtual eigenvalues if $d<3$ and we
count the number of them, and in the second case in the gap there is
an infinite number of ones, if the codimension of the extremal
manifold is less than $3$. For an indefinite perturbation we
estimate the multiplicity of virtual bound levels. Furthermore, we
show that if the codimension of the extremal manifold is at least
$3$ at both edges of the gap $(\lambda_-,\,\lambda_+)$, then under
additional conditions there is a threshold for the birth of the
impurity spectrum in the gap, that is
$\sigma(H_\gamma)\cap(\lambda_-,\,\lambda_+)=\emptyset$ for a small
enough $|\gamma|$.
\end{abstract}
 \bigskip
\noindent {\bf Mathematics Subject Classification 2000,} Primary:
47F05, Secondary: 47E05, 35Pxx


\tableofcontents

\section{Introduction} \label{sec:introduction}
\setcounter{equation}{0}

\subsection{Background} \label{subsec:background}

In this paper we consider the Schr\"odinger operator
\begin{equation}\label{dfHgm}
H_\gamma=-\Delta+ V(\e)\cdot+\gamma W(\e)\cdot
\end{equation}
acting  in the space $L_2(\R^d)$, where $V(\e)=V(x_1,x_2,\dots,x_d)$
is a periodic function with respect to all the variables and
satisfying some mild condition which will be pointed below, the
function $W(\e)$ is measurable and bounded in $\R^d$ and $\gamma$ is
a small real coupling constant. In what follows we shall impose on
the perturbation $ W(\e)$ some conditions, which mean that it tends
to zero sufficiently fast as $|\e|\rightarrow\infty$ in an integral
sense. We study so called virtual bound levels of the operator
$H_\gamma$, that is those eigenvalues of $H_\gamma$ which are born
at the moment $\gamma=0$ in a gap $(\lambda_-,\,\lambda_+)$ of the
spectrum  of the unperturbed operator $H_0=-\Delta+ V(\e)\cdot$ from
an edge of this gap while $\gamma$ increases or decreases. In
physics they are called "resonance levels" or "trapped
levels".\vskip2mm

A wide literature is devoted to the study of discrete spectrum in
gaps of the essential spectrum. In \cite{Rof} and \cite{Rof1} for
the one-dimensional case (d=1) the tests for finiteness and
infiniteness of the number of eigenvalues of the operator
$H_1=H_\gamma\vert_{\gamma=1}$ in the gaps of the spectrum of the
unperturbed operator $H_0$ (the impurity spectrum) were obtained. In
\cite{Zl1} in the one-dimensional case a local dilative perturbation
of the periodic potential $V(x)$ was studied, which under some
condition adds an infinite number of eigenvalues in a gap of the
spectrum of the unperturbed operator. In \cite{Ros} the asymptotic
formula was obtained for the counting function of negative
eigenvalues of the operator $H_1$ (for $d=1$) accumulating to the
bottom of its essential spectrum in the case where $V(\e)\equiv 0$
(a perturbation of the motion of a "free" electron). In the general
case (a perturbation of the motion of an electron in a periodic
lattice) the analogous asymptotic formula was obtained in
\cite{Zel}, \cite{Zel1}, \cite{Khr}, \cite{Khr1}, \cite{Rai}  and
\cite{Schm} for the counting function of eigenvalues accumulating to
an edge of a gap of the essential spectrum of the operator $H_1$.
The asymptotic behavior of the discrete spectrum of the operator
$H_\gamma$ in a gap of its essential spectrum is well studied for a
"strong coupling", that is for $|\gamma|\rightarrow\infty$
(\cite{Bi2}, \cite{Bi-So}, \cite{Sob}).\vskip2mm

But there is a comparatively small number of results concerning the
behavior of the discrete part of the spectrum of $H_\gamma$ for a
"weak coupling", that is for $|\gamma|\rightarrow 0$. In the paper
of M. Sh. Birman \cite{Bi1} (1961) a variational approach has been
worked out for the study of birth of negative eigenvalues under a
small perturbation in the case $V(\e)\equiv 0$. In the $1970$'s in a
series of papers the asymptotic behavior as $\gamma\rightarrow 0$ of
negative eigenvalues and the corresponding eigenfunctions of the
Schr\"odinger operator $H_\gamma$ was studied for $V(\e)\equiv 0$
with the help of analytical methods (\cite{Re-Si}, \cite{S},
\cite{S1}, \cite{Kl}, \cite{B-G-S}). These investigations were based
on the explicit form of the Green function for the unperturbed
operator $H_0=-\Delta$ and on the Birman-Schwinger principle, which
describes the discrete spectrum of the perturbed operator in the
gaps of the spectrum of the unperturbed one with the help of so
called Birman-Schwinger operator, defined by (\ref{BrmSchwop})
(\cite{Bi3}, \cite{Sc}, \cite{Re-Si}, \cite{S}). \vskip2mm

The interest in this subject was renewed in the last two decades. In
\cite{Wei} T. Weidl has developed the Birman approach for the study
of the existence of virtual eigenvalues for a wide class of elliptic
differential operators of high order and even for indefinite
perturbations. In the papers \cite{Ar-Zl1} and \cite{Ar-Zl2} the
negative virtual eigenvalues were studied for the perturbation
$(-\Delta)^l+\gamma W(\e)\cdot$ of the polyharmonic operator
$(-\Delta)^l$ with the help of an analytical method. The Green
function of the unperturbed operator $(-\Delta)^l$ was not
constructed there explicitly, but by using the Fourier transform
(the "momentum representation") a representation for the resolvent
$((-\Delta)^l-\lambda I)^{-1}$ of the unperturbed operator was
obtained near the bottom $\lambda=0$ of its spectrum, which permit
to get asymptotic formulas for the negative virtual eigenvalues of
the perturbed operator with the help of the Birman-Schwinger
principle. Observe that for the unperturbed operator $(-\Delta)^l$
the dispersion function (dependence of the energy $\lambda$ on the
momentum $\J$) has the form $\lambda=|\J|^{2l}$, hence it has one
minimum point $\J=0$. Thanks this fact, for a short range
perturbation $W(\e)$ (that is, or it has a compact support, or tends
to zero sufficiently fast as $\e\rightarrow\infty$) the perturbed
operator $(-\Delta)^l+\gamma W(\e)\cdot$ has a finite number of
negative virtual eigenvalues. In particular (for $l=1$), this
property holds for the perturbation of the Laplacian. But in
\cite{Ch-M} for $d=2$ an axially symmetric Hamiltonian describing a
spin-orbit interaction was considered such that the minima set of
its dispersion function is a circle (non-localized degeneration of
the dispersion function). By the use of the method of separating
variables it was shown in \cite{Ch-M} that in the presence of an
arbitrarily shallow axially symmetric potential well $W(\e)$ an
infinite number of negative eigenvalues of the perturbed Hamiltonian
appear.  With the physical point of view in this case an infinite
number of electrons have the minimal energy level of the unperturbed
Hamiltonian, hence an infinite number of bound energy levels can
appear beneath this level in the presence of some small impurities
in the system. For a more general situation the analogous result was
established in \cite{Br-G-P} and \cite{Pan} by using the variational
method. In \cite{H-S} a general situation of unperturbed Hamiltonian
("kinetic energy") is considered such that the minima set of its
dispersion function is a submanifold of codimension one. The
existence of an infinite number of virtual negative eigenvalues is
proved there and asymptotic formulas are obtained for them and
corresponding eigenfunctions for a small coupling constant. The
method used in \cite{H-S} is close to one used in \cite{Ar-Zl1},
\cite{Ar-Zl2}. \vskip2mm

In \cite{P-L-A-J} with the help of the variational method the
existence of virtual eigenvalues in gaps of the essential spectrum
of the operator $H_\gamma$, defined by (\ref{dfHgm}), is proved. But
this method does not permit to investigate the asymptotic behavior
of virtual eigenvalues and of the corresponding eigenfunctions for
$\gamma\rightarrow 0$. \vskip2mm

\subsection{Description of methods and results}\label{subsec:description}

In the present paper we investigate the number of virtual
eigenvalues in a gap of the essential spectrum of the operator
$H_\gamma$ and obtain asymptotic formulas for them and for the
corresponding eigenfunctions for $\gamma\rightarrow 0$. The method
we use is close to one used in \cite{Ar-Zl1}, \cite{Ar-Zl2} and
\cite{H-S}, but instead of the Fourier transform we apply the so
called Gelfand-Fourier-Floquet transform (\ref{unit}) (a
``quasi-momentum representation''of the Hamiltonian), which realizes
a unitary equivalence between the unperturbed operator
$H_0=-\Delta+V(\e)\cdot$ and a direct integral over the Brillouin
zone $B$ of a family of operators $\{H(\J)\}_{\J\in B}$, generated
on a fundamental domain $\Omega$ of the lattice of periodicity of
$V(\e)$ by the operation $-\Delta+V(\e)\cdot$ and some cyclic
boundary conditions (\cite{Gel}, \cite{Wil}, \cite{Kuch},
\cite{Zl}). In this situation the role of the dispersion function
for the operator $H_0$ plays the dependence of the energy $\lambda$
(the eigenvalue of $H(\J)$) on the quasi-momentum $\J$:
$\lambda=\lambda(\J)$. Like in \cite{Ar-Zl1}, \cite{Ar-Zl2}, our
considerations are based on a representation of the resolvent
$(H_0-\lambda I)^{-1}$ of the unperturbed operator near the edge
$\lambda=\lambda_+$ $(\lambda=\lambda_-)$ of a gap
$(\lambda_-,\,\lambda_+)$ of its spectrum as a sum of a singular
(w.r.t. $\lambda$) part and a regular remainder. We consider this
representation in two cases of behavior of the dispersion function
$\lambda(\J)$ branching from the edge $\lambda_+$ $(\lambda_-)$ of
the gap $(\lambda_-,\,\lambda_+)$ and having the extremal set
$F^+=\lambda^{-1}(\lambda_+)$ $(F^-=\lambda^{-1}(\lambda_-))$: the
case of a non-degenerate edge $\lambda_+$ $(\lambda_-)$ (Proposition
\ref{prrprres}) of the Morse's type in the sense that at any point
of the set $F^+\;(F^-)$ the Hessian operator of $\lambda(\J)$ is
non-degenerate, and the case of a non-localized degeneration of the
Bott-Morse type (Proposition \ref{prrprres1}) in the sense that the
set $F^+\;(F^-)$ is a smooth submanifold of a non-zero dimension and
at each point of $F^+\;(F^-)$ the Hessian operator of $\lambda(\J)$
is non-degenerate along the normal subspace to $F^+\;(F^-)$. In both
cases we assume that $\lambda(\J)$ is a simple eigenvalue of the
operator $H(\J)$ for any $\J\in F^+\;(\J\in F^-)$\footnote{As it was
shown in \cite{K-R}, this situation is generic in the sense that all
the edges of gaps of the spectrum of the operator $H_0$ are simple
for a dense $G_\delta$-set of periodic potentials $V\in
L_\infty(\Omega)$.}. As it is clear, if the edge $\lambda_+$
$(\lambda_-)$ of a gap $(\lambda_-,\,\lambda_+)$ is non-degenerate,
then the set $F^+\;(F^-)$ is finite. Notice that the representation
of the resolvent of $H_0$ in the non-degenerate case given by
Proposition \ref{prrprres} is close to one given by Corollary 4.2 of
\cite{Ger}, but in our case we estimate also the integral kernel of
the remainder of this representation. These representations of the
resolvent enable us to extract a singular portion from the
Birman-Schwinger operator (Propositions \ref{rprBirSchw},
\ref{rprBirSchwdeg}), which yields the leading terms of the
asymptotic formula for the virtual eigenvalues of $H_\gamma$ as
$|\gamma|\rightarrow 0$. It is known that in the one-dimensional
case ($d=1$) the edges of all the spectral gaps of the operator
$H_0$ are non-degenerate (\cite{Titch}). As it was shown in
\cite{Kir-S} (Theorem 2.1), the bottom of the spectrum of the
operator $H_0$ is non-degenerate even in the multi-dimensional case.
Taking the periodic potential in the form
$V(\e)=\sum_{k=1}^dV_k(x_k)\;(\e=(x_1,x_2,\dots,x_d))$ and using the
method of separating variables, it is not difficult to construct an
example of the multi-dimensional operator $H_0$ having a finite
spectral gap with non-degenerate edges.\vskip2mm

In Theorem \ref{thnondegedg} we count the number of virtual
eigenvalues of $H_\gamma$ in a spectral gap
$(\lambda_-,\,\lambda_+)$ of $H_0$ being born from its
non-degenerate edge $\lambda_+\;(\lambda_-)$ and obtain asymptotic
formulas for them as $|\gamma|\rightarrow 0$ for a definite
perturbation (that is, $W(\e)\ge 0$ on $\R^d$) under the assumption
that $W(\e)$ tends to zero sufficiently fast in an integral sense.
For $d=1$ there is only one virtual eigenvalue being born from the
edge $\lambda_+\;(\lambda_-)$ of the spectral gap of $H_0$ for
$\gamma<0\;(\gamma>0)$ and the leading term of the asymptotic
formula for the distance between this virtual eigenvalue and the
edge $\lambda_+\;(\lambda_-)$ has the order $O(\gamma^2)$ as
$\gamma\uparrow 0\;(\gamma\downarrow 0)$. For $d=2$ the number of
virtual eigenvalues being born from the edge
$\lambda_+\;(\lambda_-)$ of the spectral gap of $H_0$ for
$\gamma<0\;(\gamma>0)$ coincides with the number of points of the
extremal set $F^+\;(F^-)$ of the dispersion function branching from
$\lambda_+\;(\lambda_-)$ and the leading terms of the asymptotic
formulas for the distances between these virtual eigenvalues and the
edge $\lambda_+\;(\lambda_-)$ has an exponential order as
$\gamma\uparrow 0\;(\gamma\downarrow 0)$. Furthermore, Theorem
\ref{thnondegedg} claims that for $d\le 2$ the eigenfunctions
corresponding to the virtual eigenvalues of $H_\gamma$ converge  in
some sense as $\gamma\uparrow 0\;(\gamma\downarrow 0)$ with the rate
$O(\gamma)$ to linear combinations of Bloch functions of $H_0$
corresponding to the energy level $\lambda_+\;(\lambda_-)$ and the
quasi-momenta from $F^+\;(F^-)$. For $d\ge 3$ there is a threshold
for the birth of the impurity spectrum of $H_\gamma$ in the spectral
gap $(\lambda_-,\,\lambda_+)$ of $H_0$ for $\gamma<0\;(\gamma>0)$,
that is for a small enough $\gamma<0\;(\gamma>0)$ there is no
eigenvalue of $H_\gamma$ in $(\lambda_-,\lambda_+)$. Furthermore,
for $d=2$ Theorem \ref{thnondegedg} yields an asymptotic formula of
Lieb-Thirring type for the sum of inverse logarithms of distances
between the virtual eigenvalues and the edge
$\lambda_+\;(\lambda_-)$, and its leading term is expressed
explicitly via the perturbation $W(\e)$ and spectral characteristics
of the unperturbed operator $H_0$ at this edge: Bloch functions and
effective masses of electrons at the energy level
$\lambda_+\;(\lambda_-)$. \vskip2mm

Theorem \ref{thdegedg} describes the birth of virtual eigenvalues of
$H_\gamma$, under a definite perturbation, from the edge
$\lambda_+\;(\lambda_-)$ of a spectral gap $(\lambda_-,\,\lambda_+)$
of $H_0$, which has a non-localized degeneration of the Morse-Bott
type. The character of this birth depends on the codimension
$\rm{codim}(F^+)\;(\rm{codim}(F^-))$ of the extremal submanifold
$F^+\;(F^-)$ of the dispersion function branching from
$\lambda_+\;(\lambda_-)$. For $\rm{codim}(F^+)\le
2\;(\rm{codim}(F^-)\le 2)$ there is an infinite number of virtual
eigenvalues of $H_\gamma$ being born from this edge, and the
asymptotic behavior of them as $\gamma\uparrow 0\;(\gamma\downarrow
0)$ for $\rm{codim}(F^+)=1\;(\rm{codim}(F^-)=1)$ and
$\rm{codim}(F^+)=2\;(\rm{codim}(F^-)=2)$ is analogous to one in the
non-degenerate case for $d=1$ and $d=2$ respectively, and
furthermore, the asymptotic behavior of the corresponding
eigenfunctions is analogous to one in the non-degenerate case for
$d\le 2$. If $\rm{codim}(F^+)\ge 3\;(\rm{codim}(F^-)\ge
 3)$, then for $\gamma<0\;(\gamma>0)$ there is a threshold for the birth of the impurity spectrum of
$H_\gamma$ in the spectral gap $(\lambda_-,\,\lambda_+)$ of $H_0$
like in the non-degenerate case for $d\ge 3$. Furthermore, for
$\rm{codim}(F^+)\le 2\;(\rm{codim}(F^-)\le 2)$ Theorem
\ref{thdegedg} yields a weak version of the asymptotic formula of
Lieb-Thirring type, mentioned above, for the sum of square roots of
distances between virtual eigenvalues and the edge
$\lambda_+\;(\lambda_-)$ if
$\rm{codim}(F^+)=1\;(\rm{codim}(F^-)=1)$, and for the sum of inverse
logarithms of ones if $\rm{codim}(F^+)=2\;(\rm{codim}(F^-)=2)$. The
leading terms of these formulas have the form of the integrals over
$F^+\;(F^-)$, whose integrands are expressed explicitly via the
perturbation $W(\e)$ and spectral characteristics of the unperturbed
operator $H_0$ at the edge $\lambda_+\;(\lambda_-)$, mentioned above
(merely in this case the effective masses are computed in the
directions normal to the extremal submanifold). \vskip2mm

Theorem \ref{thestmult} treats the case of an indefinite
perturbation ($W(\e)$ may change the sign) and the non-degenerate
edge $\lambda_+\;(\lambda_-)$. It yields an estimate of the
multiplicity of virtual eigenvalues if $d\le 2$, and for $d\ge 3$ it
claims the existence of a threshold for the birth of the impurity
spectrum from this edge. \vskip2mm

Theorem \ref{ththresholddeg} claims the existence of a threshold for
the birth of the impurity spectrum from the degenerate edge
$\lambda_+\;(\lambda_-)$ for an indefinite perturbation, if
$\rm{codim}(F^+)\ge 3\;(\rm{codim}(F^-)\ge 3)$. \vskip2mm

In the Appendix we prove Theorem \ref{thmainApp}, which  yields some
kind of an elliptic regularity result: under a mild condition for
the periodic potential $V(\e)$ it claims that if a branch of
eigenvalues $\lambda(\J)$ of the family of operators $H(\J)$,
mentioned above, is holomorphic and a branch of corresponding
eigenfunctions $b(\cdot,\J)$ is holomorphic in the
$L_2(\Omega)$-norm, then the latter branch is holomorphic in the
$C(\Omega)$-norm. For $d\le 3$ this claim follows immediately from
results of the paper \cite{Wil}, but for $d\ge 4$ the arguments used
there fail and we use a modification of them. Corollary
\ref{cormainApp} of Theorem \ref{thmainApp} is used in the proof of
the main results.\vskip2mm

The paper is organized as follows. After this Introduction, in
Section \ref{sec:notation} we introduce some basic notation, in
Section \ref{sec:preliminaries} (Preliminaries) we recall some known
facts concerning the operator $H_0$ with the periodic potential and
define the notion of a virtual eigenvalue. In Section
\ref{sec:mainres} we formulate the main results. In Section
\ref{sec:prmainres} we prove the main results. In Section
\ref{sec:resunprop1} we obtain the representation of the resolvent
of the operator $H_0$, mentioned above. Section \ref{sec:appendix}
is the Appendix. We add the label ``A'' to numbers of claims and
formulas from the Appendix.

\section{Basic notation} \label{sec:notation}
\setcounter{equation}{0}

$\e\cdot\by\;(\e,\by\in\R^d)$ is the canonical inner product in the
real vector space $\R^d$; $|\e|=\sqrt{\e\cdot\e}$ is the Euclidean
norm in $\R^d$;\vskip1mm

$S^{d}$ is the $d$-dimensional unit sphere:
$S^d=\{\e\in\R^{d+1}:\;|\e|=1\}$; $s_d$ is the $d$-dimensional
volume of $S^{d}$;\vskip1mm

$\T^d=\times_{k=1}^d S^1$ is the $d$-dimensional torus;\vskip1mm

$\Z$ is the ring of integers;\vskip1mm

$\Z^d=\times_{k=1}^d \Z$;\vskip1mm

$\Dm(f)$ is the domain of a mapping $f$;\vskip1mm

$(f,g)\;(f,g\in\B)$, $\Vert f\Vert$ are the inner product and the
norm in a complex Hilbert space $\B$ (in particular, in
$L_2(\R^d)$); the norm of linear bounded operators, acting in $\B$,
is denoted in the same manner;\vskip1mm

$L_{2,0}(\R^d)$ is the set of all functions from $L_2(\R^d)$ having
compact supports;\vskip2mm

If $A$ is a closed linear operator acting in a Hilbert space $\B$,
then:\vskip1mm

$\ker(A)$ is the kernel of $A$, i.e.
$\ker(A)=\{x\in\B:\;Ax=0\}$;\vskip1mm

$\mathrm{Im}(A)$ is the image of $A$;\vskip1mm

$\sigma(A)$ is the spectrum of $A$;\vskip1mm

$\Rs(A)$ is the resolvent set of $A$, i.e.
$\Rs(A)=\C\setminus\sigma(A)$;\vskip1mm

$R_\lambda(A)\;\;(\lambda\in\Rs(A))$ is the resolvent of $A$, i.e.
$R_\lambda(A)=(A-\lambda I)^{-1}$.\vskip1mm

$\f(E)$ is the Banach space of linear bounded operators, acting in a
Banach space $E$.\vskip2mm

Some specific notation will be introduced in what follows.

\section{Preliminaries}\label{sec:preliminaries}
\setcounter{equation}{0}

\subsection{Spectral characteristics of the unperturbed operator} \label{subsec:spectchar}
\setcounter{equation}{0}

Consider the unperturbed operator $H_0=-\Delta+ V(\e)\cdot$.
 We assume that the potential $V(\e)$ is
periodic on the lattice
$\Gamma=\{\lb\in\R^d\;|\;\lb=(l_1T_1,l_2T_2,\dots,l_dT_d),\;\w=(l_1,l_2,\dots,l_d)\in\Z^d\}$
($T_k>0\,(k=1,2,\dots,d)$), that is $V(\e+\lb)=V(\e)$ for any
$\e\in\R^d$ and $\lb\in\Gamma$. For the simplicity we shall assume
that $T_1=T_2=\dots=T_d=1$, that is $\Gamma=\Z^d$. Denote by
$\Omega$ the fundamental domain of the lattice $\Gamma$:
$\Omega:=\times_{k=1}^d[-1/2,1/2]$. Furthermore, assume that
\begin{equation}\label{condperpotent}
V\in\left\{\begin{array}{ll}
L_2(\Omega), &\rm{if}\quad d\le 3\\
\bigcup_{q>\frac{d}{2}}L_q(\Omega), &\rm{if}\quad d\ge 4.
\end{array}\right.
\end{equation}
By claim (ii) of Proposition \ref{propselfadj}, the operator $H_0$
with the domain $W_2^2(\R^d)$ is self-adjoint and bounded below.

For any $\J\in\T^d$ consider the operator $H(\J)$, generated by the
operation $h=-\Delta+V(\e)\cdot$ in the Hilbert space $\B_{\J}$ of
functions $u\in L_{2,loc}(\R^d)$ satisfying the condition
\begin{equation}\label{Htau}
u(\e+\lb)=\exp(i\J\cdot\lb)u(\e)\quad \forall\;\e\in\R^d,\;
\lb\in\Gamma
\end{equation}
with the inner product and the norm, defined by
\begin{equation}\label{dfinnprd}
(f,g)_2=\int_\Omega f(\e)\overline{g(\e)}\,d\e\;(f,g\in\B_\J),\quad
\|f\|_2=\sqrt{(f,f)_2}.
\end{equation}
The domain of $H(\J)$ is the linear set
$\Dc_\J(\Gamma)=W_{2,loc}^2(\R^d)\cap\B_{\J}$. By claim (i) of
Proposition \ref{propselfadj} and Proposition \ref{prcompres}, the
operator $H(\J)$ is self-adjoint, bonded below uniformly w.r. to
$\J\in\T^d$ and its spectrum is discrete. As it is easy to check,
the operator $(E_\J u)(\e):=\exp(-i\J\cdot\e)u(\e)$, acting from
$\B_\J$ onto $\B_0=L_2(\R^d/\Gamma)$, realizes a unitary equivalence
between the operator $H(\J)$ and the operator
\begin{equation}\label{b16}
\tilde H(\J)=-\Delta_\J+V(\e)\cdot,
\end{equation}
with the domain $W_2^2(\R^d/\Gamma)$, where
\begin{equation}\label{b17}
\Delta_\J=\sum_{j=1}^d\left(D_j +ip_j\right)^2.
\end{equation}
Let $\lambda_1(\J)\le\lambda_2(\J)\le\dots\lambda_n(\J)\le\dots$ be
the eigenvalues of the operator $\tilde H(\J)$ (counting their
multiplicities), and $e_1(\e,\J),e_2(\e,\J),\dots,e_n(\e,\J),\dots$
be the corresponding eigenfunctions of this operator which form an
orthonormal basis in the space $L_2(\R^d/\Gamma)$. Using the
physical terminology, we shall call the vector $\J$ the {\it
quasi-momentum} and each branch of the eigenvalues $\lambda_n(\J)$
will be called the {\it dispersion function}. It is known that
$\lambda_n(\J)$ are continuous functions on $\T^d\;$ and
$\sigma(H_0)=\bigcup_{\J\in\T^d}\sigma(\tilde
H(\J))=\bigcup_{\J\in\T^d}\{\lambda_n(\J)\}_{n=1}^\infty$
(\cite{Gel}, \cite{Eas}, \cite{Eas1} \cite{Kuch}, \cite{Zl}). We
shall consider also the {\it Bloch function} corresponding to a
dispersion function $\lambda_l(\J)$ and a quasi-momentum
$\J\in\T^d$:
\begin{equation}\label{Bloch1}
b_l(\e,\J)=\exp(i\J\cdot\e)e_l(\e,\J).
\end{equation}
It is clear that for each natural $l$ $b_l(\e,\J)$ is an
eigenfunction of the operator $H(\J)$, corresponding to its
eigenvalue $\lambda_l(\J)$ and for any fixed $\J\in\T^d$ the
sequence $\{b_l(\e,\J)\}_{l=1}^\infty$ forms an orthonormal basis in
the space $\B_\J$.

It is easy to check that for any $\J\in\T^d$ $JH(\J)J=H(-\J)$, where
$J$ is the conjugation operator $(Jf)(f)(\e):=\overline {f(\e)}$
(the property of a ``time reversibility''). Hence in this case
$\sigma(H(\J))=\sigma(H(-\J))$ and for any $\mu\in\sigma(H(\J))$ the
corresponding eigenprojections $Q_\mu(\J)$ of $H(\J)$ and
$Q_\mu(-\J)$ of $H(-\J)$ are connected in the following manner:
$Q_\mu(-\J)=JQ_\mu J(\J)$. The same property is valid for the
operator $\tilde H(\J)$.

Assume that $(\lambda_-,\,\lambda_+)$ is a gap of $\sigma(H_0)$,
that is for some $j\ge 1$
$\;\lambda_-=\max_{\J\in\T^d}\lambda_j(\J)<\lambda_+=\min_{\J\in\T^d}\lambda_{j+1}(\J)$,
and $\lambda_+=\min_{\J\in\T^d}\lambda_1(\J)$, $\lambda_-=-\infty$
for $j=0$. Since $j$ will be fixed in our considerations, we shall
denote $\lambda^-(\J):=\lambda_j(\J)$ and
$\lambda^+(\J):=\lambda_{j+1}(\J)$. In other words, $\lambda^-(\J)$
and $\lambda^+(\J)$ are the dispersion functions branching from the
edges $\lambda_-$ and $\lambda_+$ respectively.  We shall denote by
$b^-(\e,\J)$ and $b^+(\e,\J)$ the eigenfunctions of $H(\J)$ (Bloch
functions), corresponding to $\lambda^-(\J)$ and $\lambda^+(\J)$
respectively, i.e. $b^-(\e,\J)=b_j(\e,\J)$ and
$b^+(\e,\J)=b_{j+1}(\e,\J)$. In the analogous manner we denote
$e^-(\e,\J)=e_j(\e,\J)$ and $e^+(\e,\J)=e_{j+1}(\e,\J)$.

Consider the following subsets of $\T^d$:
\begin{eqnarray*}
&&F^-:=\{\J\in\T^d:\;\lambda^-(\J)=\lambda_-\},\\
&&F^+:=\{\J\in\T^d:\;\lambda^+(\J)=\lambda_+\},
\end{eqnarray*}
which are the extremal sets of the functions $\lambda^-(\J)$ and
$\lambda^+(\J)$ respectively. Assume that for the edge
$\lambda_-\;(\lambda_->-\infty)$ or for the edge $\lambda_+$ of the
gap $(\lambda_-,\,\lambda_+)$ the condition is fulfilled:\vskip2mm

(A) The edge $\lambda_+\;(\lambda_-)$ is
 non-degenerate in the Morse's sense, that is\vskip1mm

(a) it is simple in the sense that for any $\J_0\in F^+\;(\J_0\in
F^-)$ the number
$\lambda^+(\J_0)=\lambda_+\;(\lambda^-(\J_0)=\lambda_-)$ is  a
simple eigenvalue of the operator $H(\J_0)$, that is
$(\lambda^+(\J_0)<\lambda_{j+2}(\J_0))$, if $j\ge 0$
($\lambda^-(\J_0)>\lambda_{j-1}(\J_0)$, if $j\ge 2$);\vskip1mm

(b) for any $\J_0\in F^+\;(\J_0\in F^-)$ the
 second differential $d^2\lambda^+(\J_0)\;(d^2\lambda^-(\J_0))$
 is a positive-definite (negative-definite) quadratic form.\vskip2mm

 Observe that if condition (A)-(a) is satisfied for the edge $\lambda_+\;(\lambda_-)$, then by
 claim (i) of Corollary \ref{cormainApp}, for any
 $\J_0\in F^+\;(\J_0\in F^-)$ there exists a neighborhood
 $\Oc^+(\J_0)\;(\Oc^-(\J_0))$ of $\J_0$ such that the function
 $\lambda^+(\J)\;(\lambda^-(\J))$ is real-analytic
 in $\Oc^+(\J_0)\;(\Oc^-(\J_0))$ (hence in particular there exists
 the second differential of this function, taking part in the
 condition (A)-(b)), and furthermore, the corresponding branch of
 eigenfunctions $b^+(\e,\J)\;(b^-(\e,\J))$ of $H(\J)$ can be chosen
 such that for
 each fixed $\J\in\Oc^+(\J_0)\;(\J\in\Oc^-(\J_0))$, $\Vert b^+(\cdot,\J) \Vert_2=1$ ($\Vert b^-(\cdot,\J) \Vert_2=1$),
 the function $b^+(\e,\J)\;(b^-(\e,\J))$ is continuous and the mapping
  $\J\rightarrow b^+(\e,\J)\in C(\Omega)\;(\J\rightarrow b^-(\e,\J)\in
C(\Omega))$ is real-analytic in $\Oc^+(\J_0)\;(\Oc^-(\J_0))$. It is
clear that the corresponding eigenfunction
$e^+(\e,\J)\;(e^-(\e,\J))$ of the operator $\tilde H(\J)$ has the
same properties.

It is clear that if the condition (A) is satisfied for
$\lambda_+\;(\lambda_-)$, then the set $F^+\;(F^-)$ is finite, that
is $F^+=\{\J_1^+,\J_2^+,\dots,\J_{n_+}^+\}$
$(F^-=\{\J_1^-,\J_2^-,\dots,\J_{n_-}^-\})$. In particular, it is
known (\cite{Titch}) that in the case where $d=1$  condition (A) is
always satisfied, all the functions $\lambda_l(\J)\,(l=1,2,\dots)$
are even, $n_+=n_-=1$ and or $p_1^+=p_1^-=0$, or
$p_1^+=p_1^-=\pi$.\footnote{For $d=1$ the gaps of the spectrum of
the operator $H_0$ are or $(-\infty,\lambda_0)$, or
$(\lambda_k,\lambda_{k+1})$, or
$(\mu_k,\mu_{k+1})\;(k=0,1,2,\dots)$, where $\lambda_k$ and $\mu_k$
are the eigenvalues of the operators $H(0)$ and $H(\pi)$
respectively, and
$\lambda_0<\mu_0\le\mu_1<\lambda_1\le\lambda_2<\dots$} Hence in this
case we shall write $p_1$ instead of $p_1^+$ and $p_1^-$. Denote
\begin{equation}\label{dfmkpl}
m_k^+=\big(J\,\rm{Hes}_{\J_k^+}(\lambda^+)\big)^{-1}\quad
\Big(m_k^-=-\big(J\,\rm{Hes}_{\J_k^-}(\lambda^-))\big)^{-1}\Big)
\end{equation}
where
\begin{eqnarray}\label{Hess}
&&J\,\rm{Hes}_\J(\lambda^+)=\rm{det}\left(\frac{\partial^2\lambda^+(\J)}{\partial
p_\mu\partial p_\nu}|_{\J=\J_k^+}\right)_{\mu,\nu=1}^d\quad\\
&&\left(J\,\rm{Hes}_\J(\lambda^-)=\rm{det}\left(\frac{\partial^2\lambda^-(\J)}{\partial
p_\mu\partial
p_\nu}|_{\J=\J_k^-}\right)_{\mu,\nu=1}^d\right).\nonumber
\end{eqnarray}
In particular, in the case where $d=1$
\begin{equation}\label{dfmupl}
(m^+_1)^{-1}=\frac{d^2\lambda^+(p)}{dp^2}|_{p=p_1}\quad
\left((m^-_1)^{-1}=-\frac{d^2\lambda^-(p)}{dp^2}|_{p=p_1}\right)
\end{equation}
$(p_1\in\{0,\pi\})$. With the physical point of view the quantity
$m_k^+\;(m_k^-)$ is (up to a physical constant multiplier) the
modulus of the determinant of the effective-mass tensor, that is the
product of effective masses (in the principal directions) of an
electron having the quasi-momentum $\J_k^+\;(\J_k^-)$ at the energy
level $\lambda_+\;(\lambda_-)$.

Along with the  Bloch functions $b^+(\e,\J)$, $b^-(\e,\J)$  we shall
consider the {\it weighted Bloch functions}
\begin{equation}\label{dfvlxp}
v^+(\e,\J)=\sqrt{W(\e)}b^+(\e,\J),\quad
v^-(\e,\J)=\sqrt{W(\e)}b^-(\e,\J),
\end{equation}
corresponding to them, with $W(\e)\ge 0$ a.e. on $\R^d$, and denote
\begin{equation}\label{dfvlpl}
v_k^+(\e)=v^+(\e,\J_k^+),\quad v_l^-(\e)=v^-(\e,\J_l^-).
\end{equation}
In particular, in the case where $d=1$
\begin{equation}\label{dfvlpld1}
v_1^+(x)=v^+(x,p_1),\quad v_1^-(x)=v^-(x,p_1)\quad
(p_1\in\{0,\pi\}).
\end{equation}

A part of our results concerns the case of degenerate edges of the
gap of the spectrum of the unperturbed operator $H_0$, where
condition (A)-(b) does not fulfilled. Let $\mathrm{Hes}_\J(f)$ be
the Hessian operator of a function $f:\,\T^d\rightarrow\R$ at a
point $\J\in\T^d$, that is this is a linear operator acting in the
tangent space $T_\J(\T^d)$ to the torus $\T^d$ at the point $\J$ and
defined by
\begin{equation}\label{dfHes}
\forall\;s,t\in T_\J(\T^d):\quad d^2f(\J)[s,t]=\rm{Hes}_\J(f)s\cdot
t.
\end{equation}
The more general condition than (A), which we shall consider, is
following:\vskip2mm

(B) The function $\lambda^+\;(\lambda^-)$ and the set $F^+\;(F^-)$
satisfy the Morse-Bott type conditions (\cite{Ban-Hur}):

(a) condition (A)-(a) is satisfied;\vskip1mm

(b) the set $F^+\;(F^-)$ consists of a finite number of disjoint
connected components:
$F^+=\bigcup_{k=1}^{n_+}F^+_k\;(F_-=\bigcup_{k=1}^{n_-}F^-_k)$, such
that each of $F^+_k\;(F^-_k)$ is a $C^\infty$-smooth submanifold of
$\T^d$ of the dimension $d_k^+\;(d_k^-)$;\vskip1mm

(c) for any  point $\J\in F^+\;(\J\in F^-)$ the normal Hessian
\begin{eqnarray}\label{dfnrmHes}
&&JN\,\rm{Hes}_\J(\lambda^+):=\det\big(Hes_\J(\lambda^+)\vert_{N_\J}\big)\\
&&\big(JN\,
\rm{Hes}_\J(\lambda^-):=\det\big(Hes_\J(\lambda^-)\vert_{N_\J}\big)\big)\nonumber
\end{eqnarray}
is not equal to zero. Here $N_\J\subseteq T_\J(\T^d)$  is a normal
subspace to $F^+\;(F^-)$ at $\J$, that is $N_\J^+=T_\J(\T^d)\ominus
T_\J(F^+)\;(N_\J^-=T_\J(\T^d)\ominus T_\J(F^-))$.\vskip1mm

If $d_k^+\ge 1\;(d_k^-\ge 1)$ for at least one $k$, we shall say
that the edge $\lambda_+\;(\lambda_-)$ is {\it degenerate in the
Morse-Bott sense}.\vskip2mm

Denote
\begin{eqnarray}\label{dfmplp}
&&m^+(\J)=\big(JN\,\rm{Hes}_\J(\lambda^+)\big)^{-1}\quad(\J\in
F^+)\\
&&\Big(m^-(\J)=-\big(JN\,\rm{Hes}_\J(\lambda^-)\big)^{-1}\quad(\J\in
F^-)\Big).\nonumber
\end{eqnarray}
In the similar manner as above the quantity $m^+(\J)\;(m^+(\J))$ is
(up to a physical constant multiplier) the product of effective
masses (in the principal directions normal to $F^+\;(F^-)$ ) of an
electron having the quasi-momentum $\J\in F^+\;(\J\in F^-)$ at the
energy level $\lambda_+\;(\lambda_-)$.

Assuming that condition (A)-(a) is satisfied and taking $\J_0\in
F^+$, consider the integral kernel $\Qc^+(\e,\bs,\J)$ of the
eigenprojection $Q^+(\J)$ of $H(\J)$, corresponding to the
dispersion function $\lambda^+(\J)$ branching from the edge
$\lambda_+$ and defined in a neighborhood $\Oc^+(\J_0)$ of $\J_0$.
We shall call it the {\it eigenkernel} of $H(\J)$, corresponding to
$\lambda^+(\J)$. By claims (ii) and (iii) of Corollary
\ref{cormainApp}, after a suitable choice of $\Oc^+(\J_0)$ this
kernel acquires the form
$\Qc^+(\e,\bs,\J)=b^+(\e,\J)\overline{b^+(\bs,\J)}$, it does not
depend on the choice of a branch of Bloch functions $b^+(\e,\J)$
having the properties mentioned above and the mapping
$\J\rightarrow\Qc^+(\cdot,\cdot,\J)\in C(\Omega\times\Omega)$  is
real-analytic in $\Oc^+(\J_0)$. The analogous properties has the
eigenkernel $\Qc^-(\e,\bs,\J)$ of $H(\J)$, corresponding to
$\lambda^-(\J)$. It is clear that the eigenkernel of $\tilde H(\J)$,
corresponding to $\lambda^+(\J)\;(\lambda^-(\J))$ has the form
\begin{eqnarray}\label{connecteigkern}
&&\tilde\Qc^+(\e,\bs,\J)=\exp(-i\J\cdot(\e-\bs))\Qc^+(\e,\bs,\J)\\
&&(\tilde\Qc^-(\e,\bs,\J)=\exp(-i\J\cdot(\e-\bs))\Qc^-(\e,\bs,\J))\nonumber
\end{eqnarray}
and it has the same properties as
$\Qc^+(\e,\bs,\J)\;(\Qc^-(\e,\bs,\J))$, but in addition it is
$\Gamma$-periodic w.r.t. $\e$ and $\bs$.  We shall consider also the
{\it weighted eigenkernels}
\begin{eqnarray}\label{dfWxsp}
&&\Qc_W^+(\e,\bs,\J)=\sqrt{W(\e)}\Qc^+(\e,\bs,\J)\sqrt{W(\bs)},\\
&&\Qc_W^-(\e,\bs,\J)=\sqrt{W(\e)}\Qc^-(\e,\bs,\J)\sqrt{W(\bs)}.\nonumber
\end{eqnarray}

\subsection{The notion of a virtual eigenvalue}
\label{subsec:notionsfacts}

Before formulating the main results, let us recall some notions and
facts from \cite{Ar-Zl1}. Consider an operator $H_\gamma=H_0+\gamma
W$ acting in a Hilbert space $\B$, where $H_0$ and $W$ are
self-adjoint operators and $\gamma$ is a real coupling constant. We
assume that the following conditions are satisfied:\vskip2mm

(1) $(\lambda_-,\,\lambda_+)\;\;(-\infty\le \lambda_-<\lambda_+\le
+\infty)$
    is a gap of the spectrum $\sigma(H_0)$
    of the unperturbed operator $H_0$.\vskip1mm

(2) The operator $W$ is bounded\footnote{In \cite{S1} and \cite{Bi2}
the perturbation $W$ is not bounded operator in general, it is
supposed to be only relatively compact w.r.t. $H_0$ in the sense of
quadratic forms. But in the present paper we consider for the
simplicity only the case of a bounded perturbation and we think that
all our results can be obtained without difficulties under a more
general assumption.} and for some $\lambda_0\in\Rs(H_0)$ the
operator $R_{\lambda_0}(H_0)|W|^{\frac{1}{2}}$ is compact.\vskip2mm

By Proposition \ref{BrScwspct} of the present paper, the set
$\sigma(H_\gamma)\cap (\lambda_-,\lambda_+)$ consists of at most
countable number of eigenvalues having finite multiplicities which
can cluster only to the edges $\lambda_+$ and $\lambda_-$:
$\lambda_-<\dots\le\rho_{-k}(\gamma)\le\dots\le\rho_{-1}(\gamma)
\le\rho_0(\gamma)\le\rho_1(\gamma)\le\dots
\le\rho_k(\gamma)\le\dots<\lambda_+$ (each eigenvalue is repeated
according to its multiplicity).

\begin{definition}\label{gone1}
{\rm Let $\rho(\gamma)\in (\lambda_-,\lambda_+)$ be a branch of
eigenvalues of the operator $H_\gamma$ which enters the gap
$(\lambda_-,\,\lambda_+)$ of $\sigma(H_0)$ across the edge
$\lambda_+<+\infty$ at the moment $\gamma=0$ as $\gamma$ decreases
(increases) from $0$ to a negative (positive) value. We call it a}
{\it branch of virtual eigenvalues} {\rm of the operator $H_\gamma$
created in $(\lambda_-,\lambda_+)$ at the edge $\lambda_+$. This
means that $(-\bar\gamma,\,0)\subseteq\Dm(\rho)$
$\big((0,\,\bar\gamma)\subseteq\Dm(\rho)\big)$ for some
$\bar\gamma>0$ and $ \lim_{\gamma\rightarrow
0}\rho(\gamma)=\lambda_+$. In the analogous manner we define a} {\it
branch of virtual eigenvalues} {\rm of the operator $H_\gamma$
created in $(\lambda_-,\lambda_+)$ at the edge $\lambda_->-\infty$}.
\end{definition}

\section{Formulation of main results}
\label{sec:mainres} \setcounter{equation}{0}

\subsection{The birth of virtual eigenvalues from a non-degenerate edge}
\label{subsec:mainresbirthnondeg}

Let us return to the Schr\"odinger operator $H_\gamma=H_0+\gamma W$
($H_0=-\Delta+ V(\e)\cdot$, $W=W(\e)\cdot$, $\B=L_2(\R^d)$),
considered in Section \ref{sec:preliminaries}. Recall that $V(\e)$
is measurable, bounded and periodic on some lattice $\Gamma$. In
this section we consider the case of a definite perturbation and of
a non-degenerate edge of a gap $(\lambda_-,\,\lambda_+)$ of the
spectrum of the unperturbed operator $H_0$, that is at least one of
the edges $\lambda_+$ or $\lambda_-$ satisfies condition (A) of
Section \ref{subsec:spectchar}. Consider the finite rank operator
\begin{equation}\label{dfGW}
G_W^+=\sum_{k=1}^{n_+}\sqrt{m_k^+}\,(\,\cdot,\,v_k^+)v_k^+\quad\Big(G_W^-=
\sum_{k=1}^{n_-}\sqrt{m_k^-}\,(\,\cdot,\,v_k^-)v_k^-\Big),
\end{equation}
$m_k^+\;(m_k^+)$ is defined by (\ref{dfmkpl})-(\ref{Hess}) and
$v_k^+(\e)\;(v_k^-(\e))$ is the weighted Bloch function defined by
(\ref{dfvlpl}), (\ref{dfvlxp}) and (\ref{Bloch1}). We shall show in
what follows (Lemma \ref{lmGWnondeg}) that $G_W^+\;(G_W^-)$ has
$n_+\;(n_-)$ positive eigenvalues
\begin{equation}\label{eigvalGWnondeg}
\nu_1^+\ge \nu_2^+\ge\dots\ge
\nu_{n_+}^+>0\;\big(\nu_1^-\ge\nu_2^-\ge\dots\ge \nu_{n_-}^->0\big)
\end{equation}
(counting their multiplicities), which are eigenvalues of the matrix
\begin{eqnarray}\label{Gram}
&&\left((m_k^+m_l^+)^{\frac{1}{4}}(v_l^+,v_k^+)\right)_{k,l=1}^{n^+}\\
&&\left(\left(
(m_k^-m_l^-)^{\frac{1}{4}}(v_l^-,v_k^-)\right)_{k,l=1}^{n^-}\right).\nonumber
\end{eqnarray}
Let
$g_1^+(\e),g_2^+(\e),\dots,g_{n_+}^+(\e)\;\big(g_1^-(\e),g_2^-(\e),\dots,g_{n_+}^-(\e)\big)$
be an orthonormal sequence of eigenfunctions of the operator
$G_W^+\;(G_W^-)$ corresponding to its eigenvalues
(\ref{eigvalGWnondeg}).

If $d\le 2 $, we shall impose on the non-negative perturbation
$W(\e)$ the following conditions of its fast decay as
$|\e|\rightarrow\infty$: for $d=1$
\begin{equation}\label{cndWd1}
\int_{-\infty}^\infty\int_{-\infty}^\infty
W(x)(x-s)^2W(s)\,dx\,ds<\infty
\end{equation}
and for $d=2$
\begin{equation}\label{cndWd2}
\int_{\R^2}\int_{\R^2}
W(\e)(\ln(1+|\e-\bs|)^2W(\bs)\,d\e\,d\bs<\infty.
\end{equation}

Our result about virtual eigenvalues in the non-degenerate case is
following:

\begin{theorem}\label{thnondegedg}
Assume that the unperturbed potential $V(\e)$ satisfies the
strengthened version of condition (\ref{condperpotent}):
\begin{equation}\label{Hold}
V\in\left\{\begin{array}{ll}
L_2(\Omega), &\rm{if}\quad d=1\\
\bigcup_{q>d}L_q(\Omega), &\rm{if}\quad d\ge 2,
\end{array}\right.
\end{equation}
the perturbation $W(\e)$ is measurable and bounded in $\R^d$,
$\lim_{|\e|\rightarrow\infty} W(\e)=0$, $W(\e)\ge 0$ a.e. on $\R^d$
and  $W(\e)>0$ on a set of positive measure. Furthermore, assume
that if the edge $\lambda_+\;(\lambda_-)$ is non-degenerate, for
$d=1$ the condition (\ref{cndWd1}) is satisfied and for $d=2$ the
condition (\ref{cndWd2}) is satisfied. Then\vskip2mm

 (i) for $\gamma<0$ the operator
$H_\gamma$ can have in $(\lambda_-,\lambda_+)$ a virtual eigenvalue
only at the edge $\lambda_+$;\vskip2mm

(ii) if the edge $\lambda_+$ is non-degenerate,  $d\le 2$ and
$\gamma<0$, the operator $H_\gamma$ has in $(\lambda_-,\lambda_+)$
virtual eigenvalues at $\lambda_+$ having the properties:\vskip1mm

\indent\indent (a) if $d=1$, there is a unique virtual eigenvalue
$\rho^+_1(\gamma)$ in $(\lambda_-,\lambda_+)$ , having the following
asymptotic representation for $\gamma\uparrow 0$:
\begin{eqnarray}\label{asympteigvd1}
\sqrt{\lambda_+-\rho^+_1(\gamma)}=|\gamma|\big(\sqrt{m^+_1}\frac{\|v^+_1\|^2}{\sqrt{2}}+O(\gamma)\big),
\end{eqnarray}
where $m^+_1$ is defined by (\ref{dfmupl}-a) and $v_1^+(x)$ is
defined by (\ref{dfvlpld1}-a); \vskip1mm

\indent\indent (b) if $d=2$, there are $\;n_+$ virtual eigenvalues
\begin{equation}\label{virteignondeg}
\rho_1^+(\gamma)\le\rho_2^+(\gamma)\le\dots\le \rho_{n_+}^+(\gamma)
\end{equation}
in $(\lambda_-,\lambda_+)$ (counting their multiplicities), and the
following asymptotic representation is valid for them for
$\gamma\uparrow 0$:
\begin{eqnarray}\label{asympteigvd2}
\left(\ln\left(\frac{1}{\lambda_+-\rho_k^+(\gamma)}\right)\right)^{-1}=
|\gamma|\Big(\frac{\nu_k^+}{2\pi}+O(\gamma)\Big)\quad(k=1,2,\dots,n_+),
\end{eqnarray}
and furthermore, the asymptotic formula of Lieb-Thirring type is
valid for $\gamma\uparrow 0$:
\begin{eqnarray}\label{asymptLbThrd2}
\sum_{k=1}^{n_+}\left(\ln\left(\frac{1}{\lambda_+-\rho_k^+(\gamma)}\right)\right)^{-1}=
\frac{|\gamma|}{2\pi}\sum_{k=1}^{n_+}\|v_k^+\|^2\sqrt{m_k^+}+O(\gamma^2);
\end{eqnarray}
\vskip2mm

(iii) if the edge $\lambda_+$ is non-degenerate and $d\le 2$, the
eigenfunctions corresponding to the virtual eigenvalues, considered
above, have the properties:\vskip1mm

\indent\indent (a) if $d=1$, there exists $\bar\gamma>0$ such that
for any $\gamma\in[-\bar\gamma,0)$ it is possible to choose an
eigenfunction $\psi^+_{\gamma,1}(\e)$ of the operator $H_\gamma$
corresponding to its eigenvalue $\rho^+_1(\gamma)$ such that
$\Vert\sqrt{W}\psi^+_{\gamma,1}- g_1^+\Vert=O(\gamma)$ for
$\gamma\uparrow 0$, where $g^+_1=\frac{v^+_1}{\Vert
v^+_1\Vert}$;\vskip1mm

\indent\indent (b) if $d=2$ and $m(j)\;(j\in\{1,2,\dots,n_+\})$ is
the multiplicity of an eigenvalue $\nu_j^+$ of the operator $G_W^+$
and
$\rho_{l(j)}^+(\gamma)\le\rho_{l(j)+1}^+(\gamma)\le\dots\le\rho_{l(j)+m(j)-1}^+(\gamma)$
$(l(j)\in\{1,2,\dots,n_+\})$ is the group of virtual eigenvalues of
$H_\gamma$, for which $\lim_{\gamma\uparrow
0}\ln\left(\frac{1}{\lambda_+-\rho_k^+(\gamma)}\right)|\gamma|=\frac{2\pi}{\nu_j^+}$,
then there exists $\bar\gamma>0$ such that for any
$\gamma\in[-\bar\gamma,0)$ there are numbers
$\gamma_0(\gamma)=\gamma$, $\{\gamma_k(\gamma)\}_{k=1}^{m(j)-1}$
having the properties: $\;\gamma-\gamma_k(\gamma)=O(\gamma^2)$ as
$\gamma\uparrow 0$, for each $k\in\{0,1,\dots,m(j)-1\}$ the number
$\rho_{l(j)}^+(\gamma)$ is an eigenvalue of the operator
$H_{\gamma_k(\gamma)}$, and it is possible to choose a basis
\begin{equation*}
\psi_{\gamma,\,l(j)}^+(\e),\psi_{\gamma,\,l(j)+1}^+(\e),\dots,\psi_{\gamma,\,l(j)+m(j)-1}^+(\e)
\end{equation*}
in the linear span of eigenspaces of all the operators
$H_{\gamma_k(\gamma)}\;(k\in\{0,1,\dots,m(j)-1\})$, corresponding to
their eigenvalue $\rho_{l(j)}^+(\gamma)$, for which the property
\begin{equation}\label{asympteigvec1}
\Vert\sqrt{W}\psi_{\gamma,k}^+-g_k^+\Vert=O(\gamma) \quad
\rm{as}\quad \gamma\uparrow 0
\end{equation}
is valid for any $k\in\{l(j),l(j)+1,\dots,l(j)+m(j)-1\}$; \vskip2mm

(iv) if the edge $\lambda_+$ is non-degenerate, $d\ge 3$, $W\in
L_1(\R^d)$ and $\gamma<0$ the operator $H_\gamma$ has in
$(\lambda_-,\lambda_+)$ no virtual eigenvalue at
$\lambda_+$;\vskip2mm

(v) for $\gamma>0$ the claims (i)-(iv) are valid with $\lambda_-$,
$n_-$, $\rho_k^-(\gamma)$, $\nu_k^-$, $m_k^-$, $v_k^-(\e)$,
 $\psi_{\gamma,k}^-(\e)$, $g_k^-(\e)$, $\gamma>0$ and $\gamma\downarrow 0$ instead of, respectively,
$\lambda_+$, $n_+$, $\rho_k^+(\gamma)$, $\nu_k^+$, $m_k^+$,
$v_k^+(\e)$, $\psi_{\gamma,k}^+(\e)$, $g_k^+(\e)$, $\gamma<0$ and
$\gamma\uparrow 0$;\vskip2mm

(vi) if both edges $\lambda_+$ and $\lambda_-$ are non-degenerate,
$d\ge 3$ and $W\in L_1(\R^d)$, there is a threshold for the birth of
the impurity spectrum in the gap $(\lambda_-,\,\lambda_+)$, that is
$\sigma(H_\gamma)\cap(\lambda_-,\lambda_+)=\emptyset$ for a small
enough $|\gamma|$.
\end{theorem}

\subsection{The birth of virtual eigenvalues from a degenerate edge}
\label{subsec:mainresbirthdeg}

Assume that at least one of the edges $\lambda_+$ or $\lambda_-$ of
a gap $(\lambda_-,\,\lambda_+)$ of the spectrum of the unperturbed
operator $H_0$ satisfies the condition (B) of Section
\ref{subsec:spectchar} such that the dimension of at least one of
the connected components of the extremal set
$F^+=(\lambda^+)^{-1}(\lambda_+)\;(F_-=(\lambda^-)^{-1}(\lambda_-))$
of the dispersion function $\lambda^+(\J)\;(\lambda^-(\J))$,
branching from the edge $\lambda_+\;(\lambda_-)$, is non-zero. In
this case we have a non-localized degeneration of the dispersion
function $\lambda_{j+1}(\J)\;(\lambda_j(\J))$ at the edge
$\lambda_+\;(\lambda_-)$ of the forbidden zone
$(\lambda_-,\lambda_+)$. For the simplicity we shall assume in this
section that the following condition is satisfied for at least one
of the edges of $(\lambda_-,\lambda_+)$:

(C) The edge $\lambda_+\;(\lambda_-)$ of a gap
$(\lambda_-,\,\lambda_+)$ of the spectrum of the unperturbed
operator $H_0$ satisfies condition (B) of Section
\ref{subsec:spectchar} with $n_+=1\,(n_-=1)$ (that is $F^+\,(F^-)$
is a connected smooth submanifold of $\T^d$)  and $d_+=\dim(F^+)\ge
1\;(d_-=\dim(F^-)\ge 1)$.

Consider the integral operator acting in $L_2(\R^d)$
\begin{equation}\label{dfintopGW}
G_W^+f=\int_{\R^d}\Gc_W^+(\e,\bs)f(\bs)\,d\bs\quad\Big(G_W^-f=\int_{\R^d}\Gc_W^-(\e,\bs)f(\bs)\,d\bs\Big)
\end{equation}
with
\begin{eqnarray}\label{dfkerGW}
&&\Gc_W^+(\e,\bs)=\int_{F^+}\Qc_W^+(\e,\bs,\J)\sqrt{m^+(\J)}\,dF(\J)\\
&&
\Big(\Gc_W^-(\e,\bs)=\int_{F^-}\Qc_W^-(\e,\bs,\J)\sqrt{m^-(\J)}\,dF(\J)\Big),
\nonumber
\end{eqnarray}
where $\Qc_W^+(\e,\bs,\J)\;(\Qc_W^-(\e,\bs,\J))$ is the weighted
eigenkernel, corresponding to $\lambda^+(\J)\;(\lambda^-(\J))$ and
defined by (\ref{dfWxsp}), $m^+(\J)\;(m^-(\J))$ is defined by
(\ref{dfmplp}) and $dF(\J)$ is the volume form on the submanifold
$F^+\,(F^-)$. In what follows we shall prove (Lemma
\ref{lmspecGWpl}) that if $W\in L_1(\R^d)$, $W(\e)\ge 0$ a.e. on
$\R^d$ and $W(\e)> 0$ on a set of the positive measure, then the
integral operator $G_W^+\,(G_W^-)$ is self-adjoint, nonnegative,
belongs to the trace class and has an infinite number of positive
eigenvalues
\begin{equation}\label{eigvalGW}
\nu_1^+\ge\nu_2^+\ge\dots\ge\nu_n^+\ge\dots\;(\nu_1^-\ge\nu_2^-\ge\dots\ge\nu_n^-\ge\dots)
\end{equation}
(each eigenvalue is repeated according to its multiplicity). Let
\begin{equation*}
g_1^+(\e),g_2^+(\e),\dots,g_n^+(\e),\dots\;\big(g_1^-(\e),g_2^-(\e),\dots,g_n^-(\e),\dots\big)
\end{equation*}
be an orthonormal sequence of eigenfunctions of the operator
$G_W^+\;(G_W^-)$ corresponding to its eigenvalues (\ref{eigvalGW}).

If $d-d_+\le 2\;(d-d_-\le 2)$, we shall impose on the non-negative
perturbation $W(\e)$ the following conditions: for
$d-d_+=1\;(d-d_-=1)$
\begin{equation}\label{cndWd1deg}
\int_{\R^d}\int_{\R^d} W(\e)(\e-\bs)^2W(\bs)\,d\e\,d\bs<\infty
\end{equation}
and for $d-d_+=2\;(d-d_-=2)$
\begin{equation}\label{cndWcod2}
\int_{\R^d}\int_{\R^d}
W(\e)(\ln(1+|\e-\bs|)^2W(\bs)\,d\e\,d\bs<\infty
\end{equation}

Our result about virtual eigenvalues in the degenerate case is
following:

\begin{theorem}\label{thdegedg}
Assume that the unperturbed potential $V(\e)$ satisfies the
condition (\ref{Hold}), the perturbation $W(\e)$ is measurable and
bounded in $\R^d$, $\lim_{|\e|\rightarrow\infty}\\ W(\e)=0$,
$W(\e)\ge 0$ a.e. on $\R^d$ and  $W(\e)>0$ on a set of positive
measure. Furthermore, assume that if the edge
$\lambda_+\;(\lambda_-)$ satisfies the condition (C), for
$d-d_+=1\;(d-d_-=1)$ the condition (\ref{cndWd1deg}) is satisfied
and for $d-d_+=2\;(d-d_-=2)$ the condition (\ref{cndWcod2}) is
satisfied. Then\vskip2mm

(i) for $\gamma<0$ the operator $H_\gamma=H_0+\gamma W\cdot$ can
have in $(\lambda_-,\lambda_+)$ a virtual eigenvalue only at the
edge $\lambda_+$;\vskip2mm

(ii) if the edge $\lambda_+$ satisfies the condition (C), $d-d_+\le
2$ and $\gamma<0$, the operator $H_\gamma$ has in
$(\lambda_-,\lambda_+)$ an infinite number of virtual eigenvalues
$\rho_1^+(\gamma)\le\rho_2^+(\gamma)\le\dots\rho_n^+(\gamma)\le\dots$
at $\lambda_+$, and moreover the following asymptotic representation
is valid for $\gamma\uparrow 0$:
\begin{equation}\label{asympteigvd2deg}
\Psi\big(\lambda_+-\rho_n^+(\gamma)\big)=|\gamma|\big(\nu_n^++O(\gamma)\big),
\end{equation}
where
\begin{displaymath}
\Psi(s)=\left\{\begin{array}{ll}
\frac{(2\pi)^d}{\sqrt{2}\pi}\sqrt{s},&\rm{if}\quad d-d_+=1,\\
(2\pi)^{d-1}\Big(\ln\big(\frac{1}{s}\big)\Big)^{-1},&\rm{if}\quad
d-d_+=2
\end{array}\right.
\end{displaymath}
$(s\in(0,\;\lambda_+-\lambda_-))$, and furthermore, the asymptotic
formula of Lieb-Thirring type is valid:
\begin{eqnarray}\label{asymptLbThrd2deg}
&&\lim_{n\rightarrow\infty}\lim_{\gamma\uparrow
0}\frac{1}{|\gamma|}\sum_{k=1}^n
\Psi\big(\lambda_+-\rho_k^+(\gamma)\big)=\nonumber\\
&&\int_{F^+}\int_{\R^d}\Qc_W^+(\bs,\bs,\J)\;d\bs\sqrt{m^+(\J)}\,dF(\J);
\end{eqnarray}
\vskip2mm

(iii) if the edge $\lambda_+$ satisfies the condition (C) and
$d-d_+\le 2$, the eigenfunctions corresponding to the virtual
eigenvalues, considered above, have the property: if $\;m(j)$ is the
multiplicity of an eigenvalue $\nu_j^+$ of the operator $G_W^+$ and
$\rho_{l(j)}^+(\gamma)\le\rho_{l(j)+1}^+(\gamma)\le\dots\le\rho_{l(j)+m(j)-1}^+(\gamma)$
is the group of virtual eigenvalues of $H_\gamma$, for which
$\lim_{\gamma\uparrow
0}\frac{1}{|\gamma|}\Psi\big(\lambda_+-\rho_k^+(\gamma)\big)=\nu_j^+$
$(k\in\{l(j),l(j)+1,\dots,l(j)+m(j)-1\})$, then there exists
$\bar\gamma>0$ such that for any $\gamma\in[-\bar\gamma,0)$  there
are numbers $\gamma_0(\gamma)=\gamma$,
$\{\gamma_k(\gamma)\}_{k=1}^{m(j)-1}$ such that they and the
eigenvectors of the operators
$H_{\gamma_k(\gamma)}\;(k\in\{l(j),l(j)+1,\dots,l(j)+m(j)-1\})$,
corresponding to their eigenvalue $\rho_{l(j)}^+(\gamma)$, have the
same asymptotic properties as in claim (iii-b) of Theorem
\ref{thnondegedg}; \vskip2mm

(iv) if the edge $\lambda_+$ satisfies the condition (C), $d-d_+\ge
3$, $W\in L_1(\R^d)$ and $\gamma<0$, the operator $H_\gamma$ has in
$(\lambda_-,\lambda_+)$ no virtual eigenvalue at
$\lambda_+$;\vskip2mm

(v) for $\gamma>0$ the claims (i)-(iv) are valid with $d_-$,
$\lambda_-$, $\rho_k^-(\gamma)$, $\nu_k^-$, $\Qc_W^-(\e,\bs,\J)$,
$m^-(\J)$, $\psi_{\gamma,k}^-(\e)$, $g_k^-(\e)$, $\gamma>0$ and
$\gamma\downarrow 0$ instead of, respectively, $d_+$, $\lambda_+$,
$\rho_k^+(\gamma)$, $\nu_k^+$, $\Qc_W^+(\e,\bs,\J)$, $m^+(\J)$,
$\psi_{\gamma,k}^+(\e)$, $g_k^+(\e)$, $\gamma<0$ and $\gamma\uparrow
0$;\vskip2mm

(vi) if both edges $\lambda_+$ and $\lambda_-$ satisfy the condition
(C), $d-d_+\ge 3$, $d-d_-\ge 3$ and  $W\in L_1(\R^d)$, there is a
threshold for the birth of the impurity spectrum in the gap
$(\lambda_-,\,\lambda_+)$, that is
$\sigma(H_\gamma)\cap(\lambda_-,\,\lambda_+)=\emptyset$ for a small
enough $|\gamma|$.
\end{theorem}

\subsection{The estimate of the multiplicity of virtual eigenvalues for an indefinite perturbation
in the case of a non-degenerate edge}
\label{subsec:mainresestmult}

In the case of an indefinite perturbation $W(\e)$ we only can, under
some conditions, estimate from above the multiplicity of virtual
eigenvalues of $H_\gamma$ in the gap $(\lambda_-,\,\lambda_+)$ and,
in particular, establish the existence of a threshold for the birth
of the impurity spectrum for $d\ge 3$. Denote by $W_+(\e)$ and
$W_-(\e)$ the positive and negative parts of $W(\e)$, that is
$W_+(\e)=\frac{1}{2}(W(\e)+|W(\e)|)$ and
$W_-(\e)=\frac{1}{2}(W(\e)-|W(\e)|)$. We shall consider the
following conditions: in the case where $d=1$
\begin{equation}\label{cndWpld1}
\int_{-\infty}^\infty\int_{-\infty}^\infty
W_+(x)(x-s)^2W_+(s)\,dx\,ds<\infty,
\end{equation}
\begin{equation}\label{cndWmnd1}
\int_{-\infty}^\infty\int_{-\infty}^\infty
W_-(x)(x-s)^2W_-(s)\,dx\,ds<\infty,
\end{equation}
and in the case where $d=2$
\begin{equation}\label{cndWpld2}
\int_{\R^2}\int_{\R^2}
W_+(\e)(\ln(1+|\e-\bs|)^2W_+(\bs)\,d\e\,d\bs<\infty,
\end{equation}
\begin{equation}\label{cndWmnd2}
\int_{\R^2}\int_{\R^2}
W_-(\e)(\ln(1+|\e-\bs|)^2W_+(\bs)\,d\e\,d\bs<\infty.
\end{equation}

Recall that if the edge $\lambda_+\;(\lambda_-)$ is non-degenerate,
then the extremal set $F^+\;(F^-)$ of the dispersion function
$\lambda^+(\J)\;(\lambda^-(\J))$, branching from this edge, is
finite and we denote by $n_+\;(n_-)$ the number of its points.

The following result is valid:

\begin{theorem}\label{thestmult}
Let $(\lambda_-,\lambda_+)$ be a gap of the spectrum of the
unperturbed operator $H_0$. Assume that the unperturbed potential
$V(\e)$ satisfies the condition (\ref{Hold}), the perturbation
$W(\e)$ is measurable and bounded in $\R^d$ and
$\lim_{|\e|\rightarrow\infty} W(\e)=0$. Let us take
$\delta\in(0,\,\lambda_+-\lambda_-)$. Then\vskip2mm

(i) for $d=1$: if the edge $\lambda_+$ is non-degenerate and
condition (\ref{cndWpld1}) is satisfied, then there exists
$\bar\gamma>0$ such that for any $\gamma\in (-\bar\gamma,0)$ all the
virtual eigenvalues of the operator $H_\gamma$, being born from the
edge $\lambda_+$ and lying in $(\lambda_+-\delta,\;\lambda_+)$, are
simple; if the edge $\lambda_-$ is non-degenerate and condition
(\ref{cndWmnd1}) is satisfied, then there exists $\bar\gamma>0$ such
that for any $\gamma\in (-\bar\gamma,0)$ all the virtual eigenvalues
of the operator $H_\gamma$, being born from the edge $\lambda_-$ and
lying in $(\lambda_-,\;\lambda_-+\delta)$ are simple;\vskip2mm

(ii) for $d=2$: if the edge $\lambda_+$ is non-degenerate the
condition (\ref{cndWpld2}) is satisfied, then there exists
$\bar\gamma>0$ such that for any $\gamma\in (-\bar\gamma,0)$ the
multiplicity of each virtual eigenvalue of the operator $H_\gamma$,
being born from the edge $\lambda_+$ and lying in
$(\lambda_+-\delta,\;\lambda_+)$ is not bigger than $n_+$; if the
edge $\lambda_-$ is non-degenerate and condition (\ref{cndWmnd2}) is
satisfied, then there exists $\bar\gamma>0$ such that for any
$\gamma\in (-\bar\gamma,0)$ the multiplicity of each virtual
eigenvalue of the operator $H_\gamma$, being born from the edge
$\lambda_-$ and lying in $(\lambda_-,\;\lambda_-+\delta)$, is not
bigger than $n_-$;\vskip2mm

(iii) for $d\ge 3$: if the edge $\lambda_+$ is non-degenerate and
$W_+\in L_1(\R^d)$, then there exists $\bar\gamma>0$ such that for
any $\gamma\in (-\bar\gamma,0)$ no eigenvalue of the operator
$H_\gamma$ lies in $(\lambda_+-\delta,\;\lambda_+)$; if the edge
$\lambda_-$ is non-degenerate and $W_-\in L_1(\R^d)$, then there
exists $\bar\gamma>0$ such that for any $\gamma\in (-\bar\gamma,0)$
no eigenvalue of the operator $H_\gamma$ lies in
$(\lambda_-,\lambda_-+\delta)$; \vskip2mm

(iv) for $\gamma>0$ all the above claims are valid with $\lambda_-$,
$\lambda_+$, $W_-(\e)$, $W_+(\e)$, $\gamma\in (0,\bar\gamma)$
$(\lambda_-,\;\lambda_-+\delta)$, $(\lambda_+-\delta,\lambda_+)$,
$n_-$ and $n_+$ instead of, respectively, $\lambda_+$, $\lambda_-$,
$W_+(\e)$, $W_-(\e)$, $\gamma\in (-\bar\gamma,0)$,
$(\lambda_+-\delta,\lambda_+)$, $(\lambda_-,\;\lambda_-+\delta)$,
 $n_+$ and $n_-$;\vskip2mm

(v) for $d\ge 3$: if  both edges $\lambda_-$ and $\lambda_+$ are
non-degenerate and $W\in L_1(\R^d)$, then there is a threshold for
the birth of the impurity spectrum in the gap
$(\lambda_-,\,\lambda_+)$, that is
$\sigma(H_\gamma)\cap(\lambda_-,\,\lambda_+)=\emptyset$ for a small
enough $|\gamma|$.
\end{theorem}

\subsection{A threshold for the birth of virtual eigenvalues for an indefinite perturbation
in the case of a degenerate edge} \label{subsec:mainresthresdeg}

\begin{theorem}\label{ththresholddeg}
Let $(\lambda_-,\,\lambda_+)$ be a gap of the spectrum of the
unperturbed operator $H_0$. Assume that the unperturbed potential
$V(\e)$ satisfies the condition (\ref{Hold}), the perturbation
$W(\e)$ is measurable and bounded in $\R^d$ and
$\lim_{|\e|\rightarrow\infty} W(\e)=0$. Let us take
$\delta\in(0,\,\lambda_+-\lambda_-)$. Then\vskip2mm

(i) if the edge $\lambda_+$ satisfies condition (C) of Section
\ref{subsec:mainresbirthdeg}, $d-d_+\ge 3$ and  $W_+\in L_1(\R^d)$,
there exists $\bar\gamma>0$ such that for any $\gamma\in
(-\bar\gamma,0)$ no eigenvalue of the operator $H_\gamma$ lies in
$(\lambda_+-\delta,\;\lambda_+)$; if the edge $\lambda_-$ satisfies
condition (C), $d-d_-\ge 3$ and $W_-\in L_1(\R^d)$, there exists
$\bar\gamma>0$ such that  for any $\gamma\in (-\bar\gamma,0)$ no
eigenvalue of the operator $H_\gamma$ lies in
$(\lambda_-,\;\lambda_-+\delta)$; \vskip2mm

(ii) for $\gamma>0$ claim (i) is valid with $d_-$, $d_+$,
$\lambda_-$, $\lambda_+$, $W_-(\e)$, $W_+(\e)$, $\gamma\in
(0,\bar\gamma)$ $(\lambda_-,\;\lambda_-+\delta)$ and
$(\lambda_+-\delta,\lambda_+)$ instead of, respectively, $d_+$,
$d_-$, $\lambda_+$, $\lambda_-$, $W_+(\e)$, $W_-(\e)$, $\gamma\in
(-\bar\gamma,0)$, $(\lambda_+-\delta,\lambda_+)$ and
$(\lambda_-,\;\lambda_-+\delta)$;\vskip2mm

(iii) if condition (C) is satisfied for both edges $\lambda_-$ and
$\lambda_+$, $d-d_+\ge 3$, $d-d_-\ge 3$  and $W\in L_1(\R^d)$, there
is a threshold for the birth of the impurity spectrum in the gap
$(\lambda_-,\lambda_+)$, that is
$\sigma(H_\gamma)\cap(\lambda_-,\,\lambda_+)=\emptyset$ for a small
enough $|\gamma|$.
\end{theorem}

\section{Proof of main results}
\label{sec:prmainres}
\setcounter{equation}{0}

\subsection{General results on the birth of virtual eigenvalues}
\label{subsec:genres}

Let $H_0$,$\;H_\gamma=H_0+\gamma W$ and $(\lambda_-,\lambda_+)$ are
the same as in Section \ref{subsec:spectchar}. For
$\lambda\in(\lambda_-,\lambda_+)$ consider the operator
\begin{equation}\label{BrmSchwop}
X_W(\lambda)=W^{\frac{1}{2}}R_\lambda(H_0)|W|^{\frac{1}{2}},
\end{equation}
which is called in the literature the {\it Birman-Schwinger
operator} (\cite{S1}, \cite{Sc}). Here
\begin{eqnarray}\label{dfsqrtW}
W^{1/2}x=\left\{\begin{array}{ll}
W(|W||_{(\ker(W))^{\perp}})^{-1/2}x,& \rm{if}\; x\in
(\ker(W))^{\perp},\\
0,& \rm{if}\; x\in \ker(W),
\end{array}\right.
\end{eqnarray}
hence $W=|W|^{\frac{1}{2}}W^{\frac{1}{2}}$. In \cite{S1} (Lemma 7.1)
a connection between the spectrum of the operator pencil $I+\gamma
X_W(\lambda)\,(\lambda\in(-\infty, 0))$ and the spectrum of
$H_\gamma$ in the gap $(-\infty,0)$ of $\sigma(H_0)$ was established
for the case where $\B=L_2(\R^d)$ and $H_0=-\Delta$. For the case of
a definite perturbation $(W\ge 0)$ of a periodic potential the
analogous result was obtained in \cite{Bi2} (Proposition 1.5). The
following claim is a generalization of this claim to the case of an
indefinite perturbation. We also have indicated an operator which
realizes a linear isomorphism between $\ker(H_\gamma-\lambda I)$ and
$\ker(I+\gamma X_W(\lambda))$.
\begin{proposition}\label{BrScwspct} Assume that for some
$\lambda_0\in\Rs(H_0)$ the operator
$R_{\lambda_0}(H_0)|W|^{\frac{1}{2}}$ is compact and
$(\lambda_-,\,\lambda_+)$ is a gap of the spectrum of the operator
$H_0$. Then the set
$\sigma_\gamma(\lambda_-,\lambda_+)=\sigma(H_\gamma)\cap(\lambda_-,\lambda_+)$
consists of at most countable number of eigenvalues of finite
multiplicity of the operator $H_\gamma$ and they can cluster only to
the edges of the gap $(\lambda_-,\,\lambda_+)$. Furthermore,
$\sigma_\gamma(\lambda_-,\lambda_+)$ coincides with the spectrum of
the pencil of operators $\Phi(\gamma)=I+\gamma
X_W(\lambda)\,(\lambda\in(\lambda_-,\lambda_+))$. Moreover, the
operator function $X_W(\lambda)$ is holomorphic in $\Rs(H_0)$ in the
operator norm, each of the operators $X_W(\lambda)$ is compact and
any point $\lambda\in\sigma_\gamma(\lambda_-,\lambda_+)$ is an
eigenvalue of the pencil $\Phi(\gamma)$ such that the operator
$W^{1/2}\vert_{E_\lambda}$ realizes a linear isomorphism between the
subspaces $E_\lambda=\ker(H_\gamma-\lambda I)$ and
$L_\lambda=\ker(I+\gamma X_W(\lambda))$, hence
$\dim(E_\lambda)=\dim(L_\lambda)$.
\end{proposition}
\begin{proof}
Assume that $\lambda\in\Rs(H_0)$. The subspace $E_\lambda$ is the
set of solutions of the equation $H_0x-\lambda x+\gamma Wx=0$, which
is equivalent to the equation $x+\gamma R_\lambda(H_0)Wx=0$. Denote
$y=P_Wx$, $z=(I-P_W)x$, where $P_W$ is the orthogonal projection on
the subspace $\B_W=\B\ominus\ker(W)$. Since $W(I-P_W)=0$, the last
equation is equivalent to the system
\begin{equation}\label{systeq}
y+\gamma P_WR_\lambda(H_0)Wy=0, \quad z+\gamma\Xi y=0,
\end{equation}
where
\begin{equation}\label{dfXi}
\Xi=(I-P_W)R_\lambda(H_0)W.
\end{equation}
Hence the operator $\Psi=(I-\gamma\Xi)\vert_{G_\lambda}$ realizes a
linear isomorphism between the subspace $G_\lambda$ of solutions of
equation (\ref{systeq}-a) and the subspace $E_\lambda$. As it is
clear, $\Psi^{-1}=P_W\vert_{E_\lambda}$. Consider the operator
$S=W^{1/2}\vert_{\B_W}$. Since $\ker(W^{1/2})=\ker(W)$, $\B_W$ is an
invariant subspace for the operator $W^{1/2}$. Hence $S$ maps
injectively $\B_W$ into itself. Let us show that
$S(G_\lambda)=L_\lambda$. Assume that $y\in G_\lambda$, that is $y$
is a solution of the equation (\ref{systeq}-a). Multiplying the both
sides of (\ref{systeq}-a) by $W^{1/2}$, we get that $v=Sy=W^{1/2}y$
is a solution of the equation $v+\gamma
W^{1/2}P_WR_\lambda(H_0)|W|^{1/2}v=0$. Since $W^{1/2}(I-P_W)=0$, the
last equation is equivalent to
\begin{equation}\label{kerXW}
v+\gamma W^{1/2}R_\lambda(H_0)|W|^{1/2}v=0.
\end{equation}
This means that $v\in\ker(I+\gamma X_W(\lambda))=L_\lambda$, hence
$S(G_\lambda)\subseteq L_\lambda$. Let us prove the inverse
inclusion. Assume that $v\in L_\lambda$, that is $v$ is a solution
of equation (\ref{kerXW}). Then we see that $v\in
\mathrm{Im}(W^{1/2})\subseteq H_W$. Hence there exists $y\in H_W$
such that $v=Sy=W^{1/2}y$ and $W^{1/2}y+\gamma
W^{1/2}R_\lambda(H_0)Wy=0$. The last equality and $W^{1/2}(I-P_W)=0$
imply that $Sy+\gamma SP_WR_\lambda(H_0)Wy=0$. Since the operator
$S$ is injective, we get that $y$ satisfies equation
(\ref{systeq}-a), that is $y\in G_\lambda$. We have shown that
$L_\lambda\subseteq S(G_\lambda)$. Since we have proved above the
inverse inclusion, we get: $L_\lambda= S(G_\lambda)$. Observe that
$W^{1/2}(I-P_W)=0$. Thus, the operator
$S\Psi^{-1}=W^{1/2}P_W\vert_{E_\lambda}=W^{1/2}\vert_{E_\lambda}$
realizes a linear isomorphism between $E_\lambda$ and $L_\lambda$.
Hence, in particular, $\dim(E_\lambda)=\dim(L_\lambda)$.

The compactness of the operator
$R_{\lambda_0}(H_0)|W|^{\frac{1}{2}}$, the boundedness of the
operator $W^{\frac{1}{2}}$ and the Hilbert identity
$R_{\lambda}(H_0)-R_{\lambda_0}(H_0)=(\lambda-\lambda_0)R_{\lambda}(H_0)R_{\lambda_0}(H_0)$
imply that for any $\lambda\in\Rs(H_0)$ the operator $X_W(\lambda)$
is compact and it is a holomorphic on $\C\setminus\sigma(H_0)$
operator function w.r.t. the operator norm. Then, as it is known,
$\alpha(\lambda)=\dim(\ker(I+\gamma X_W(\lambda)))<\infty$ for any
$\lambda\in\Rs(H_0)$ and $\alpha(\lambda)$ takes a constant value
$m$ at any point $\lambda\in\Rs(H_0)$, except a set of isolated
points in $\C\setminus\sigma(H_0)$ at which $\alpha(\lambda)>m$
(\cite{Gh-Kr}, Chapt. 1, Sect. 5, Theorem 5.4). On the other hand,
since $H_0$ is self-adjoint, for $\Im(\lambda)\neq 0$ the estimate
$\Vert X_W(\lambda)\Vert\le\Vert W\Vert\Vert
R_{\lambda_0}(H_0)\Vert\le \frac{\Vert W\Vert}{|\Im(\lambda)|)}$ is
valid, hence for a large enough $|\Im(\lambda)|$ the operator
$I+\gamma X_W(\lambda)$ is continuously invertible. This means that
$m=0$. Therefore, in particular, $\alpha(\lambda)=0$ for any
$\lambda\in(\lambda_-,\lambda_+)$, except a set
$\Lambda\subset(\lambda_-,\lambda_+)$ of isolated in
$(\lambda_-,\lambda_+)$ points. Taking into account that
$\dim(E_\lambda)=\dim(L_\lambda)$ for
$\lambda\in(\lambda_-,\lambda_+)$, we obtain that each point of
$\Lambda$ is an isolated eigenvalue of the operator $H_\gamma$ of a
finite multiplicity. On the other hand, if
$\lambda\in(\lambda_-,\lambda_+)\setminus\Lambda$, the operator
$I+\gamma X_W(\lambda)$ is continuously invertible, because the
operator $X_W(\lambda)$ is compact. Hence $\mathrm{Im}(I+\gamma
X_W(\lambda))=\B$. Therefore, in particular, for any $f\in\B$ the
equation $v+\gamma X_W(\lambda)v=W^{1/2}P_Wf$ has a solution $v$
which, as it is clear, belongs to $\mathrm{Im}(W^{1/2})$. Hence the
vector $y=\big((W^{1/2}\vert_{\B_W}\big)^{-1}v\in\B_W$ satisfies the
equation $W^{1/2}y+\gamma X_W(\lambda)W^{1/2}y=W^{1/2}P_Wf$. Since
the operator $(W^{1/2}\vert_{\B_W}$ is injective, the last equality
implies that $y+\gamma P_WR_\lambda(H_0)Wy=P_Wf$. Hence the vector
$x=(I-\gamma\Xi)y+(I-P_W)f$ satisfies the equation
$x+R_\lambda(H_0)W=f$, which is equivalent to $H_\gamma x-\lambda
x=f$. Recall that the operator $\Xi$ is defined by (\ref{dfXi}).
Thus, we have proved that for any
$\lambda\in(\lambda_-,\lambda_+)\setminus\Lambda$
$\;\mathrm{Im}(H_\gamma-\lambda I)=\B$. Since $H_\gamma$ is
self-adjoint, this means that
$(\lambda_-,\lambda_+)\setminus\Lambda\subseteq\Rs(H_\gamma)$. Thus,
the set $\Lambda$ which is the spectrum of the pencil
$X_W(\lambda)\;(\lambda\in(\lambda_-,\lambda_+))$, coincides with
the part of the spectrum of the operator $H_\gamma$ lying in
$(\lambda_-,\lambda_+)$ and it consists of isolated eigenvalues of
$H_\gamma$ having finite multiplicities.
\end{proof}

In what follows we shall assume that the condition of Proposition
\ref{BrScwspct} is fulfilled.

Consider the case where $W\ge 0$, hence
$X_W(\lambda)=W^{\frac{1}{2}}R_\lambda(H_0)W^{\frac{1}{2}}$. Let us recall some notions and results from \cite{Ar-Zl1} used in
this section. By  Proposition \ref{BrScwspct}, for each $\lambda\in(\lambda_-,\lambda_+)$ the operator
    $X_W(\lambda)$ is compact. Then since it is self-adjoint, its spectrum
    consists of at most a countable number of real eigenvalues
    which can cluster only to the point $0$. Furthermore, each its
    non-zero eigenvalue has a finite multiplicity.
\begin{definition}\label{gde1}(\cite{Ar-Zl1})
    {\rm For any fixed $\lambda\in(\lambda_-,\lambda_+)$ let us number all the
    positive eigenvalues ${\mu_k^+(\lambda)}_{k\in\N}$ of the operator $X_W(\lambda)$ in the non-increasing
    ordering $\mu_1^+(\lambda)\ge\mu_2^+(\lambda)\ge\dots\ge\mu_k^+(\lambda)\ge
    \dots$ and all the
    negative ones ${\mu_k^-(\lambda)}_{k\in\N}$ in the non-decreasing
    ordering $\mu_1^-(\lambda)\le\mu_2^-(\lambda)\le\dots\le\mu_k^-(\lambda)\le
    \dots$ (each eigenvalue is repeated according to its multiplicity).
    So, by such ordering we have chosen one-valued branches of eigenvalues
    of the operator function $X_W(\lambda)$.
     We call these branches the} {\it characteristic branches} {\rm (positive and negative)
    of the operator $H_0$ with respect to the operator $W$ on a gap $(\lambda_-,\,\lambda_+)$
    of $\sigma(H_0)$}.
    \end{definition}
\begin{remark}\label{rembranch}
    As it was shown in \cite{Ar-Zl1} (Proposition 3.7), if $W\ge 0$ the functions
    $\mu_k^+(\lambda)$ and $\mu_k^-(\lambda)$ are continuous
    and increasing. Since they can ``go to zero'' at some points of the gap
    $(\lambda_-,\,\lambda_+)$, each of them has its domain $\Dm(\mu_k^+)$ and $\Dm(\mu_k^-)$
    which have the form $\Dm(\mu_k^+)=(\eta_k^+,\lambda_+),\;\eta_k^+\in[\lambda_-,\lambda_+]$,
$\Dm(\mu_k^-)=(\lambda_-,\eta_k^-),\;\eta_k^-\in[\lambda_-,\lambda_+]$
and the following property is valid:
$\eta_k^+\in(\lambda_-,\lambda_+)\Rightarrow
    \lim_{\lambda\downarrow\eta_k^+}\mu_k^+(\lambda)=0$, $\eta_k^-\in(\lambda_-,\lambda_+)\Rightarrow
    \lim_{\lambda\uparrow\eta_k^-}\mu_k^-(\lambda)=0$.
Furthermore, the sequence $\{\eta_k^+\}$ is non-decreasing and the
    sequence $\{\eta_k^-\}$ is non-increasing.
\end{remark}

Denote $\tilde\mu_k^+=\lim_{\lambda\uparrow
\lambda_+}\mu_k^+(\lambda)$,
$\;\tilde\mu_k^-=\lim_{\lambda\downarrow \lambda_-}\mu_k^-(\lambda)$
(the values $+\infty$ and $-\infty$ are allowed for these limits).

    \begin{definition}\label{gde2}
    {\rm Consider $l(\lambda_-), l(\lambda_+)\in\Z_+\cup\{+\infty\}$
    defined by the conditions:
    \begin{eqnarray*}
    &&-\infty=\tilde\mu_1^-=\tilde\mu_2^-=\dots=\tilde\mu_{l(\lambda_-)}^-
    <\tilde\mu_{l(\lambda_-)+1}^-\le\tilde\mu_{l(\lambda_-)+2}^-\le\dots,\\
    && \ifr\quad\lambda_->-\infty,
    \end{eqnarray*}
\begin{eqnarray*}
   && +\infty=\tilde\mu_1^+=\tilde\mu_2^+=\dots=\tilde\mu_{l(\lambda_+)}^+
    >\tilde\mu_{l(\lambda_+)+1}^+\ge\tilde\mu_{l(\lambda_+)+2}^+\ge\dots,\\
    && \ifr\quad\lambda_+<+\infty.
\end{eqnarray*}
We call $l(\lambda_-)$ and $l(\lambda_+)$ the} {\it asymptotic
multiplicities
    of the edges $\lambda_-$ and $\lambda_+$} {\rm of the gap $(\lambda_-,\,\lambda_+)$ of
    $\sigma(H_0)$ with respect to the operator $W$
    and denote them
    $
    l(\lambda_-)=M(\lambda_-,H_0,W),\;\;l(\lambda_+)=M(\lambda_+,H_0,W).
    $
    The branches
    $\{\mu_k^-(\lambda)\}_{k=1}^{l(\lambda_-)}$,
    $\{\mu_k^+(\lambda)\}_{k=1}^{l(\lambda_+)}$
    are called the} {\it main characteristic branches} {\rm of the operator
    $H_0$ with respect to the operator $W$ near the edges
    $\lambda_-$ and $\lambda_+$
    respectively.}
    \end{definition}

    We shall use the following results from \cite{Ar-Zl1}:

\begin{proposition}\label{ordl}
The set $\sigma(H_\gamma)\cap(\lambda_-,\lambda_+)$ has the
representation:
\begin{eqnarray*}
\sigma(H_\gamma)\cap(\lambda_-,\lambda_+)=\left\{\begin{array}{ll}
\bigcup_{k=1}^\infty\{\rho_k^+(\gamma)\} &\fr\;\;\gamma<0,\\
\bigcup_{k=1}^\infty\{\rho_k^-(\gamma)\} &\fr\;\;\gamma>0,
\end{array}\right.
\end{eqnarray*}
where
$\rho_k^+(\gamma)=(\mu_k^+)^{-1}\left(-\frac{1}{\gamma}\right)$,
$\rho_k^-(\gamma)=(\mu_k^-)^{-1}\left(-\frac{1}{\gamma}\right)$ and
$(\mu_k^+)^{-1},\;\;(\mu_k^-)^{-1}$ are the inverses of the
functions $\mu_k^+(\lambda),\;\mu_k^-(\lambda)$, the positive and
negative characteristic branches of $H_0$ with respect to $V$  on
the gap $(\lambda_-,\,\lambda_+)$.
\end{proposition}

\begin{proposition}\label{gpbr}
Assume that the edge $\lambda_+$ of the gap
$(\lambda_-,\,\lambda_+)$ of $\sigma(H_0)$ is finite.\vskip2mm

(i) If $l(\lambda_+)=M(\lambda_+, H_0,W)<\infty$, then for
$\gamma<0$ the operator $H_\gamma$ has exactly $l(\lambda_+)$
branches $\lambda_-<\rho_1^+(\gamma)\le\rho_2^+(\gamma)\le
\dots\le\rho_{l(\lambda_+)}^+(\gamma)$ of virtual eigenvalues near
the edge $\lambda_+$ of $(\lambda_-,\,\lambda_+)$. Furthermore, for
any $1\le k\le l(\lambda_+)$ and $\gamma\in\Dm(\lambda_k^+)$,
\begin{equation}\label{frm}
\rho_k^+(\gamma)=(\mu_k^+)^{-1}\left(-\frac{1}{\gamma}\right),
\end{equation}
where $(\mu_k^+)^{-1}$ is the inverse of the function
$\mu_k^+(\lambda)$ (the main characteristic branch of $H_0$ with
respect to $W$ near the edge $\lambda_+$). If
$\;l(\lambda_+)<\infty$ and $\lambda_-=-\infty$, then
$\Dm(\rho_k^+)=(-\infty,0)\;\;\forall k\in\{1,2,\dots,l(b)\}$. Hence
the operator $H_\gamma$ has at least $l(\lambda_+)$ eigenvalues in
the gap $(-\infty,\,\lambda_+)$ for any $\gamma<0$.\vskip2mm

(ii) If $M(\lambda_+,H_0,W)=\infty$, then for $\gamma<0$ the number
of the branches of eigenvalues $\{\lambda_k^+(\gamma)$ of the
operator $H_\gamma$, which enter the gap $(\lambda_-,\,\lambda_+)$
across the edge $\lambda_+$, is infinite, each of them is a virtual
eigenvalue and the property is valid for them $(-\theta_k,
0)\subseteq\Dm(\rho_k^+),\; {\rm where}\;\;\theta_k\uparrow
+\infty\;\;\fr\;\;k\rightarrow\infty$. The latter fact means that
the operator $H_\gamma$ has an infinite number of eigenvalues in the
gap $(\lambda_-,\,\lambda_+)$ for any $\gamma<0$. These eigenvalues
cluster to the edge $\lambda_+$ only and formula (\ref{frm}) with
$l(\lambda_+)=\infty$ is valid for them.\vskip2mm

(iii) If $\gamma<0$, for the branches of virtual eigenvalues of the
operator $H_\gamma$ at the edge $\lambda_-$ of
$(\lambda_-,\,\lambda_+)$ the analogous claims are valid like in the
case of the edge $\lambda_+$.
\end{proposition}

In order to get in what follows an asymptotic representation for
virtual eigenvalues, we need the following claim:

\begin{lemma}\label{lmasymp}
If $\mu:\,(a,b)\rightarrow\R$ be an increasing continuous function
having the asymptotic representation for $\lambda\uparrow b$:
$\mu(\lambda)=\frac{A}{f(\lambda)}+O(1)$, where $A>0$ and
$f:\,(a,b)\rightarrow\R$ is a decreasing continuous function such
that $\lim_{\lambda\uparrow b}f(\lambda)=0$, then\vskip2mm

(a) there is $\bar\gamma>0$ such that for any
$\gamma\in(0,\bar\gamma]$ the equation
$\mu(\lambda)=\frac{1}{\gamma}$ has in $(a,b)$ a unique solution
$\lambda=\rho(\gamma)$;\vskip2mm

(b) the asymptotic representation
$f(\rho(\gamma))=\gamma(A+O(\gamma))$ is valid for $\gamma\downarrow
0$.
\end{lemma}
\begin{proof}
Claim (a) follows from the monotony and continuity of $\mu$ and the
fact that $\lim_{\lambda\uparrow b}\mu(\lambda)=\infty$. Then in
view of  the equality
$\frac{A}{f(\rho(\gamma))}+O(1)=\frac{1}{\gamma}$, we have:
$f(\rho(\gamma)=\gamma A+\gamma f(\lambda(\gamma)O(1)$, hence
$f(\rho(\gamma)=O(\gamma)$ for $\gamma\downarrow 0$. The last two
equalities imply the desired asymptotic representation.
\end{proof}

\subsection{Perturbation of a compact operator}
\label{subsec:pertcompoper}

In what follows we shall use the following result about a
perturbation of a compact operator in a Hilbert space.
\begin{proposition}\label{prpertcompoper}
Let $\Phi(t)\;(t\in[0,\delta]\,\delta>0)$ be a family of linear
compact operators acting in a Hilbert space $\B$ such that
\begin{equation}\label{rprPhit}
\Phi(t)=\Phi(0)+t\theta(t),
\end{equation}
where
\begin{equation}\label{cndtht}
\bar\theta=\sup_{t\in[0,\delta]}\Vert\theta(t)\Vert<\infty
\end{equation}
and the operator $\Phi(0)$ has positive eigenvalues. Let
$\mu_1^0\ge\mu_2^0\ge\dots\ge\mu_N^0\;(N\le\infty)$ be such
eigenvalues of $\Phi(0)$ arranged counting their multiplicities and
$e_1^0,e_2^0,\\\dots,e_N^0$ be an orthonormal system of eigenvectors
corresponding to them. Then\vskip2mm

(i) if $N<\infty$, there exists $\tilde\delta\in[0,\delta]$ such
that \vskip1mm

\indent\indent (a) for any $t\in[0,\tilde\delta]$ the operator
$\Phi(t)$ has $N$ positive eigenvalues
$\mu_1(t)\ge\mu_2(t)\ge\dots\ge\mu_N(t)$ (counting their
multiplicities) having the asymptotic representation
\begin{equation}\label{asymptlmu}
\mu_k(t)=\mu_k^0+O(t)\quad\rm{as}\quad t\rightarrow
0\quad(k=1,2,\dots, N))
\end{equation}
and all the rest of positive eigenvalues of $\Phi(t)$ (if they
exist) have the asymptotic representation
\begin{equation}\label{asymptlmurest}
\mu_k(t)=O(t)\quad\rm{as}\quad t\rightarrow 0\quad(k> N);
\end{equation}
\vskip1mm

\indent\indent (b) if $m(j)$ is the multiplicity of the positive
eigenvalue $\mu_j^0$ of the operator
$\Phi(0)\;(j\in\{1,2,\dots,N\})$ and
\begin{equation}\label{groupeigval}
\mu_{l(j)}(t)\ge\mu_{l(j)+1}(t)\ge\dots\ge\mu_{l(j)+m(j)-1}(t)\quad(l(j)\in\{1,2,\dots,N\})
\end{equation}
is the group of the eigenvalues of the operator $\Phi(t)$  which
tend to $\mu_j^0$ as $t\rightarrow 0$, then for any
$t\in[0,\tilde\delta]$ it is possible to choose an orthonormal basis
$e_k(t)\;(k\in\{l(j),l(j)+1,\dots,l(j)+m(j)-1\})$ in the subspace
spanned by eigenvectors of $\Phi(t)$ corresponding to these
eigenvalues, such that for any
$k\in\{l(j),l(j)+1,\dots,l(j)+m(j)-1\}$ $\Vert
e_k(t)-e_k^0\Vert=O(t)$ for $t\rightarrow 0$.\vskip2mm

(ii) If $N=\infty$, for any natural $n$ there exists
$\tilde\delta\in[0,\delta]$ such that for any $t\in[0,\tilde\delta]$
the operator $\Phi(t)$ has $n$ positive eigenvalues
$\mu_1(t)\ge\mu_2(t)\ge\dots\ge\mu_n(t)$ (counting their
multiplicities). Furthermore, these eigenvalues and eigenvectors,
corresponding to them, have the same properties as in claim (i)
except (\ref{asymptlmurest}), with the number $n$ instead of $N$.
\end{proposition}
\begin{proof}
(i) In view of (\ref{rprPhit}) and (\ref{cndtht}),
$\Phi(0)-\bar\theta tI\le\Phi(t)\le\Phi(0)+\bar\theta tI$ for any
$t\in[0,\delta]$. Hence by Lemma 3.4 of \cite{Ar-Zl1} ( a
modification of the comparison theorem on the base of the minimax
principle), if $\tilde
\delta<\min\{\delta,\frac{\tilde\mu_N^0}{\bar\theta}\}$, then for
any $t\in[0,\tilde\delta]$ the operator $\Phi(t)$ has $N$ positive
eigenvalues $\mu_1(t)\le\mu_2(t)\le\dots\le\mu_N(t)$ (counting their
multiplicities), for which the estimates
\begin{equation}\label{esttlmukt}
\forall\;t\in[0,\tilde\delta]:\quad\mu_k^0-\bar\theta
t\le\mu_k(t)\le\mu_k^0+\bar\theta t\quad (k=1,2,\dots, N)
\end{equation}
are valid and all the rest of positive eigenvalues of $\Phi(t)$ (if
they exist) satisfy the estimates $\mu_k(t)\le\bar\theta t\quad
(k>N)$ for any $t\in[0,\tilde\delta]$. These estimates imply the
desired asymptotic formulas (\ref{asymptlmu}) and
(\ref{asymptlmurest}). Claim (i-a) is proven.

Let $\tilde\mu_1^0>\tilde\mu_2^0>\dots>\tilde\mu_{\tilde
N}^0\;(\tilde N\le N)$ be all the mutually different positive
eigenvalues of the operator $\Phi(0)$.  The sum of their
multiplicities is equal to $N$. Let $E_1^0,E_2^0,\dots,E_{\tilde
N}^0$ be the eigenspaces of $\Phi(0)$ corresponding to these
eigenvalues. In order to prove the second part of claim (i), we
shall show that there exists an operator function
$U(\cdot):\,[0,\tilde\delta]\rightarrow\Bc(\B)$ taking unitary
values such that for any $j\in\{1,2,\dots \tilde N\}$ the subspace
$E_j(t)=U(E_j^0)$ is spanned by the eigenspaces of those eigenvalues
of the operator $\Phi(t)$ which tend to the eigenvalue $\tilde
\mu_j^0$ of the operator $\Phi(0)$ as $t\rightarrow 0$ and moreover,
\begin{equation}\label{estdiffIandU}
\Vert I-U(t)\Vert=O(t)\quad \rm{for}\quad t\rightarrow 0;
\end{equation}
Using (\ref{esttlmukt}), we can choose $r>0$ such that for any
$j\in\{1,2,\dots N\}$ and $t\in[0,\tilde\delta]$ the circle $C_r^j$
in the complex plane with the radius $r$ and the center at the
eigenvalue $\tilde \mu_j^0$ lies in the resolvent set of the
operator $\Phi(t)$ and surrounds the group of eigenvalues
(\ref{groupeigval}) of the operator $\Phi(t)$. Consider the total
orthogonal projection $P_j(t)$ on the subspace $E_j(t)$ spanned by
the eigenspaces of the above group of eigenvalues of $\Phi(t)$:
$P_j(t)=-\frac{1}{2\pi i}\oint_{C_r^j}R_\lambda(\Phi(t))\,d\lambda$.
In particular, $P_j^0=P_j(0)$ is the orthogonal projection on the
eigenspace $E_j^0$ of the operator $\Phi(0)$ corresponding to its
eigenvalue $\tilde\mu_j^0$. Then
\begin{equation}\label{diffproj}
P_j(t)-P_j^0=-\frac{1}{2\pi
i}\oint_{C_r^j}\big(R_\lambda(\Phi(t))-R_\lambda(\Phi(0))\big)\,d\lambda.
\end{equation}
From (\ref{rprPhit}) we get easily the representation
$R_\lambda(\Phi(t))-R_\lambda(\Phi(0))=-tR_\lambda(\Phi(0))\\\theta(t)R_\lambda(\Phi(t))$.
Then choosing $\tilde\delta>0$ such that
\begin{equation*}
\forall\;t\in[0,\tilde\delta],\quad\forall\;\lambda\in\bigcup_{j=1}^NC_r^j:\quad
\Vert
tR_\lambda(\Phi(0))\theta(t)\Vert\le\tilde\delta\bar\theta\Vert
R_\lambda(\Phi(0))\Vert<1
\end{equation*}
and using standard arguments, we obtain from representation
(\ref{diffproj}) that
\begin{equation*}
\exists\;C>0,\quad\forall\;t\in[0,\tilde\delta],\quad\forall\;j\in\{1,2,\dots
\tilde N\}:\quad\Vert P_j(t)-P_j^0\Vert\le Ct.
\end{equation*}
Then after a suitable restriction of the interval $[0,\tilde\delta]$
the inequality $\Vert P_j(t)-P_j^0\Vert\le 1$ is valid for any
$t\in[0,\tilde\delta]$, which implies easily that the operator
$Q(t)=P_j(t)\vert_{E_j^0}$ realizes a linear isomorphism between the
subspaces $E_j^0$ and $E_j(t)$. Since for any $x\in E_j^0\;$ $\Vert
Q(t)x-x\Vert\le\Vert P_j(t)-P_j^0\Vert\Vert x\Vert\le Ct\Vert
x\Vert$, then for the sequence of vectors
$e_k(t)=Q(t)e_k^0\;(k\in\{k(j),k(j)+1,\dots,k(j)+m(j)-1\})$ claim
(i-b) is valid.

(ii) Claim (ii) follows straightforwardly from all the arguments
used in the proof of claim (i).
\end{proof}

\subsection{The case of non-degenerate edges of a gap of the spectrum of the unperturbed operator}
\label{subsec:nondegedg}

In what follows we need some results about the Birman-Schwinger
operator corresponding to the Schr\"odinger operator
$H_0=-\Delta+V(\e)\cdot$ with a periodic potential $V(\e)$ and the
perturbation $W(\e)$ of this potential (see Section
\ref{sec:preliminaries}).

\subsubsection{Compactness of Birman-Schwinger operator}
\label{subsubsec:compBirmSchwing}

\begin{proposition}\label{compBirSchw}
If the periodic potential $V(\e)$ satisfies the condition
(\ref{condperpotent}) and the perturbation $W(\e)$ is measurable and
bounded in $\R^d$ and $\lim_{|\e|\rightarrow\infty} W(\e)=0$, then
for some $\lambda_0\in\R\setminus\sigma(H_0)$ the operator
$R_{\lambda_0}|W|^{\frac{1}{2}}$ is compact, hence for the operator
$H_\gamma$ and the Birman-Schwinger operator
$X_W(\lambda)=W^{\frac{1}{2}}R_\lambda(H_0)|W|^{\frac{1}{2}}$ all
the claims of Proposition \ref{BrScwspct} are valid.
\end{proposition}
\begin{proof}
By claim (ii) of Proposition \ref{propselfadj}, the operator $H_0$
is self-adjoint, bounded below and its domain is $W_2^2(\R^d)$. It
is enough to prove that the operator
$|W|^{\frac{1}{2}}R_{\lambda_0}(H_0)^{\frac{1}{2}}=(R_{\lambda_0}(H_0)^{\frac{1}{2}}|W|^{\frac{1}{2}})^\star$
is compact for some $\lambda_0<\inf(\sigma(H_0))$, that is the set
\begin{equation}\label{defK}
K=\{u\in W_2^2(\R^d)\;\vert\;\big((H_0-\lambda_0I) u,u\big)\le 1\}
\end{equation}
is compact in the semi-norm $\sqrt{(|W|u,u)}$. By claims (ii) and
(iii) of Lemma \ref{lmestVuOm} and claims (ii) and (iii) of Lemma
\ref{lmestVuRd}, it is possible to choose
$\lambda_0<\inf(\sigma(H_0))$ such that for some $m,\,m_1>0$ and any
$u\in W_2^2(\R^d)$, $\lb\in\Gamma$
$\int_{\Omega_\lb}(H_0u-\lambda_0u)\bar u\,d\e\ge
m\int_{\Omega_\lb}(|\nabla u|^2+|u|^2)\,d\e$ and
$\big((H_0-\lambda_0I)u,u\big)\ge m_1(\|\nabla u\|^2+\|u\|^2)$,
where $\Omega_\lb=\Omega+\{\lb\}$ (recall that $\Gamma$ is the
lattice of periodicity of $V(\e)$). Then
$\int_{G_N}|W(\e)||u(\e)|^2\,d\e\le m\sup_{\e\in
G_N}|W(\e)|\big((H_0-\lambda_0I)u,u\big)$, where
$G_N=\bigcup_{|\lb|_\infty\ge N}\Omega_\lb$ and
$|\lb|_\infty=\max_{1\le j\le d}|(\lb)_j|$. These circumstances and
the compactness of the embedding of $W_2^1(\R^d\setminus G_N)$ into
$L_2(\R^d\setminus G_N)$ imply that the set $K$, defined by
(\ref{defK}), is compact in the semi-norm $\sqrt{(|W|u,u)}$.
\end{proof}

\subsubsection{Representation of Birman-Schwinger operator in the case of a non-degenerate edge}
\label{subsubsec:reprBirmSchwingnondeg}

\begin{proposition}\label{rprBirSchw}
Let $(\lambda_-,\,\lambda_+)$ be a gap of the spectrum of the
unperturbed operator $H_0$. Assume that $W(\e)$ is measurable and
bounded in $\R^d$, $W\in L_1(\R^d)$ and $W(\e)\ge 0$ a.e. on $\R^d$.
Let us take $\delta\in(0,\,(\lambda_+-\lambda_-))$.\vskip2mm

(i) If $d=1$, the edge $\lambda_+$ is non-degenerate and condition
(\ref{cndWd1}) is satisfied, then for
$\lambda\in((\lambda_+-\delta,\;\lambda_+)$ the representation is
valid:
\begin{equation}\label{rprXwd1}
X_W(\lambda)=\Phi_+(\lambda)+\Theta_+(\lambda),
\end{equation}
where $\Phi_+(\lambda))$ is a rank one operator of the form
\begin{equation}\label{dfPhipld1}
\Phi_+(\lambda)=\frac{\sqrt{m^+_1}}{\sqrt{2(\lambda_+-\lambda)}}(\cdot,v^+_1)v^+_1,
\end{equation}
$m_1^+$ and $v^+_1$ are defined by (\ref{dfmupl}) and
(\ref{dfvlpld1}) and $\Theta_+(\lambda)$ is a bounded operator in
$L_2(\R)$ such that
\begin{eqnarray}\label{bndthtpl}
\sup_{\lambda\in(\lambda_+-\delta,\;\lambda_+)}\|\Theta_+(\lambda)\|<\infty
\end{eqnarray}
\vskip2mm

(ii) If $d=2$, the edge $\lambda_+$ is non-degenerate and condition
(\ref{cndWd2}) is satisfied, then for
$\lambda\in(\lambda_+-\delta,\;\lambda_+)$ the representation
(\ref{rprXwd1}) is valid, where $\Theta_+(\lambda)$ is a bounded
operator in $L_2(\R^2)$ such that condition (\ref{bndthtpl}) is
satisfied, $\Phi_+(\lambda)$ is a finite rank operator of the form
\begin{equation}\label{dfPhipld2}
\Phi_+(\lambda)=\frac{1}{2\pi}\ln\left(\frac{1}{\lambda_+-\lambda}\right)G_W^+,
\end{equation}
and $G_W^+$ is defined by (\ref{dfGW}), in which $v_k^+$ is defined
by (\ref{dfvlpl}), (\ref{dfvlxp}) and (\ref{Bloch1}); \vskip2mm

(iii) If $d\ge 3$ and the edge $\lambda_+$ is non-degenerate, then
\begin{equation}\label{bndXw}
\sup_{\lambda\in(\lambda_+-\delta,\;\lambda_+)}\|X_W(\lambda)\|<\infty.
\end{equation}
\vskip2mm

(iv) The claims analogous to (i)-(iii) are valid with $\lambda_-$,
$(\lambda_-,\;\lambda_-+\delta)$, $m_1^-$, $G_W^-$ and $v_k^-$
instead of, respectively, $\lambda_+$,
$(\lambda_+-\delta,\;\lambda_+)$, $m_1^+$, $G_W^+$ and $v_k^+$.
\end{proposition}
\begin{proof}
Using claim (i) of Proposition \ref{prrprres}, we get:
$X_W(\lambda)=\sqrt{W}K(\lambda)\sqrt{W}+\sqrt{W}\Theta(\lambda)\sqrt{W}$,
where, in view of (\ref{bndtht2}) and the boundedness of $W(\e)$,
\begin{eqnarray}\label{bndtht1}
&&\sup_{\lambda\in(\lambda_+-\delta,\;\lambda_+)}\|\sqrt{W}\Theta(\lambda)\sqrt{W}\|\le\sup_{\e\in\R^d}|W(\e)|
\nonumber\times\\
&&\sup_{\lambda\in(\lambda_+-\delta,\;\lambda_+)}\|\Theta(\lambda)\|<\infty.
\end{eqnarray}
Assume that $d=1$ and $\lambda\in(\lambda_+-\delta,\;\lambda_+)$.
Observe that the multiplication operator $\sqrt{W}\cdot$ maps the
set $L_{2,0}(\R)$ into itself. In view of claim (ii,a) of
Proposition \ref{prrprres},
$\sqrt{W}K(\lambda)\sqrt{W}\vert_{L_{2,0}(\R^d)}$ is an operator
with the integral kernel
\begin{equation*}
\sqrt{W(x)}F^+(x,s,\lambda)\sqrt{W(s)}+\sqrt{W(x)}\tilde
K^+(x,s,\lambda)\sqrt{W(s)},
\end{equation*}
and
\begin{equation*}
(\Phi_+(\lambda)f)(x)=\frac{1}{2\mu^+\sqrt{\lambda_+-\lambda}}
\int_{-\infty}^\infty\sqrt{W(x)}b^+(x,p_1)\overline{b^+(s,p_1)}\sqrt{W(s)}f(s)\,ds.
\end{equation*}
is the integral kernel of the operator $\sqrt{W(x)}
F^+(x,s,\lambda)\sqrt{W(s)}$. Since $W\in L_1(\R)$ and, in view of
claim (i) of Theorem \ref{thmainApp}, the function $b^+(x,p_1)$ is
bounded on $\R$, then the function $v^+_1(x)$, defined by
(\ref{dfvlpld1}), belongs to $L_2(\R)$. Hence the operator
$\Phi_+(\lambda)$ can be written in the form (\ref{dfPhipld1}).
Furthermore, in view of (\ref{cndK1pl}), there is $M>0$ such that
for any $\lambda\in(\lambda_+-\delta,\;\lambda_+)$ and $x,s\in\R$
$\sqrt{W(x}|\tilde K^+(x,s,\lambda)|\sqrt{W(s)}\le
M(1+|x-s|)\sqrt{W(x)}\sqrt{W(s)}$. Then, in view of condition
(\ref{cndWd1}),
\begin{eqnarray*}
&&\sup_{\lambda\in((\lambda_++\lambda_-)/2,\;\lambda_+)}
\int_{-\infty}^\infty\int_{-\infty}^\infty W(x)|\tilde K^+(x,s,\lambda)|^2W(s)\,dx\,ds\le\\
&&M^2\int_{-\infty}^\infty\int_{-\infty}^\infty W(x)(1+|x-s|)^2W(s)\,dx\,ds<\infty.
\end{eqnarray*}
This means that for any $\lambda\in(\lambda_+-\delta,\;\lambda_+)$
the operator $\tilde K^+(\lambda)$ with the integral kernel
$\sqrt{W(x)}\tilde K^+(x,s,\lambda)\sqrt{W(s)}$ belongs to the
Hilbert-Schmidt class (hence it is bounded in $L_2(\R)$), and
furthermore $\sup_{\lambda\in(\lambda_+-\delta,\;\lambda_+)}\|\tilde
K^+(\lambda)\|<\infty$. Then, in view of (\ref{bndtht1}) and the
density of $L_{2,0}(\R)$ in $L_2(\R)$, we obtain representation
(\ref{rprXwd1}), where the operator $\Theta_+(\lambda)=\tilde
K^+(\lambda)+\sqrt{W}\Theta(\lambda)\sqrt{W}$ has the property
(\ref{bndthtpl}). So, we have proved claim (i). Claims (ii) and
(iii) are proved analogously using claims (ii,b) and (i) of
Proposition \ref{prrprres}. Claim (iv) is proved in the same manner
as the previous ones.
\end{proof}

\subsubsection{Linear independence of weighted Bloch functions}
\label{subsubsec:linindepBloch}

For the proof of Theorem \ref{thnondegedg} we need three lemmas.

\begin{lemma}\label{lmBloch}(\cite{Kuch})
A finite sequence of Bloch functions
\begin{equation}\label{seqBloch}
\{b_k(\e)=\exp(i\J_k\cdot\e)e_k(\e)\}_{k=1}^N,
\end{equation}
corresponding to mutually different quasi-momenta
$\J_k\in\T^d\,(k=1,2,\dots,N)$, is linearly independent. Recall that
the functions $e_k(\e)$ are $\Gamma$-periodic.
\end{lemma}
\begin{proof}
Consider the functions $b_k(\e)$ as distributions from the space
$\mathcal{S}^\prime$, which is dual to the Schwartz space
$\mathcal{S}$. As it is known, the Gelfand-Fourier-Floquet transform
\begin{equation}\label{unit}
\hat
f(\e,\J)=(Uf)(\e,\J)=\frac{1}{(2\pi)^{\frac{d}{2}}}\sum_{\lb\in\Gamma}\exp(i\J\cdot\lb)f(\e-\lb)
\end{equation}
maps the space $\mathcal{S}^\prime$ onto the space of distributions
$\mathcal{L}^\prime$, which is dual to the set $\mathcal{L}$ of all
$C^\infty$-sections of the direct integral
$\int_{\T^d}^\oplus\B_\J\,d\J$ (\cite{Kuch}). Recall that $\B_{\J}$
is the Hilbert space  of functions $u\in L_{2,loc}(\R^d)$ satisfying
the condition (\ref{Htau}) with the inner product defined by
(\ref{dfinnprd}-a). Using the $\Gamma$-periodicity of $e_k(\e)$, we
have:
\begin{eqnarray}\label{transpsi}
&&\hskip-12mm\hat
b_k(\e,\J)=(Ub_k)(\e,\J)=\frac{1}{(2\pi)^{\frac{d}{2}}}\sum_{\lb\in\Gamma}
\exp(i\J\cdot\lb+\J_k\cdot(\e-\lb))e_k(\e)=\nonumber\\
&&\hskip-12mm\frac{\exp(i\J\cdot\e)}{(2\pi)^{\frac{d}{2}}}\sum_{\lb\in\Z^d}
\exp(i(\J-\J_k)\cdot\lb)e_k(\e)=\frac{\exp(i\J\cdot\e)}{(2\pi)^{\frac{d}{2}}}\,\delta(\J-\J_k)
e_k(\e)
\end{eqnarray}
$(k=1,2,\dots,N)$. Here $\delta(\J)$ is the Dirac delta function.
Assume, on the contrary, that $\{b_k(\e)\}_{k=1}^N$ are linearly
dependent. Then
\begin{equation*}
b_{k_0}(\e)=\sum_{k\in\{1,2,\dots,N\}\setminus\{k_0\}}c_kb_k(\e),
\end{equation*}
hence
\begin{equation*}
\hat
b_{k_0}(\e,\J)=\sum_{k\in\{1,2,\dots,N\}\setminus\{k_0\}}c_k\hat
b_k(\e,\J).
\end{equation*}
In view of (\ref{transpsi}), the support of the distribution in
l.h.s. of the last equality is $\{\J_{k_0}\}$, but the distribution
in its r.h.s. vanishes in a neighborhood of $\J_{k_0}$. We have a
contradiction. Thus, the sequence $\{b_k(\e)\}_{k=1}^N$ is linearly
independent.
\end{proof}
\begin{lemma}\label{lmlinidep}
Assume that the unperturbed potential $V(\e)$ satisfies condition
(\ref{Hold}). Let $\{b_k(\e)\}_{k=1}^N$ be a finite sequence of
Bloch functions of the form (\ref{seqBloch}) corresponding to the
same energy level $\lambda$ of the unperturbed Hamiltonian
$H_0=-\Delta+V(\e)\cdot$ and having mutually different quasi-momenta
$\J_k\in\T^d\,(k=1,2,\dots,N)$. If $W\in L_1(\R^d)$, $W(\e)\ge 0$
a.e. on $\R^d$ and  $W(\e)>0$ on a set of positive measure, then the
sequence of functions $v_k(\e)=\sqrt{W(\e)}b_k(\e)\,(k=1,2,\dots,N)$
belongs to $L_2(\R^d)$ and it is linearly independent.
\end{lemma}
\begin{proof}
Since all the Bloch functions are bounded on $\R^d$, the inclusion
$v_k\in L_2(\R^d)$ is obvious. Since $b_k(\e)$ correspond to the
same energy level $\lambda$, they are solutions of the Schr\"odinger
equation $H_0b=\lambda b$ belonging to $W_{2,\,loc}^2(\R^d)$.
Assume, on the contrary, that $\{v_k(\e)\}_{k=1}^N$ are linearly
dependent, that is there exist constants $c_1,c_2,\dots,c_N$ such
that $\sum_{k=1}^N|c_k|>0$ and $\sum_{k=1}^N
c_kv_k(\e)=\sqrt{W(\e)}\sum_{k=1}^Nc_kb_k(\e)=0$ for almost all
$\e\in\R^d$. The latter equality implies that the function
$b(\e)=\sum_{k=1}^Nc_kb_k(\e)$ vanishes on a set of positive
measure. On the other hand, since by Lemma \ref{lmBloch} the
functions $b_k(\e)$ are linearly independent, then the function
$b(\e)$ is a non-trivial solution of the equation $H_0b=\lambda b$
belonging to $W_{2,\,loc}^2(\R^d)$. Observe that its first
generalized derivatives $w_j=D_jb\,(j=1,2,\dots d)$ belong to
$W_{2,\,loc}^1(\R^d)$ and satisfy the equations $\Delta w_j+\lambda
w_j=D_j(V(\e)b)$. Observe that by claim (i) of Theorem
\ref{thmainApp}, the function $b(\e)$ is continuous on $\R^d$.
Then, in view of condition (\ref{Hold}) and by the elliptic
regularity theorem (\cite{Gil-Tr}, Theorem 8.22), the function
$b(\e)$ is locally $C^{1,\alpha}\;(\alpha\in(0,1))$. Then it follows
from Theorem 1.7. of \cite{Har-Sim} that the set of zeros of $b(\e)$
(the nodal set) has a locally finite $d-1$-dimensional Hausdorff
measure, hence it has a zero Lebesgue measure. We have a
contradiction. Thus, the sequence $\{v_k(\e)\}_{k=1}^N$ is linearly
independent.
\end{proof}

It is not difficult to prove the following consequence of Lemma
\ref{lmlinidep}:

\begin{lemma}\label{lmGWnondeg}
Under all the conditions of Lemma \ref{lmlinidep} the finite rank
operator $G=\sum_{k=1}^N\theta_k\,(\,\cdot,\,v_k)v_k$ ($\theta_k>0$
for any $k\in\{1,2,\dots,N\}$) is non-negative, it has exactly $N$
positive eigenvalues (counting their multiplicities) and they
coincide with the eigenvalues of the matrix
$\left((\theta_k\theta_l)^{\frac{1}{2}}(v_l,v_k)\right)_{k,l=1}^N$.
\end{lemma}

\subsubsection{Proof of Theorem \ref{thnondegedg}}
\label{subsubsec:proofthnondegedg}

\begin{proof}
Since each positive characteristic branch $\mu_k^+(\lambda)$ of the
operator $H_0$ w.r. to $W$ on the gap $(\lambda_-,\,\lambda_+)$ is
an increasing function (see Remark \ref{rembranch}), then by
Proposition \ref{ordl} each eigenvalue $\rho_k^+(\gamma)$ of the
operator $H_\gamma$ appearing in the gap $(\lambda_-,\,\lambda_+)$
for $\gamma<0$ increases, hence it cannot approach the edge
$\lambda_-$ as $\gamma\uparrow 0$. Hence it cannot be a virtual
eigenvalue of $H_\gamma$ at the edge $\lambda_-$, that is claim (i)
is valid (see Definitions \ref{gone1}, \ref{gde1} and \ref{gde2}).

Observe that, as it is easy to show, for $d\le 2$ each of the
conditions (\ref{cndWd1}) and (\ref{cndWd2}) imply that $W\in
L_1(\R^d)$. On the other hand, by claim (i) of Theorem
\ref{thmainApp}, the Bloch functions $b_k^+(\e)$ are bounded on
$\R^d$. Hence the corresponding weighted Bloch functions
$v_k^+(\e)=\sqrt{W(\e)}b_k^+(\e)$ belong to $L_2(\R^d)$.

We now turn to the proof of the claim (ii-a)  (the case where
$d=1$). Recall that for any fixed $\lambda\in(\lambda_-,\lambda_+)$
$\{\mu_k^+(\lambda)\}_{k=1}^{m_+}(\lambda)$ is the set of all
positive eigenvalues of the Birman-Schwinger operator $X_W(\lambda)$
arranged in the non-decreasing ordering
$\mu_1^+(\lambda)\ge\mu_2^+(\lambda)\ge\dots\ge\mu_{N_+}^+(\lambda)>0$,
were each $\mu_k^+(\lambda)$ is repeated according to its
multiplicity (the value $N_+=\infty$ is allowed). Denote
$t=\sqrt{\lambda_+-\lambda}$ and $\Phi(t)=tX_W(\lambda_+-t^2)$. In
view of the representation (\ref{rprXwd1}), definition
(\ref{dfPhipld1}) and the relation (\ref{bndthtpl}) (Proposition
\ref{rprBirSchw}), for some $\delta>0$ and any $t\in[0,\delta]$
$\Phi(t)=\Phi(0)+t\theta(t)$, where
$\Phi(0)=\frac{\sqrt{m^+}(\,\cdot\,,v_1^+)v_1^+}{\sqrt{2}}$ and
$\sup_{t\in[0,\delta]}\Vert\theta(t)\Vert<\infty$. As it is clear,
the one-rank operator $\Phi(0)$ has the unique non-zero eigenvalue
$\mu_1^0=\frac{\sqrt{m^+}\|v_1^+\|^2}{\sqrt{2}}$ and it is positive.
Then by claim (i-a) of Proposition \ref{prpertcompoper}, there
exists $\tilde\delta\in[0,\delta]$ such that for any
$t\in[0,\tilde\delta]$ the operator $\Phi(t)$ has a unique simple
positive eigenvalue $\mu_1(t)$ having the asymptotic representation
$\mu_1(t)=\mu_1^0+O(t)$ as $t\rightarrow 0$ and all the rest of
positive eigenvalues of $\Phi(t)$ (if they exist) have the
asymptotic representation $\mu_k(t)=O(t)$ as $t\rightarrow
0\;(k>1)$. Returning from the variable $t$ to the variable $\lambda$
and taking into account that all the eigenvalues of $X_W(\lambda)$
are obtained from the eigenvalues of $\Phi(t)$ via multiplication by
$\frac{1}{t}=\frac{1}{\sqrt{\lambda_+-\lambda}}$, we obtain that for
some $\hat\delta\in(0,\lambda_+-\lambda_-)$ and for any
$\lambda\in[\lambda_+-\hat\delta,\lambda_+)$ the operator
$X_W(\lambda)$ has a unique positive eigenvalue $\mu_1^+(\lambda)$
having the asymptotic representation
\begin{equation}\label{asymptmu}
\mu_1^+(\lambda)=\frac{\sqrt{m^+}\|v_1^+\|^2}{\sqrt{2(\lambda_+-\lambda)}}+O(1)\quad
\rm{for}\quad\lambda\uparrow\lambda_+
\end{equation}
(hence $\lim_{\lambda\uparrow\lambda_+}\mu_1^+(\lambda)=\infty$) and
all the rest of its positive eigenvalues $\mu_k^+(\lambda)\;(k>1)$
(if they exist) are bounded in $(\lambda_+-\hat\delta,\lambda_+)$.
This means that $\mu_1^+(\lambda)$ is a unique main characteristic
branch of $H_0$ w.r. to $W$ near the edge $\lambda_+$ of the gap
$(\lambda_-,\,\lambda_+)$, hence the corresponding asymptotic
multiplicity $M(\lambda_+,H_0,W)$ of $\lambda_+$ is one (see
Definition \ref{gde2}). By Proposition \ref{gpbr} the latter
circumstances mean that for $\gamma<0$ there exists a unique branch
 $\rho_1^+(\gamma)$ of virtual eigenvalues of the operator $H_\gamma$ at the edge $\lambda_+$
 (see Definition \ref{gone1}) and it is the solution of the equation $\mu_1^+(\lambda)=\frac{1}{|\gamma|}$
 for a sufficiently small $|\gamma|$. From (\ref{asymptmu}) and Lemma \ref{lmasymp}, we obtain the desired
 asymptotic formula (\ref{asympteigvd1}). Claim (ii-a) is proven.

 Let us prove claim (ii-b) (the case where $d=2$). By Proposition \ref{rprBirSchw},
 in this case the finite rank operator $\Phi_+(\lambda)$, taking part in the representation
(\ref{rprXwd1}) for $X_W(\lambda)$ has the form (\ref{dfPhipld2}),
where $G_W^+$ is defined by (\ref{dfGW}). Like above, let
$\mu_k^+(\lambda)\;(k=1,2,\dots,N_+)$ be positive characteristic
branches of $H_0$ w.r.t. $W$ in the gap $(\lambda_-,\,\lambda_+)$,
that is they are all the positive eigenvalues of the operator
$X_W(\lambda)$. Denote
\begin{equation}\label{dfPhit}
t=t(\lambda)=\Big(\ln\left(\frac{1}{\lambda_+-\lambda}\right)\Big)^{-1}\quad
\rm{and}\quad \Phi(t)=tX_W\big(\lambda_+-\exp(-t^{-1})\big).
\end{equation}
In view of representation (\ref{rprXwd1}), definition
(\ref{dfPhipld2}) and condition (\ref{bndthtpl}), for some
$\delta>0$ and any $t\in[0,\delta]$ representation (\ref{rprPhit})
is valid with $\Phi(0)=\frac{1}{2\pi}G_W^+$ and $\theta(t)$
satisfying the condition (\ref{cndtht}). Then by claim (i-a) of
Proposition \ref{prpertcompoper} and Lemma \ref{lmGWnondeg}, there
exists $\tilde\delta\in[0,\delta]$ such that for any
$t\in[0,\tilde\delta]$ the operator $\Phi(t)$ has $n_+$ positive
eigenvalues $\mu_1(t)\ge\mu_2(t)\ge\dots\ge\mu_{n_+}(t)$ (counting
their multiplicities), having the asymptotic representation
$\mu_k(t)=\frac{1}{2\pi}\nu_k^++O(t)\;(k=1,2,\dots,n_+)$ for
$t\rightarrow 0$ and all the rest of positive eigenvalues of
$\Phi(t)$ (if they exist) have the asymptotic representation
$\mu_k(t)=O(t)$ for $t\rightarrow 0\;(k>n_+)$. Recall that $\nu_k^+$
are the positive eigenvalues of the operator $G_W^+$. Returning from
the variable $t$ to the variable $\lambda$, we obtain that there
exists $\hat\delta\in (0,\lambda_+-\lambda_-)$ such that
$\mu_1^+(\lambda),\mu_2^+(\lambda),\dots,\mu_{n_+}^+(\lambda)$ exist
for any $\lambda\in[\lambda_+-\hat\delta,\lambda_+)$, they have the
asymptotic representation
\begin{equation}\label{asympmud2}
\mu_k^+(\lambda)=\frac{1}{2\pi}\ln\left(\frac{1}{\lambda_+-\lambda}\right)\nu_k^++O(1)\quad
\rm{for}\quad\lambda\uparrow\lambda_+\quad (k=1,2,\dots,n_+),
\end{equation}
(hence $\lim_{\lambda\uparrow\lambda_+}\mu_k^+(\lambda)=\infty$ for
$k=1,2,\dots,n_+$), and for $k>n_+$ $\mu_k^+(\lambda)$ are bounded
in $(\lambda_+-\hat\delta,\lambda_+)$ (if they exist). This means
that $\mu_k^+(\lambda)\;(k=1,2,\dots,n_+)$ is the set of all main
characteristic branches of $H_0$ w.r. to $W$ near the edge
$\lambda_+$ of $(\lambda_-,\lambda_+)$, hence
$M(\lambda_+,H_0,W)=n_+$. By Proposition \ref{gpbr} the latter
circumstances mean that for $\gamma<0$ there exist exactly $n_+$
branches
$\rho_1^+(\gamma)\le\rho_2^+(\gamma)\le\dots\le\rho_{n_+}^+(\gamma)$
of virtual eigenvalues of the operator $H_\gamma$ near the edge
$\lambda_+$ and each of them is the solution of the equation
$\mu_k^+(\lambda)=\frac{1}{|\gamma|}$ for any
$\gamma\in[-\bar\gamma,0)$  and for some $\bar\gamma>0$. Using Lemma
\ref{lmasymp}, we obtain from (\ref{asympmud2}) the desired
asymptotic formula (\ref{asympteigvd2}). Formula
(\ref{asymptLbThrd2}) follows from (\ref{asympteigvd2}) and the fact
that $\mathrm{tr}(G_W^+)=\sum_{k=1}^{n_+}\nu_k^+$. Claim (ii-b) is
proven.

We shall prove only claim (iii-b) (the case where $d=2$), because
claim (iii-a) (the case where $d=1$) is proved analogously.  For any
fixed $j\in\{1,2,\dots,n_+\}$ consider the group of virtual
eigenvalues
\begin{eqnarray*}
&&\rho_{l(j)}^{\,+}(\gamma)\le\rho_{l(j)+1}^{\,+}(\gamma)\le\dots\le\rho_{l(j)+m(j)-1}^{\,+}(\gamma)\\
&&(l(j)\in\{1,2,\dots,n_+\},\;\gamma\in[-\bar\gamma,0))\nonumber
\end{eqnarray*}
of the operator $H_\gamma$ in $(\lambda_-,\,\lambda_+)$, for which
the quantities $\frac{t(\rho_k^+(\gamma))}{|\gamma|}$
$(k\in\{l(j),l(j)+1,\dots,l(j))+m(j)-1\})$ tend as $\gamma\uparrow
0$ to the eigenvalue $\frac{\nu_j^+}{2\pi}$ of the operator
$\Phi(0)$, whose multiplicity is $m(j)$. Recall that the function
$t(\lambda)$ and the operator function $\Phi(t)$ are defined by
(\ref{dfPhit}) and $g_1^+(\e),g_2^+(\e),\dots,g_{n_+}^+(\e)$ is an
orthonormal sequence of eigenfunctions of the operator
$G_W^+=2\pi\Phi(0)$ corresponding to its positive eigenvalues
$\nu_1^+,\nu_2^+,\dots, \nu_{n_+}^+$. Denote
$\tilde\rho(\gamma)=\rho_{l(j)}^{\,+}(\gamma)$. Let
$\tilde\mu_0^+(\gamma)=\mu_{l(j)}^+\big(\tilde\rho(\gamma)\big)>
\tilde\mu_1^+(\gamma)>\dots>\tilde\mu_{n(\gamma)}^+(\gamma)$
$(n(\gamma)\le m(j)-1)$ be all the mutually different numbers from
the sequence
$\{\mu_{l(j)+k}^+\big(\tilde\rho(\gamma)\big)\}_{k=0}^{m(j)-1}$.
These numbers are positive eigenvalues of the self-adjoint operator
$X_W\big(\tilde\rho(\gamma)\big)$. Let $L_0(\gamma),
L_1(\gamma),\dots, L_{n(\gamma}(\gamma)$ be the eigenspaces of
$X_W\big(\tilde\rho(\gamma)\big)$, corresponding to these
eigenvalues. Observe that each of $L_k(\gamma)$ is the eigenspace of
the operator $\tilde\Phi(\gamma)=\Phi(t(\tilde\rho(\gamma)))$,
corresponding to its eigenvalue
$t(\tilde\rho(\gamma)))\tilde\mu_k^+(\gamma)$. Denote
$L(\gamma)=\bigoplus_{k=0}^{n(\gamma)}L_k(\gamma)$. By claim (i-b)
of Proposition \ref{prpertcompoper}, for a suitable
$\tilde\gamma\in(0,\,\bar\gamma]$   and any
$\gamma\in[0,\tilde\gamma]$ it is possible to choose a basis
\begin{equation*}
e_{l(j)}(\gamma),e_{l(j)+1}(\gamma),\dots,e_{l(j)+m(j)-1}(\gamma)
\end{equation*}
in the subspace $L(\gamma)$ such that $\Vert
e_k(\gamma)-g_k^+\Vert=O(t(\tilde\rho(\gamma))$ for $\gamma\uparrow
0$. Then taking into account the asymptotic formula
(\ref{asympteigvd2}), we get that
\begin{equation}\label{asympteigvec}
\Vert e_k(\gamma)-g_k^+\Vert=O(\gamma)\quad\rm{for}\quad
\gamma\uparrow 0\quad (k=l(j),l(j)+1,\dots,l(j)+m(j)-1),
\end{equation}
For any $\gamma\in[0,\tilde\gamma]$ consider the sequence
$\gamma_0(\gamma)=\gamma$,
$\gamma_k(\gamma)=-\frac{1}{\mu_{l(j)+k}^+\big(\tilde\rho(\gamma)\big)}\;
(k=1,2,\dots,m(j)-1)$. Since the r.h.s. of the asymptotic formula
(\ref{asympmud2}) is the same for all the functions from the
sequence $\{\mu_{l(j)+k}^+(\lambda)\}_{k=0}^{m(j)-1}$, then
\begin{equation*}
\mu_{l(j)}^+\big(\tilde\rho(\gamma)\big)-\mu_{l(j)+k}^+\big(\tilde\rho(\gamma)\big)=
-\gamma^{-1}+\big(\gamma_k(\gamma)\big)^{-1}=O(1)\quad \rm{for}\quad
\gamma\uparrow 0
\end{equation*}
$(k=1,2,\dots,m(j)-1)$. This fact implies  easily that
$\gamma-\gamma_k(\gamma)=O(\gamma^2)$ as $\gamma\uparrow 0$
$(k=0,1,\dots, m(j)-1)$. Observe that
$\tilde\gamma_k(\gamma)=-\frac{1}{\tilde\mu_k^+(\gamma)}\;(k=0,1,\dots,n(\gamma))$
are all the different numbers from the sequence
$\{\gamma_k(\gamma)\}_{k=0}^{m(j)-1}$. It is clear that for any
$k\in\{0,1,\dots, n(\gamma)\}$ the number $\tilde\rho(\gamma)$ is an
eigenvalue of the operator $H_{\tilde\gamma_k(\gamma)}$. Let
$E_k(\gamma)$ be the eigenspace of $H_{\tilde\gamma_k(\gamma)}$,
corresponding to the eigenvalue $\tilde\rho(\gamma)$. By Proposition
\ref{BrScwspct}, the multiplication operator
$W^{1/2}\cdot\vert_{E_k(\gamma)}$ realizes a linear isomorphism
between $E_k(\gamma)$ and $L_k(\gamma)$. This circumstance and the
fact that the subspaces $L_k(\gamma)$ are mutually orthogonal imply
that the subspaces $E_k(\gamma)$ are linearly independent and the
operator $Q(\gamma)=W^{1/2}\cdot\vert_{E(\gamma)}$ realizes a linear
isomorphism between the subspace
$E(\gamma)=E_0(\gamma)+E_1(\gamma)+\dots+E_{n(\gamma)}(\gamma)$ and
the subspace $L(\gamma)$. Let us define the following basis in
$E(\gamma)$: $\psi^+_{\gamma,k}=(Q(\gamma))^{-1} e_k(\gamma)$
$(k\in\{l(j),l(j)+1,\dots,l(j)+m(j)-1\})$. Then we obtain from
(\ref{asympteigvec}) the desired property (\ref{asympteigvec1}).
Claim (iii-b) is proven.

Let us prove claim (iv). In view of (\ref{bndXw}) (claim (iii) of
Proposition \ref{rprBirSchw}), all the positive characteristic
branches of $H_0$ w.r.t. $W$ on the gap $(\lambda_-,\,\lambda_+)$
are bounded, hence $M(\lambda_+,H_0,W)=0$. This fact, Proposition
\ref{gpbr} and claim (i) imply the desired claim (iv).

Claim (v) is proved in the same manner as the previous ones. Claim
(vi) follows from claims (iv) and (v).

Theorem \ref{thnondegedg} is proven.
\end{proof}

\subsection{The case of degenerate edges of a gap of the spectrum of the unperturbed operator}
\label{subsec:degedg}

As above, $(\lambda_-,\,\lambda_+)$ is a gap of the spectrum of the
unperturbed operator $H_0$ and we assume that the perturbation
$W(\e)$ of the periodic potential $V(\e)$ is measurable and bounded
in $\R^d$, $W\in L_1(\R^d)$ and $W(\e)\ge 0$ a.e. on $\R^d$.

\subsubsection{Representation of Birman-Schwinger operator in the case of a degenerate edge}
\label{subsubsec:reprBirmSchwingdeg}

In the case of a degenerate edge of a gap of the spectrum of the unperturbed operator
we have the following modification of Proposition \ref{rprBirSchw}:

\begin{proposition}\label{rprBirSchwdeg}
Let us take $\delta\in(0,\,(\lambda_+-\lambda_-))$.\vskip2mm

(i) If $d-d_+=1$, the edge $\lambda_+$ satisfies condition (C) of
Section \ref{subsec:mainresbirthdeg} and condition (\ref{cndWd1deg})
is satisfied, then for $\lambda\in(\lambda_+-\delta,\;\lambda_+)$
the representation is valid:
\begin{equation}\label{rprXwd1deg}
X_W(\lambda)=\Phi_+(\lambda)+\Theta_+(\lambda),
\end{equation}
where
\begin{equation}\label{dfPhipld1deg}
\Phi_+(\lambda)=\frac{1}{\sqrt{\lambda_+-\lambda}}\frac{\sqrt{2}\pi}{(2\pi)^d}G_W^+,
\end{equation}
the operator $G_W^+$ is defined by (\ref{dfintopGW})-(\ref{dfkerGW})
and $\Theta_+(\lambda)$ is a bounded operator in $L_2(\R^2)$
satisfying the condition
\begin{eqnarray}\label{bndthtpldeg}
&&\bar\Theta_+=\sup_{\lambda\in(\lambda_+-\delta,\;\lambda_+)}\|\Theta_+(\lambda)\|<\infty;
\end{eqnarray}

(ii) if $d-d_+=2$, the edge $\lambda_+$ satisfies condition (C) and
condition (\ref{cndWcod2}) is satisfied, then for
$\lambda\in((\lambda_+-\delta,\;\lambda_+)$ the representation
(\ref{rprXwd1deg}) is valid with $\Theta_+(\lambda)$ satisfying
condition (\ref{bndthtpldeg}) and with $\Phi_+(\lambda)$ having the
form:
\begin{eqnarray}\label{dfPhipld2deg}
&&\Phi_+(\lambda)=\ln\left(\frac{1}{\lambda_+-\lambda}\right)\frac{\sqrt{2}\pi}{(2\pi)^d}G_W^+;
\end{eqnarray}

(iii) if $d-d_+\ge 3$ and the edge $\lambda_+$ satisfies condition
(C), then
\begin{equation*}
\sup_{\lambda\in((\lambda_+-\delta,\;\lambda_+)}\|X_W(\lambda)\|<\infty;
\end{equation*}

(iv) The claims analogous to (i)-(iii) are valid with $d_-$,
$\lambda_-$, $(\lambda_-,\;\lambda_-+\delta)$ and $G_W^-$ instead
of, respectively, $d_+$, $\lambda_+$,
$(\lambda_+-\delta,\;\lambda_+)$ and $G_W^+$.
\end{proposition}
\begin{proof}
The proof is the same as the proof of Proposition \ref{rprBirSchw},
merely instead of Proposition \ref{prrprres} it is supported by
Proposition \ref{prrprres1} with $n_+=n_-=1$.
\end{proof}

\subsubsection{Spectrum of the operator $G_W^+\;(G_W^-)$}
\label{subsubsec:specopGWpl}

In  what follows we need the following
\begin{lemma}\label{lmspecGWpl}
If the edge $\lambda_+\;(\lambda_-)$ satisfies condition (C) of
Section \ref{subsec:mainresbirthdeg}, the unperturbed potential
$V(\e)$ satisfies condition (\ref{Hold}) and $W(\e)> 0$ on a set of
the positive measure, then the integral operator $G_W^+$, defined by
(\ref{dfintopGW})-(\ref{dfkerGW}),\vskip2mm

(i) is self-adjoint, non-negative, belongs to the trace class and
\begin{eqnarray}\label{traceform}
&&\hskip-4mm\hskip-5mm\rm{tr}(G_w^+)=\int_{\R^d}\Gc_W^+(\bs,\bs)\,d\bs=
\int_{F^+}\int_{\R^d}\Qc_W^+(\bs,\bs,\J)\,d\bs\sqrt{m^+(\J)}\,dF(\J)
\end{eqnarray}
\vskip2mm

(ii) has an infinite number of positive eigenvalues having finite
multiplicities.\vskip2mm

The analogous properties has the operator $G_W^-$.
\end{lemma}
\begin{proof}
 We shall restrict ourselves only on the operator $G_W^+$, since the operator $G_W^-$ has the same structure.
Let us prove claim (i). Using claim (i) of Corollary
\ref{cormainApp} and compactness of the submanifold $F^+$, it is
possible to construct a finite open covering $\{U_k\}_{k=1}^L$ of
$F^+$ such that for any $k\in\{1,2,\dots,L\}$
$\Qc_W^+(\e,\bs,\J)=v_k^+(\e,\J)v_k^+(\bs,\J)$ for any $\J\in U_k$,
where $v_k^+(\e,\J)=\sqrt{W(\e)}b_k^+(\e,\J)$ and
$b_k^+(\e,\J)\\(\J\in U_k)$ is a branch of Bloch functions,
corresponding to the energy level $\lambda_+$ of the family of the
operators $H(\J)\;(\J\in U_k)$ such that the mapping $\J\rightarrow
b_k^+(\cdot,\J)$ belongs to the class $C(U_k,\, C(\Omega))$. Hence
since $W\in L_1(\R^d)$, the inclusion
\begin{equation}\label{mmbshipvpl}
\big(\J\rightarrow v_k^+(\cdot,\J)\big)\in C(U_k,\,L_2(\R^d)).
\end{equation}
is valid. Let $\{\phi_k(\J)\}_{k=1}^L$ be a decomposition of unit,
corresponding to the covering $\{U_k\}_{k=1}^L$ such that each
$\phi_k$ is continuous on $F^+$ and positive on $U_k$. Then the
operator $G_W^+$ can be represented in the form $G_W^+=\sum_{k=1}^L
G_k$, where $G_kf=\int_{\R^d}\Gc_k(\e,\bs)f(\bs)\,d\bs\;(f\in
L_2(\R^d))$ and
\begin{equation*}
\Gc_k(\e,\bs)=\int_{U_k}v_k^+(\e,\J)v_k^+(\bs,\J)\phi_k(\J)\sqrt{m^+(\J)}\,dF(\J).
\end{equation*}
In order to prove that $G_W^+$ belongs to the trace class, it is
enough to show that each of $G_k$ belongs to the trace class.  We
see that
\begin{equation*}
\Gc_k(\e,\bs)=\int_{U_k}\mathcal{A}_k(\e,\J)\overline{\mathcal{A}_k(\bs,\J)}\,dF(\J),
\end{equation*}
where
$\mathcal{A}_k(\e,\J)=v_k^+(\e,\J)(\phi_k(\J))^{1/2}(m^+(\J))^{1/4}$.
Consider the integral operator
\begin{equation*}
(A_k\phi)(\e)=\int_{U_k}\mathcal{A}_k(\e,\J)\phi(\J)\,dF(\J),
\end{equation*}
which acts from $L_2(U_k)$ to  $L_2(\R^d)$, hence
$G_k=A_k(A_k)^\star$. Formally in order to prove that the operator
$G_k$ belongs to the trace class, it is enough to show that $A_k$
belongs to the Hilbert-Schmidt class. But since the domain and the
range of $A_k$ are different, we are forced to use a direct sum
argument. Consider the Hilbert space $\mathcal{\tilde
H}=L_2(\R^d)\oplus L_2(F^+)$ and the operator $\tilde A_W$ acting in
it and defined by the operator matrix:
\begin{eqnarray}\label{dftlAW}
\tilde A_k=\left(\begin{array}{ll}
0&\;A_k\\
A_k^\star&\;0
\end{array}\right)
\end{eqnarray}
Since the dispersion function $\lambda^+(\J)$ is real-analytic and
satisfies the Morse-Bott condition in a neighborhood of $F^+$, then
in view of (\ref{dfmplp})
\begin{equation}\label{mmbshipHess}
m^+\in C(F^+,,\R).
\end{equation}
This fact and inclusion (\ref{mmbshipvpl}) imply easily that
$\int_{\R^d}\int_{U_k}|\mathcal{A}_k^+(\e,\J)|^2\,d\e\,dF(\J)<\infty$,
hence in view of (\ref{dftlAW}) the matrix integral kernel of the
operator $\tilde A_k$ is square integrable on $\R^d\times U_k$.
Therefore $\tilde A_k$ belongs to the Hilbert-Schmidt class and it
is self-adjoint. On the other hand, since $\tilde
A_k^2=\rm{diag}(G_W^+,\;A_k^\star A_k)$, the operator $G_k$ is
self-adjoint, non-negative and belongs to the trace class. In view
of Theorem 3.1 of \cite{Bris},
$\rm{tr}(G_k)=\int_{\R^d}\tilde\Gc_k(\bs,\bs)\,d\bs$, where
$\tilde\Gc_k(\e,\bs)$ is the local average of $\Gc_k(\e,\bs)$
defined in \cite{Bris}. On the other hand, using the decomposition
$G_k^+=A_k(A_k)^\star$ and the arguments of the proof  and Theorem
3.5 from \cite{Bris}, it is not difficult to show that
$\tilde\Gc_k(\e,\e)=\Gc_k(\e,\e)$ a.e. on $\R^d$. Hence, taking into
account that $\mathrm{tr}(G_w^+)=\sum_{k=1}^L\mathrm{tr}(G_k)$, we
obtain the desired trace formula (\ref{traceform}). Claim (i) is
proven.

Let us prove claim (ii). Claim (i) implies that the spectrum of the
operator $G_W^+$ consists of at most countable number of
non-negative eigenvalues, which can cluster only to zero and each
positive eigenvalue has a finite multiplicity. Let
$\J_1,\J_2,\dots,\J_N$ be mutually different quasi-momenta belonging
to the submanifold $F^+$ and
$\Oc^\epsilon_1,2,\dots,\Oc^\epsilon_N\subset F^+$ are disjoint
neighborhoods of these points having the same diameter $\epsilon>0$
and the same volume $F(\epsilon)$. We can choose $\epsilon$ such
that for any $j\in\{1,2,\dots,N\}$ there is
$k=k(j)\in\{1,2,\dots,L\}$ such that $\Oc^\epsilon_j\subset U_k$.
Denote $\tilde v_j^+=v_{k(j)}^+$. We have for $f\in L_2(\R^d)$:
\begin{equation}\label{lowestGWpl}
(G_W^+f,\,f)\ge F(\epsilon)(\tilde
G_Nf,\,f)+(\Theta_{N,\epsilon}f,\,f),
\end{equation}
where
\begin{equation}\label{dftlGN}
\tilde G_N=\sum_{j=1}^N(\,\cdot,\,\tilde v_j^+(\cdot,\J_j))\tilde
v_j^+(\cdot,\J_j)\sqrt{m^+(\J_j)}
\end{equation}
and
\begin{eqnarray*}
&&\Theta_{N,\epsilon}=\sum_{k=1}^N\int_{\Oc^\epsilon_k}
\Big((\,\cdot,\,\tilde v_j^+(\cdot,\J))\tilde
v_j^+(\cdot,\J)\sqrt{m^+(\J)}-
\\
&&(\,\cdot,\,\tilde v_j^+(\cdot,\J_j))\tilde
v_j^+(\cdot,\J_j)\sqrt{m^+(\J_j)}\Big)\,dF(\J).
\end{eqnarray*}
The last equality and the inclusions (\ref{mmbshipvpl}) and (\ref{mmbshipHess}) imply that
\begin{equation}\label{limnrm}
\lim_{\epsilon\downarrow
0}\frac{\Vert\Theta_{N,\epsilon}\Vert}{F(\epsilon)}=0.
\end{equation}
In view of (\ref{dftlGN}) and  Lemma \ref{lmGWnondeg}, the operator
$\tilde G_N$ has exactly $N$ positive eigenvalues $\nu_1\ge
\nu_2\ge\dots\ge \nu_N>0$ (counting their multiplicities). Then in
the same manner as in the proof of Proposition \ref{prpertcompoper},
we obtain using (\ref{limnrm}) and a modification of the comparison
theorem on the base of the minimax principle (Lemma 3.4 of
\cite{Ar-Zl1}) that for a small enough $\epsilon>0$ the operator
$F(\epsilon)\tilde G_N+\Theta_{N,\epsilon}$ has at least $N$
positive eigenvalues $\nu_1(\epsilon)\ge \nu_2(\epsilon)\ge\dots\ge
\nu_N(\epsilon)>0$ (counting their multiplicities). Using again the
minimax principle, we obtain from (\ref{lowestGWpl}) that the
operator $G_W^+$ has at least $N$ positive eigenvalues (counting
their multiplicities). Since $N$ is arbitrary, claim (ii) is proven.
\end{proof}

\subsubsection{Proof of Theorem \ref{thdegedg}}
\label{subsubsec:proofthdegedg}

\begin{proof}
Claim (i) is proved in the same manner as claim (i) of Theorem
\ref{thnondegedg}.

We now turn to the proof of the claim (ii)  in the case where
$d-d_+=2$. By claim (ii) of Proposition \ref{rprBirSchwdeg}, for the
Birman-Schwinger operator $X_W(\lambda)$ representation
(\ref{rprXwd1deg}) is valid for
$\lambda\in((\lambda_++\lambda_-)/2,\;\lambda_+)$, in which
$\Phi_+(\lambda)$ has the form (\ref{dfPhipld2deg}) and
$\Theta_+(\lambda)$ satisfies the condition (\ref{bndthtpldeg}). Let
$\mu_1^+(\lambda)\ge
\mu_2^+(\lambda)\ge\dots\ge\mu_n^+(\lambda)\ge\dots\;(\lambda\in(\lambda_-,\lambda_+))$
be the positive characteristic branches of the operator $H_0$ w.r.
to $W$ on the gap $(\lambda_-,\,\lambda_+)$, that is they are
positive eigenvalues of the operator $X_W(\lambda)$ (Definition
\ref{gde1}). By Lemma \ref{lmspecGWpl}, the operator $G_W^+$ has an
infinite number of positive eigenvalues
$\nu_1^+\ge\nu_2^+\ge\dots\ge\nu_n^+\ge\dots$. Let $\Phi(t)$ be the
operator function defined by (\ref{dfPhit}). Then in view of the
representation (\ref{rprXwd1deg}), definition (\ref{dfPhipld2deg})
and the condition (\ref{bndthtpldeg}), for some $\delta>0$ and any
$t\in[0,\delta]$ representation (\ref{rprPhit}) is valid with
$\Phi(0)=\frac{1}{2\pi}G_W$ and $\theta(t)$ satisfying the condition
(\ref{cndtht}). Then claim (ii) of Proposition \ref{prpertcompoper}
implies claim (ii) of Theorem \ref{thdegedg} for $d-d_+=2$ in the
same manner as the corresponding claim was obtained in the
non-degenerate case (see proof of claim (ii-b) of Theorem
\ref{thnondegedg}). Formula (\ref{asymptLbThrd2deg}) follows from
(\ref{asympteigvd2deg}), the formula
$\sum_{n=1}^\infty\nu_n^+=\rm{tr}(G_w^+)$ and claim (i) of Lemma
\ref{lmspecGWpl}. In the analogous manner claim (ii) for $d-d_+=1$
is proved by the use of representation (\ref{rprXwd1deg}) and
definition (\ref{dfPhipld1deg}). Claim (iii) follows from claim (ii)
of Proposition \ref{prpertcompoper} in the same manner as in the
proof of Theorem \ref{thnondegedg}) claim (iii-b) follows from claim
(i-b) of Proposition \ref{prpertcompoper}.

Claim (iv) is proved in the same manner as the corresponding claim
of Theorem \ref{thnondegedg} using claim (iii) of Proposition
\ref{rprBirSchwdeg}.

Claim (v) is proved in the same manner as the previous ones. Claim
(vi) follows from claims (iv) and (v).

Theorem \ref{thdegedg} is proven.
\end{proof}

\subsection{Estimate for the multiplicity of virtual eigenvalues in the case of an indefinite perturbation}
\label{subsec:estmult}

\subsubsection{General results on the multiplicity of virtual eigenvalues}
\label{subsubsec:genresmult}

Before proving Theorem \ref{thestmult}, we shall prove some auxiliary claims. If $A$ is a linear
operator acting in a Hilbert space, we denote $\Re(A):=\frac{1}{2}(A+A^\star)$;
if $A$ is self-adjoint, we denote $A_+:=\frac{1}{2}(|A|+A)$, $A_-:=\frac{1}{2}(|A|-A)$.

\begin{lemma}\label{lmrprBrSchw}
Let $(\lambda_-,\lambda_+)$ be a gap of the spectrum of a
self-adjoint operator $H_0$ acting in a Hilbert space $\B$ be and
$W$ be a bounded self-adjoint operator in $\B$. Then for any
$\lambda\in(\lambda_-,\lambda_+)$ the real part of the
Birman-Schwinger operator
$X_W(\lambda)=W^{\frac{1}{2}}R_\lambda(H_0)|W|^{\frac{1}{2}}$ admits
the representation:
\begin{equation}\label{rprBrSchw}
\Re(X_W(\lambda))=X_{W_+}(\lambda)-X_{W_-}(\lambda).
\end{equation}
Recall that $W^{\frac{1}{2}}$ is defined by (\ref{dfsqrtW}).
\end{lemma}
\begin{proof}
We have:
\begin{eqnarray}\label{calcul}
&&\hskip-8mm(W^{\frac{1}{2}})_+R_\lambda(H_0)(W^{\frac{1}{2}})_+-(W^{\frac{1}{2}})_-R_\lambda(H_0)(W^{\frac{1}{2}})_-=
\nonumber\\
&&\hskip-8mm\frac{1}{4}\left((|W|^{\frac{1}{2}}+W^{\frac{1}{2}})R_\lambda(H_0)(|W|^{\frac{1}{2}}+W^{\frac{1}{2}})\right.-
\nonumber\\
&&\hskip-8mm\left.(|W|^{\frac{1}{2}}-W^{\frac{1}{2}})R_\lambda(H_0)(|W|^{\frac{1}{2}}-W^{\frac{1}{2}})\right)=
\frac{1}{2}\left(W^{\frac{1}{2}}R_\lambda(H_0)|W|^{\frac{1}{2}}-\right.
\nonumber\\
&&\hskip-8mm\left.|W|^{\frac{1}{2}}R_\lambda(H_0)W^{\frac{1}{2}}\right)=
\frac{1}{2}\left(W^{\frac{1}{2}}R_\lambda(H_0)|W|^{\frac{1}{2}}+(W^{\frac{1}{2}}R_\lambda(H_0)|W|^{\frac{1}{2}})^\star\right).
\end{eqnarray}
Let $E_\lambda\,(\lambda\in\R)$ be the spectral resolution of identity, corresponding to the operator $W$. Since
$|W|^{\frac{1}{2}}=\int_{\R}|\lambda|^{\frac{1}{2}}\,dE_\lambda$ and, in view of (\ref{dfsqrtW}),
\begin{equation*}
W^{\frac{1}{2}}=\int_{\R}\rm{sign}(\lambda)|\lambda|^{\frac{1}{2}}\,dE_\lambda,
\end{equation*}
the equalities are valid:
$(W^{\frac{1}{2}})_+=(W_+)^{\frac{1}{2}}$ and $(W^{\frac{1}{2}})_-=(W_-)^{\frac{1}{2}}$. Then (\ref{calcul}) can be written
in the form (\ref{rprBrSchw}).
\end{proof}

\begin{proposition}\label{prestmult}
Assume that all the conditions of Lemma \ref{lmrprBrSchw} are
satisfied and furthermore assume that for some
$\lambda_0\in\Rs(H_0)$ the operator $R_{\lambda_0}|W|^{\frac{1}{2}}$
is compact. Let us take $\delta\in(0,\,\lambda_+-\lambda_-)$.
Then\vskip2mm

(i) if $M(\lambda_+,H_0,W_+)<\infty$, there exists $\bar\gamma>0$
such that for any $\gamma\in(-\bar\gamma,0)$ the multiplicity of
each eigenvalue of the operator $H_\gamma$ lying in
$(\lambda_+-\delta,\;\lambda_+)$ is at most $M(\lambda_+,H_0,W_+)$;
\vskip2mm

(ii) the claim analogous to (i) is valid for $\gamma>0$ with $W_-$
and $\gamma\in(0,\bar\gamma)$ instead of, respectively, $W_+$ and
$\gamma\in(-\bar\gamma,0)$;\vskip2mm

(iii) the claims analogous to (i), (ii) are valid near the edge
$\lambda_-$ with $\lambda_-$, $(\lambda_-,\lambda_-+\delta)$, $W_-$,
$W_+$, $\gamma\in(0,\bar\gamma)$ and $\gamma\in(-\bar\gamma,0)$
instead of, respectively, $\lambda_+$,
$(\lambda_+-\delta,\;\lambda_+)$, $W_+$, $W_-$,
$\gamma\in(-\bar\gamma,0)$ and $\gamma\in(0,\bar\gamma)$.
\end{proposition}

\begin{proof}
We shall prove only claim (i), because (ii) and (iii) are proved
analogously. Denote $l_+=M(\lambda_+,H_0,W_+)$. By Proposition
\ref{BrScwspct}, we should prove that
\begin{equation}\label{shldprv}
\forall\,\lambda\in(\lambda_+-\delta,\;\lambda_+):\quad
\dim(\ker(I+\gamma X_W(\lambda)))\le l_+.
\end{equation}
In view of Lemma \ref{lmrprBrSchw}, we have for $u\in\B$:
\begin{equation}\label{rprrequadfrm}
\Re(((I+\gamma X_W(\lambda))u,u))=(u,u)+\gamma((X_{W_+}(\lambda)u,u)-(X_{W_-}(\lambda)u,u)).
\end{equation}
Let $\mu_k^+(\lambda)$ and $\mu_k^-(\lambda)$ be positive and
negative characteristic branches of the operator $H_0$ w.r.t. $W_+$,
that is they are continuous branches of eigenvalues of the operator
$X_{W_+}(\lambda)$ such that $\mu_k^+(\lambda)$ are numbered in the
non-increasing ordering
$\mu_1^+(\lambda)\ge\mu_2^+(\lambda)\ge\dots\mu_k^+(\lambda)\ge\dots>0$
and $\mu_k^-(\lambda)$ are numbered in the non-increasing ordering
$\mu_1^-(\lambda)\le\mu_2^-(\lambda)\le\dots\mu_k^-(\lambda)\le\dots<0$,
$\Dm(\mu_k^+)=(\eta_k^+,\lambda_+),\;\eta_k^+\in[\lambda_-,\lambda_+]$,
$\Dm(\mu_k^-)=(\lambda_-,\eta_k^-),\;\eta_k^-\in[\lambda_-,\lambda_+]$'
and they are increasing functions (Definition \ref{gde1}, Remark
\ref{rembranch}). Let $\{\mu_k^+(\lambda)\}_{k=1}^{l_+}$ be the main
characteristic branches of $H_0$ w.r.t. $W_+$ near the edge
$\lambda_+$, that is
\begin{eqnarray}\label{mainbrnchWpl}
\lim_{\lambda\uparrow\lambda_+}\mu_k^+(\lambda)\left\{
\begin{array}{ll}
=\infty &\rm{for}\;1\le k\le l_+\\
<\infty &\rm{for}\; k>l_+.
\end{array}\right.
\end{eqnarray}
As it is clear, $\eta_1^+\le\eta_2^+\le\dots\le \eta_{l_+}^+$. For
each $\lambda\in(\eta_+,\lambda_+)$ and $j\in\{1,2,\dots,l_+\}$
consider the eigenspace $E_j^+(\lambda)$ of the operator
$X_{W_+}(\lambda)$, corresponding to its eigenvalue
$\mu_1^+(\lambda)$. Consider the following family of subspaces:
\begin{equation}\label{dffamsbsp}
\Fc(\lambda)=\left\{\begin{array}{ll}
\B\ominus\left(\bigoplus_{j=1}^{l_+}E_j^+(\lambda)\right) &\rm{for}\;
\lambda\in(\lambda_+-\delta,\;\lambda_+)\cap(\eta_{l_+}^+,\lambda_+),\\
\B &\rm{for}\;
\lambda\in(\lambda_+-\delta,\;\lambda_+)\setminus(\eta_{l_+}^+,\lambda_+).
\end{array}\right.
\end{equation}
Then we obtain:
\begin{eqnarray*}
&&\hskip-7mm\forall\;u\in\Fc(\lambda):\\
&&\hskip-7mm(X_{W_+}(\lambda)u,\,u)\le\left\{\begin{array}{ll}
\max_{j>l+}\mu_j^+(\lambda)\|u\|^2, &\rm{if}\;
\lambda\in(\lambda_+-\delta,\;\lambda_+)\cap(\eta_{l_+}^+,\lambda_+),\\
\mu_1^+(\eta_{l_+}^+)\|u\|^2, &\rm{if}\;
\lambda\in(\lambda_+-\delta,\;\lambda_+)\setminus(\eta_{l_+}^+,\lambda_+).
\end{array}\right.
\end{eqnarray*}
Hence, in view of (\ref{mainbrnchWpl}), we have:
\begin{equation}\label{estXWpl}
\sup_{\lambda\in(\lambda_+-\delta,\;\lambda_+),\;u\in\Fc(\lambda)\setminus\{0\}}
\frac{(X_{W_+}(\lambda)u,\,u)}{\|u\|^2}<\infty.
\end{equation}
Since the operator function $X_{W_-}(\lambda)$ increases in
$(\lambda_-,\lambda_+)$ in the sense of comparison of quadratic
forms (\cite{Ar-Zl1}, Lemma 3.5), we have:
\begin{eqnarray}\label{estXWmn}
&&
\inf_{\lambda\in(\lambda_+-\delta,\;\lambda_+),\;u\in\B\setminus\{0\}}
\frac{(X_{W_-}(\lambda)u,\,u)}{\|u\|^2}\ge
\nonumber\\
&& \inf_{u\in\B\setminus\{0\}}
\frac{\left(X_{W_-}(\lambda_+-\delta)u,\,u\right)}{\|u\|^2}>-\infty.
\end{eqnarray}
Then we conclude from (\ref{rprrequadfrm}), (\ref{estXWpl}) and (\ref{estXWmn}) that there
exists $\bar\gamma>0$ such that
\begin{equation*}
\forall\;\gamma\in(-\bar\gamma,0):\quad\inf_{\lambda\in(\lambda_+-\delta,\;\lambda_+),\;u\in\Fc(\lambda)\setminus\{0\}}
\frac{\Re(((I+\gamma X_W(\lambda))u,u))}{\|u\|^2}\ge\frac{1}{2}.
\end{equation*}
Taking into account that $\Re(((I+\gamma X_W(\lambda))u,u))\le\|(I+\gamma X_W(\lambda))u\|\|u\|$, we obtain
from the last estimate that
\begin{equation}\label{estinfqdfrm}
\forall\;\gamma\in(-\bar\gamma,0):\quad\inf_{\lambda\in(\lambda_+-\delta,\;\lambda_+),\;u\in\Fc(\lambda)\setminus\{0\}}
\frac{\|(I+\gamma X_W(\lambda))u\|}{\|u\|}\ge\frac{1}{2}.
\end{equation}
Let us prove estimate (\ref{shldprv}). Assume, on the contrary, that
\begin{eqnarray}\label{asscntr}
&&\exists\,\gamma_0\in(-\bar\gamma,0),\;\exists\,\lambda_0\in(\lambda_+-\delta,\;\lambda_+):\nonumber\\
&&\dim(\ker(I+\gamma_0 X_W(\lambda_0)))>l_+.
\end{eqnarray}
On the other hand, in view of (\ref{dffamsbsp}),
\begin{equation*}
\forall\,\lambda\in(\lambda_+-\delta,\;\lambda_+):\quad
\rm{codim}(\Fc(\lambda))\le l_+.
\end{equation*}
This inequality together with (\ref{asscntr}) imply that
\begin{equation*}
\Fc(\lambda_0)\cap\ker(I+\gamma_0 X_W(\lambda_0))\neq\{0\}.
\end{equation*}
The last fact contradicts the inequality (\ref{estinfqdfrm}). So, estimate (\ref{shldprv}) is proven.
\end{proof}

\subsubsection{Proof of Theorem \ref{thestmult}}
\label{subsec:proofthmult}

\begin{proof}
Assume that the edge $\lambda_+$ is non-degenerate.  If in the in
the proof of Theorem \ref{thnondegedg} we take $W_+$ instead of $W$,
we obtain that for $d=1$ $M(\lambda_+,H_0,W_+)=1$, for $d=2$
$M(\lambda_+,H_0,W_+)=n_+$ and for $d\ge 3$
$M(\lambda_+,H_0,W_+)=0$. Then by Proposition \ref{prestmult} all
the claims (i)-(iii) are valid in the case where $\lambda_+$ is
non-degenerate. The case where $\lambda_-$ is non-degenerate is
treated analogously.

Claim (iv) is proved in the same manner as the previous ones. Claim
(v) follows from the claims (iii) and (iv).

Theorem \ref{thestmult} is proven.
\end{proof}

\subsubsection{Proof of Theorem \ref{ththresholddeg}}
\label{subsec:proofthmultdeg}
\begin{proof}
The proof is the same as the proof of Theorem \ref{thestmult} in the
case where $d\ge 3$ and it is supported by Proposition
\ref{prestmult} and claim (iv) of Theorem \ref{thdegedg}.
\end{proof}

\section{Representation of the resolvent of the unperturbed operator $H_0$ near the edges of a gap of its spectrum}
\label{sec:resunprop1}\setcounter{equation}{0}

\subsection{Main claims} \label{sec:frmmaincl}

\begin{proposition}\label{prrprres1}
Let $(\lambda_-,\,\lambda_+)$ be a gap of the spectrum of the
unperturbed operator $H_0$ such that its edge $\lambda_+$ satisfies
condition (B) of Section \ref{subsec:spectchar}, and $\lambda^+(\J)$
be the dispersion function, branching from the edge $\lambda_+$. Let
us take $\delta\in(0,\,(\lambda_+-\lambda_-))$.\vskip2mm

(i) For $\lambda\in(\lambda_-,\lambda_+)$ the resolvent
$R_\lambda(H_0)$ has the form:
\begin{equation}\label{represol1}
R_\lambda(H_0)=K(\lambda)+\Theta(\lambda),
\end{equation}
where $\Theta(\lambda)$ is a bounded operator in $L_2(\R^d)$ for any
$\lambda\in(\lambda_-,\lambda_+)$ such that the condition is
satisfied
\begin{eqnarray}\label{bndtht2}
&&\sup_{\lambda\in(\lambda_+-\delta,\;\lambda_+)}\|\Theta(\lambda)\|<\infty,
\end{eqnarray}
$K(\lambda)=\sum_{k=1}^{n_+}K_k^+(\lambda)$ and for any
$k\in\{1,2,\dots,n_+\}$ the restriction
$K_k^+(\lambda)|_{L_{2,0}(\R^d)}$ of $K_k^+(\lambda)$ on
$L_{2,0}(\R^d)$ is an operator with a continuous in
$\R^d\times\R^d\times(\lambda_-,\lambda_+)$ integral kernel
$K_k^+(\e,\bs,\lambda)$;\vskip2mm

(ii) if $d-d_k^+\ge 3$, then for any $k\in\{1,2,\dots,n_+\}$ the
function $K_k^+(\e,\bs,\lambda)$ satisfies the condition
\begin{eqnarray}\label{cndK1}
\sup_{\lambda\in((\lambda_+-\delta,\;\lambda_+),\;\e,\bs\in\R^d}|K_k^+(\e,\bs,\lambda)|<\infty;
\end{eqnarray}
\vskip2mm

(iii) If $d-d_k^+\le 2$, then for any $k\in\{1,2,\dots,n_+\}$,
$\lambda\in(\lambda_-,\,\lambda_+)$ the function
$K_k^+(\e,\bs,\lambda)$ admits the representation
$K_k^+(\e,\bs,\lambda)=G_k^+(\e,\bs,\lambda)+\tilde
K_k^+(\e,\bs,\lambda)$, where $\tilde K_k^+(\e,\bs,\lambda)$ is
continuous in $\R^d\times\R^d\times(\lambda_-,\,\lambda_+)$ and
\vskip1mm

\indent\indent(a) for $d-d_k^+=1$
\begin{eqnarray}\label{dfFpl1}
G_k^+(\e,\bs,\lambda)=\frac{\sqrt{2}\pi}{(2\pi)^d\sqrt{\lambda_+-\lambda}}
\int_{F_k^+}\Qc^+(\e,\bs,\J)\sqrt{m^+(\J)}\,dF(\J)
\end{eqnarray}
where $\Qc^+(\e,\bs,\J)$ is the eigenkernel of $H(\J)$ corresponding
to the eigenvalue $\lambda^+(\J)$, defined in Section
\ref{subsec:spectchar}, $m^+(\J)$ is defined by (\ref{dfmplp}),
(\ref{dfnrmHes}), (\ref{dfHes}), and
\begin{eqnarray*}
\sup_{\lambda\in(\lambda_+-\delta,\;\lambda_+),\;\e,\bs\in\R^d}
\frac{|\tilde K_k^+(\e,\bs,\lambda)|}{(1+|\e-\bs|)}<\infty,
\end{eqnarray*}
\vskip1mm

\indent\indent(b) for  $d-d_k^+=2$
\begin{eqnarray}\label{dfFpl2}
&&\hskip-10mm
G_k^+(\e,\bs,\lambda)=\nonumber\\
&&\hskip-10mm\frac{1}{(2\pi)^{d-1}}\ln\Big(\frac{1}{\lambda_+-\lambda}\Big)
\int_{F_k^+}\Qc^+(\e,\bs,\J)\sqrt{m^+(\J)}\,dF(\J),
\end{eqnarray}
and
\begin{eqnarray*}
&&\sup_{\lambda\in(\lambda_+-\delta,\;\lambda_+),\;\e,\bs\in\R^d}
\frac{|\tilde K_k^+(\e,\bs,\lambda)|}{1+\ln(1+|\e-\bs|)}<\infty.
\end{eqnarray*}
\vskip2mm

The claims analogous to above ones are valid with $\lambda_-$,
$\lambda^-(\J)$, $n_-$, $d_k^-$, $F_k^-$, $\Qc^-(\e,\bs,\J)$,
$m^-(\J)$ and $(\lambda_-,\;\lambda_-+\delta)$ instead of,
respectively, $\lambda_+$, $\lambda^+(\J)$, $n_+$, $d_k^+$, $F_k^+$,
$\Qc^+(\e,\bs,\J)$, $m^+(\J)$ and $(\lambda_+-\delta,\;\lambda_+)$.
\end{proposition}

In particular, in the case of non-degenerate edges of the gap of
$\sigma(H_0)$ the previous proposition acquires the form:

\begin{proposition}\label{prrprres}
Let $(\lambda_-,\,\lambda_+)$ is a gap of the spectrum of the
unperturbed operator $H_0$ such that its edge $\lambda_+$ is
non-degenerate and $\lambda^+(\J)$ be the dispersion function,
branching from the edge $\lambda_+$.  Let us take
$\delta\in(0,\,(\lambda_+-\lambda_-))$.\vskip2mm

(i) The resolvent $R_\lambda(H_0)$ has the form for
$\lambda\in(\lambda_-,\lambda_+)$:
$R_\lambda(H_0)=K(\lambda)+\Theta(\lambda)$, where $\Theta(\lambda)$
is a bounded operator in $L_2(\R^d)$ for any
$\lambda\in(\lambda_-,\lambda_+)$, such that condition
(\ref{bndtht2}) is satisfied and the restriction
$K(\lambda)|_{L_{2,0}(\R^d)}$ of $K(\lambda)$ on $L_{2,0}(\R^d)$ is
an operator with a continuous in
$\R^d\times\R^d\times(\lambda_-,\lambda_+)$ integral kernel
$K(\e,\bs,\lambda)$;  \vskip2mm

(ii) if $d\ge 3$, then the function $K(\e,\bs,\lambda)$ satisfies
the condition
\begin{eqnarray*}
\label{cndK}
\sup_{\lambda\in(\lambda_+-\delta,\;\lambda_+),\;\e,\bs\in\R^d}|K(\e,\bs,\lambda)|<\infty;
\end{eqnarray*}
\vskip2mm

(iii) if $d\le2$ and $\lambda\in(\lambda_-,\,\lambda_+)$, then the
function $K(\e,\bs,\lambda)$ admits the representation
$K(\e,\bs,\lambda)=F^+(\e,\bs,\lambda)+\tilde K^+(\e,\bs,\lambda)$,
where $\tilde K^+(\e,\bs,\lambda)$ is continuous in
$\R^d\times\R^d\times(\lambda_-,\,\lambda_+)$ and \vskip2mm

\indent\indent(a) for $d=1$
\begin{eqnarray*}
F^+(x,s,\lambda)=\sqrt{m_1^+}\;\frac{b^+(x,p_1)\overline{b^+(s,p_1)}}{\sqrt{2(\lambda_+-\lambda)}},
\end{eqnarray*}
where $m_1^+$ is defined by (\ref{dfmupl}) and
\begin{eqnarray}\label{cndK1pl}
\sup_{\lambda\in(\lambda_+-\delta,\;\lambda_+),\;x,s\in\R}
\frac{|\tilde K^+(x,s,\lambda)|}{1+|x-s|}<\infty;
\end{eqnarray}
\vskip2mm

\indent\indent(b) for  $d=2$
\begin{eqnarray*}
F^+(\e,\bs,\lambda)=\frac{1}{2\pi}\ln\left(\frac{1}{\lambda_+-\lambda}\right)
\sum_{k=1}^{n_+}\sqrt{m_k^+}\,b_k^+(\e)\overline{b_k^+(\bs)},
\end{eqnarray*}
where $m_k^+$ is defined by (\ref{dfmkpl}), (\ref{Hess}),
$b_k^+(\e)$ is the Bloch function, corresponding to the dispersion
function $\lambda^+(\J)$ and the quasi-momentum $\J=\J_k^+$, i.e.
$b_k^+(\e)=b^+(\J_k^+)$; furthermore
\begin{eqnarray*}
\sup_{\lambda\in(\lambda_+-\delta,\;\lambda_+),\;\e,\bs\in\R^2}
\frac{|\tilde K^+(\e,\bs,\lambda)|}{1+\ln(1+|\e-\bs|)}<\infty.
\end{eqnarray*}
\vskip2mm

 The claims analogous to above ones are valid with
$\lambda_-$, $\lambda^-(\J)$, $n_-$, $\J_k^-$, $b_k^-(\e)$, $m_k^-$
and $(\lambda_-,\;\lambda_-+\delta)$ instead of, respectively,
$\lambda_+$, $\lambda^+(\J)$, $n_+$, $\J_k^+$, $b_k^+(\e)$, $m_k^+$
and $(\lambda_+-\delta,\;\lambda_+)$.
\end{proposition}

\subsection{Auxiliary claims}\label{sec:auxclaims}

For the proof of the propositions, formulated above, we need some
auxiliary claims.

\subsubsection{Estimates of some integrals} \label{sec:estintegr}

\begin{lemma}\label{lmestint}
If $\rho>0$, $\alpha>1$, $\mu>0$, $\beta\in\R$, $n\in\{1,\,2\}$ and
$\phi:\;[0,\infty)\rightarrow\R$ is a continuous function such that
$\bar\phi=\sup_{u\in[0,\rho]}\frac{\phi(u)}{u^2}<\infty$, then the
following estimates are valid:
\begin{eqnarray}\label{estint1}
&&I_1=\int_0^{\rho\alpha}\frac{u^{n-1}|\exp(i\phi(u))-1|}{u^2+\mu}\,du\le\\
&&\left\{\begin{array}{ll}
\bar\phi\rho+2\rho^{-1}(1-\alpha^{-1})&\rm{for}\quad n=1,\\
\bar\phi\rho^2/2+2\ln\alpha &\rm{for}\quad n=2;
\end{array}\right.\nonumber
\end{eqnarray}
\begin{eqnarray}\label{estint2}
&&I_2=\int_0^{\rho\alpha}\frac{u^{n-1}(1-\cos(\beta
u))}{u^2+\mu}\,du\le\\
&&\left\{\begin{array}{ll}
\beta^2\rho/2+2\rho^{-1}(1-\alpha^{-1})&\rm{for}\quad n=1,\\
\beta^2\rho^2/4+2\ln\alpha &\rm{for}\quad n=2;
\end{array}\right.\nonumber
\end{eqnarray}
\end{lemma}
\begin{eqnarray}\label{estint3}
I_3=\Big|\int_0^{\rho\alpha}\frac{u^n\sin(\beta
u)}{u^2+\mu}\,du\Big|\le\left\{\begin{array}{ll}
|\beta|\rho+\ln\alpha&\rm{for}\quad n=1,\\
|\beta|\rho^2/2+\rho(\alpha-1)&\rm{for}\quad n=2.
\end{array}\right.
\end{eqnarray}
\begin{proof}
We have for $u\in[0,\rho]$:
\begin{eqnarray*}
|\exp(i\phi(u))-1|=\Big|\int_0^1\partial_\tau\exp(i\tau\phi(u))\,d\tau\Big|=
\Big|\int_0^1\exp(i\tau\phi(u))i\phi(u)\,d\tau\Big|\le\bar\phi u^2,
\end{eqnarray*}
hence
\begin{eqnarray*}
&&I_1=\int_0^\rho\frac{u^{n-1}|\exp(i\phi(u))-1|}{u^2+\mu}\,du+
\int_\rho^{\rho\alpha}\frac{u^{n-1}|\exp(i\phi(u))-1|}{u^2+\mu}\,du\le\\
&&\bar\phi\int_0^\rho u^{n-1}\,du+2\int_\rho^\rho\alpha u^{n-3}\,du.
\end{eqnarray*}
Therefore, (\ref{estint1}) is valid. Estimates (\ref{estint2}) and
(\ref{estint3}) are proved analogously.
\end{proof}

\subsubsection{Some geometric claims} \label{sec:geomclaims}

Assume that the edge $\lambda_+\,(\lambda_-)$ of the gap
$(\lambda_-,\,\lambda_+)$ of the spectrum of the unperturbed
operator $H_0$ satisfies condition (B) of Section
\ref{subsec:spectchar}. Taking into account that $F^+\,(F^-)$ is the
set of all minimum (maximum) points of the dispersion function
$\lambda^+\,(\lambda^-)$ and using the Morse-Bott lemma
(\cite{Ban-Hur}), we obtain that for any $k\in\{1,2,\dots
n_+\}\;(k\in\{1,2,\dots n_-\})$ and $\J_\star\in F^+_k\,(\J_\star\in
F^-_k)$ on the torus $\T^d$ there exists a smooth chart
$(U,\phi)\,(\J_\star\in U)$ such that
$\phi:\,U\rightarrow\R^{d_k^+}\times\R^{d-d_k^+}\;\big(\phi:\,U\rightarrow\R^{d_k^-}\times\R^{d-d_k^-}\big)$
and

(a) $\phi(\J_\star)=0$;

(b) $\phi(U\cap F^+_k)=\{(x,y)\in\R^{d_k^+}\times\R^{d-d_k^+}\;\vert\;y=0\}$
$(\phi(U\cap F^-_k)=\{(x,y)\in\R^{d_k^-}\times\R^{d-d_k^-}\;\vert\;y=0\})$;

(c)
$(\lambda^+\circ\phi^{-1})(x,y)=\sum_{l=1}^{d-d_k^+}y_l^2+\lambda_+$
$\big((\lambda^-\circ\phi^{-1})(x,y)=-\sum_{l=1}^{d-d_k^-}y_l^2+\lambda_-\big)$.

The chart $(U,\phi)$ is called the {\it reducing chart} at the point
$\J_\star\in F^+_k\,(\J_\star\in F^-_k)$.

\begin{lemma}\label{lmredchrt}
Let $(U,\phi)$ be a reducing chart at a point $\J_\star\in
F^+_k\,(\J_\star\in F^-_k)$ defined above and $\Phi=P\circ\phi$,
where
$P:\;\R^d\rightarrow\R^{d-d_k^+}\;\big(P:\;\R^d\rightarrow\R^{d-d_k^-}\big)$
is the orthogonal projection on the subspace
$\R^{d-d_k^+}\;(\R^{d-d_k^-})$ of $\R^d$. Then for any $\J\in
F^+_k\cap U\,(\J\in F^-_k\cap U)$ the normal Jacobian of $\Phi$
$\;NJ\,\Phi(\J)=\det\big(\Phi^\prime(\J)\vert_{N_\J}\big)$ is equal
to
$\sqrt{\frac{NJ\,\rm{Hes}_\J(\lambda^+)}{2^{d-d^+_k}}}\;\big(\sqrt{\frac{NJ\,\rm{Hes}_\J(\lambda^-)}{2^{d-d^-_k}}}\big)$.
Hence, in particular, for any $\J\in F^+_k\cap U\,(\J\in F^-_k\cap
U)$
$\;\mathrm{Im}\big(\Phi^\prime(\J)\big)=\R^{d-d_k^+}\;\big(\mathrm{Im}\big(d\Phi(\J)\big)=\R^{d-d_k^-}\big)$
and the mapping
$\Phi^\prime(\J)\vert_{N_\J}:\,N_\J\rightarrow\R^{d-d_k^+}\;
\big(\Phi^\prime(\J)\vert_{N_\J}:\,N_\J\rightarrow\R^{d-d_k^-}\big)$
is a linear isomorphism. Recall that
$NJ\,\rm{Hes}_\J(\lambda^+)\;\big(NJ\,\rm{Hes}_\J(\lambda^-)\big)$
is the normal Hessian of $\lambda^+\;(\lambda^-)$ at the point $\J$,
defined by (\ref{dfnrmHes}), and
$\rm{Hes}_\J(\lambda^+)\;(\rm{Hes}_\J(\lambda^-))$ is the Hessian
operator of $\lambda^+\;(\lambda^-)$ at the point $\J$, defined by
(\ref{dfHes}).
\end{lemma}
\begin{proof}
We shall prove the claim only for the function $\lambda^+(\J)$,
because for $\lambda^-(\J)$ it is proved analogously. In view of
condition (c) for the reducing chart,, for $\J\in F^+_k\cap U$ and
$y=\phi(\J)$ $\lambda^+\circ\phi^{-1}(y)=Py\cdot y$. Then for
$v,w\in\R^d$
\begin{equation*}
d(\lambda^+\circ\phi^{-1}(y))[v]=d\lambda^+(\phi^{-1}(y))[d(\phi^{-1}(y))[v]]=2v\cdot
Py,
\end{equation*}
\begin{eqnarray*}
&&d^2(\lambda^+\circ\phi^{-1}(y))[v,w]=d^2\lambda^+(\phi^{-1}(y))[d(\phi^{-1}(y))[v],\,d(\phi^{-1}(y))[w]]+\\
&&d\lambda^+(\phi^{-1}(y))[d^2(\phi^{-1}(y))[v,w]]=2Pv\cdot w
\end{eqnarray*}
(we identify the tangent and the second tangent bundles over $\R^d$
and the flat torus $T^d$ with $\R^d\times\R^d$ and
$\R^d\times\R^d\times\R^d$ respectively). Since $\J=\phi^{-1}(y)$ is
a minimum point of the function $\lambda^+$,
$d\lambda^+(\phi^{-1}(y))=0$. Hence we get, in view of definition
(\ref{dfHes}) of the Hessian operator:
$\rm{Hes}_\J(\lambda^+)d(\phi^{-1}(y))[v]\cdot
d(\phi^{-1}(y))[w]=2Pv\cdot Pw$, that is
\begin{equation}\label{Heseq}
\forall\,s,t\in\R^d:\quad\rm{Hes}_\J(\lambda^+)s\cdot
t=2d\Phi(\J)[s]\cdot d\Phi(\J)[t].
\end{equation}
In view of condition (b) for the reducing chart, $\Phi(F^+_k\cap
U)=\{0\}$, hence for any $\J\in F^+_k\cap U$
$T_\J(F^+_k)\subseteq\ker(\Phi^\prime(\J))$. The last fact and
(\ref{Heseq}) imply that
$T_\J(F^+_k)\subseteq\ker\big(\rm{Hes}_\J(\lambda^+)\big)$. Hence
since the operator $\rm{Hes}_\J(\lambda^+)$ is self-adjoint, it maps
the subspace $N_\J^+$ into itself. Then by (\ref{Heseq})
$\rm{Hes}_\J(\lambda^+)\vert_{N_\J^+}=2\big(\Phi^\prime(\J)
\vert_{N_\J}\big)^\star\big(\Phi^\prime(\J)\vert_{N_\J}\big)$. The
last equality implies the desired claim.
\end{proof}

\begin{lemma}\label{lmdifeq}
Let $(U,\phi)$ be a reducing chart at a point $\J_\star\in
F^+_k\,(\J_\star\in F^-_k)$ defined above and $\Phi$ is the same
mapping as in Lemma \ref{lmredchrt}. Consider the differential
equation in $U$:
\begin{equation}\label{diffeq}
\frac{d\J}{dt}=\Big(\Phi^\prime(\J)\vert_{N_\J}\Big)^{-1}\by,\quad\by\in\R^{d-d_k^+}\;(\by\in\R^{d-d_k^-}).
\end{equation}
Let $\J(t,\J_0,\by)$ be the flow of this equation, that is the
solution of it satisfying the initial condition
$\J(0,\J_0,\by)=\J_0$. Then it is possible to restrict the
neighborhood $U$ of $\J_\star$ such that\vskip2mm

(i) for some $r>0$, any $t\in[0,r]$ and any $\by$ belonging to the
unit sphere
$S^{d-d_k^+-1}\subset\R^{d-d_k^+}\;\big(S^{d-d_k^--1}\subset\R^{d-d_k^-}\big)$
the set $F_{t\by}=\Phi^{-1}(t\by)$ is a $d_k^+\;(d_k^-)$-dimensional
smooth submanifold of $\T^d$ (in particular, $F_0=F^+_k\cap
U\;\big(F_0=F^-_k\cap U\big)$), $\J(t,F^+_k\cap U,\by)=F_{t\by}$
$\big(\J(t,F^+_k\cap U,\by)=F_{t\by}\big)$ and the mapping
$\J(t,\cdot,\by):\;F^+_k\cap U\rightarrow
F_{ty}\;\big(\J(t,\cdot,\by):\;F^-_k\cap U\rightarrow F_{t\by}\big)$
is a diffeomorphism;\vskip2mm

(ii) for $\bq\in F^+_k\cap U\;(\bq\in F^-_k\cap U)$ the functions
$\partial_t\J(t,\bq,\by)\vert_{t=0}$ and
$\partial_tTJ\J(t,\bq,\by)\vert_{t=0}$ are odd w.r.t. $\by$. Here
$TJ\J(t,\bq,\by)$ is the tangential Jacobian of the mapping
$\J(t,\cdot,\by)$, that is
\begin{eqnarray*}
&&TJ\J(t,\bq,\by)=\det\big(\partial_{\bq}\J(t,\bq,\by)\vert_{T_{\bq}(F^+_k)}\big)\\
&&\Big(TJ\J(t,\bq,\by)=\det\big(\partial_{\bq}\J(t,\bq,\by)\vert_{T_{\bq}(F^-_k)}\big)\Big).
\end{eqnarray*}
\end{lemma}
\begin{proof}
Consider only the case where $\J_\star\in F^+_k$, because the case
$\J_\star\in F^-_k$ is treated analogously. Using the fact that
$\Phi(F_k^+\cap U)=\{0\}$, we see that if $\bq\in F_k^+\cap U_{k,l}$
and for some $\by\in S^{d-d_k^+-1}$, $t>0$ the solution
$\J(\cdot,\bq,\by)$ of (\ref{diffeq}) exists in the interval
$[0,t]$, then
$\Phi^\prime(\J(\tau,\bq,\by))\partial_\tau\J(\tau,\bq,\by)=\by\;(\tau\in[0,t])$.
Integrating the both sides of the last equality by $\tau$ over
$[0,t]$, we get: $\Phi(\J(t,\bq,\by))=t\by$, that is
$\J(t,\bq,\by)\in\Phi^{-1}(t\by)$. Using the theorem on the local
existence and uniqueness of solution of the Cauchy problem for a
dynamical system in a Banach space, we can restrict the neighborhood
$U$ of $\J_\star$ such that claim (i) is valid.

Let us prove claim (ii). The fact that the function
$\partial_t\J(t,\bq,\by)\vert_{t=0}$ is odd w.r. to $\by$ for
$\bq\in F_k^+\cap U$ follows immediately from (\ref{diffeq}) and the
equality $\J(0,\bq,\by)=\bq$.  As it is known, the derivative
$Y(t)=\partial_{\bq}\J(t,\bq,\by)$ satisfies the following linear
equation, which is the linearization of equation (\ref{diffeq}) at
its solution $\J(t)=\J(t,\bq,\by)$:
\begin{equation}\label{lineq}
\frac{dY}{dt}=(A(t)Y)\by,
\end{equation}
where $A(t)=d_\J\big(\Phi^\prime(\J(t))\vert_{N_{\J(t)}}\big)^{-1}$.
Furthermore, $Y(0)=I$. In view of claim (i), for any $\bq\in F_k^+$
and $t\in[o,r]$ the operator $Y(t)\vert_{T_{\bq}(F_k^+)}$ realizes a
linear isomorphism between $T_{\bq}(F_k^+)$ and
$T_{\J(t)}(F_{t\by})$. Let $dv=dp_1\wedge dp_2\wedge\dots\wedge
dp_d$ be the volume form on the flat torus $\T^d$. As it is known,
the volume form $dF_{t\by}$ on the submanifold $F_{t\by}$ has the
form $dF_{t\by}(\J)=dv\vert_{T_{\J}(F_{t\by})}\;(\J\in F_{t\by})$.
Let $d^\star F_{t\by}$ be the pullback of the form $dF_{t\by}$
w.r.t. the mapping $\J(t,\cdot,\by)\vert_{F_k^+\cap U}$. Taking into
account equation (\ref{lineq}) and the equalities $F_0=F_k^+\cap U$,
$Y(0)=I$ and $\J(0)=\J(0,\bq,\by)=\bq\;(\bq\in F_k^+\cap U)$, we
have for $\bs=(s_1,s_2,\dots,s_d)\in T_{\bq}(F_k^+)$:
\begin{eqnarray*}
&&\frac{d}{dt}d^\star F_{t\by}[\bs]\vert_{t=0}=\sum_{i=1}^d(Y(t)\bs)_1\wedge(Y(t)\bs)_2\wedge\dots\\
&&\wedge(Y(t)\bs)_{i-1}\wedge((A(t)Y(t)\bs)\by)_i\wedge\wedge(Y(t)\bs)_{i+1}\wedge\dots\wedge(Y(t)\bs)_d\vert_{t=0}=\\
&&\sum_{i=1}^ds_1\wedge s_2\wedge\dots\wedge
s_{i-1}\wedge((A(0)\bs)\by)_i\wedge s_{i+1}\wedge\dots\wedge s_d.
\end{eqnarray*}
The last representation implies that $\frac{d}{dt}d^\star
F_{t\by}[\bs]\vert_{t=0}=-\frac{d}{dt}d^\star
F_{-t\by}[\bs]\vert_{t=0}$. On the other hand, it is known that
$d^\star F_{t\by}(\bq)=TJ\J(t,\bq,\by)d F_0(\bq)$, where $d
F_0=dF_{t\by}\vert_{t=0}$ is the volume form on $F_k^+\cap U$. These
circumstances imply that the function
$\frac{d}{dt}TJ\J(t,\bq,\by)\vert_{t=0}$ is odd w.r.t. $\by$. Claim
(ii) is proven.
\end{proof}

\subsection{Proof of claim (i) of Proposition \ref{prrprres1}}
\label{sec:proofclipr1}

\begin{proof}
Let us choose a suitable neighborhood of each of the submanifolds
$F_k^+\;(k\in\{1,2,\dots,n_+\})$. Let $(U_{\J_\star},\phi)$ be a
reducing chart at a point $\J_\star\in F^+_k$. Since $F_k^+$ are
closed disjoint subsets of the torus $\T^d$, we can choose it such
that $ U_{\J_\star}\cap F^+_i=\emptyset\;(U_{\J_\star}\cap
F^-_i=\emptyset)$ for any $i\neq k$. Let
$\Phi:\;U_{\J_\star}\rightarrow\R^{d-d_k^+}$ be the same as in
Lemmas \ref{lmredchrt} and \ref{lmdifeq}. Then by Lemma
\ref{lmdifeq} it is possible to restrict the neighborhood
$U_{\J_\star}$ of $\J_\star$ such that for the flow $\J(t,\bq,\by)$
of the differential equation (\ref{diffeq}) claim (i) of this lemma
is valid. Recall that the dispersion function $\lambda^+(\J)$
branching from the edge $\lambda_+$ of the gap
$(\lambda_-,\,\lambda_+)$ in $\sigma(H_0)$ has the form
$\lambda^+(\J)=\lambda_{j+1}(\J)$ for some $j\ge 0$. Like in Section
\ref{subsec:spectchar}, we can restrict the neighborhood
$U_{\J_\star}$, taking into account condition (A)-(a) and Corollary
\ref{cormainApp}, such that for any $\J\in U_{\J_\star}$
$\lambda^+(\J)$ is a simple eigenvalue of $\tilde H(\J)$, i.e.
\begin{equation}\label{neib1}
\forall\,\J\in U_{\J_\star}:\quad \lambda^+(\J)<\lambda_{j+2}(\J),
\end{equation}
the function $\lambda^+(\J)$ is real-analytic in $U_{\J_\star}$ and
furthermore, the mapping $\J\rightarrow \tilde
\Qc^+(\cdot,\cdot,\J)\in C(\Omega\times\Omega)$ is real-analytic in
$U_{\J_\star}$. Recall that $\tilde\Qc^+(\e,\bs,\J)$ is the
eigenkernel of $\tilde H(\J)$, corresponding to $\lambda^+(\J)$.
Since each $F^+_k\;(k\in\{1,2,\dots,n_+\})$ is compact, then it is
possible to select a finite subcovering $\{U_{k,l}\}_{l=1}^{L_k^+}$
from its open covering $\{U_{\J_\star}\}_{\J_\star\in F^+_k}$.
Consider the neighborhood $U_k^+=\bigcup_{l=1}^{L_k^+} U_{k,l}$ of
$F^+_k$ for each $k\in\{1,2,\dots,n_+\}$. As it is clear, these
neighborhoods are disjoint. Denote $S_+=\bigcup_{k=1}^{n_+}U_k$.

As it is known (\cite{Gel}, \cite{Kuch}, \cite{Zl}), the operator
$\tilde U=\exp(-i\J\cdot\e)\cdot U$, where $U$ is defined by
(\ref{unit}), maps the space $L_2(\R^d)$ on the direct integral
$\int_{\T^d}^\oplus L_2(\R^d/\Gamma)\,d\J\\=L_2(\R^d/\Gamma)\otimes
L_2(\T^d)$ and realizes a unitary equivalence
 between the operator $H_0$ and the direct integral $\tilde H=\int_{\T^d}^\oplus\tilde
 H(\J)\,d\J$ of the operators $\tilde H(\J)$. Recall that the operator $\tilde H(\J)$ is defined
 by (\ref{b16}), (\ref{b17}) and its domain is $W_2^2(\R^d/\Gamma)$. Formula (\ref{unit})
 implies that the operator inverse to $\tilde U$, has the
 form: for any $\phi\in L_2(\R^d/\Gamma)\otimes L_2(\T^d)$
 \begin{eqnarray}\label{invunit1}
(\tilde
U^{-1}\phi)(\e)=\frac{1}{(2\pi)^{\frac{d}{2}}}\int_{\T^d}\exp(i\J\cdot\e)\phi(\e,\J)\,d\J.
 \end{eqnarray}
Let us take $f\in L_2(\R^d)$ and denote $\tilde f=\tilde Uf$. Since
$R_\lambda(H_0)=\tilde U^{-1}R_\lambda(\tilde H)\tilde U$ for any
$\lambda\in\Rs(H_0)$, we get using (\ref{invunit1}):
\begin{eqnarray}\label{res1}
&&(R_\lambda(H_0)f)(\e)=\frac{1}{(2\pi)^{\frac{d}{2}}}\int_{\T^d}\exp(i\J\cdot\e)
(R_\lambda(\tilde H(\J))\tilde f(\cdot,\J))(\e)\,d\J=\nonumber\\
&&(R_\lambda^+f)(\e)+(\Theta_\lambda^+f)(\e),
\end{eqnarray}
where
\begin{eqnarray}\label{dfRpl1}
&&(R_\lambda^+f)(\e)=\frac{1}{(2\pi)^{\frac{d}{2}}}\int_{S_+}\exp(i\J\cdot\e)
(R_\lambda(\tilde H(\J))\tilde f(\cdot,\,\J))(\e)\,d\J=\nonumber\\
&&(\tilde U^{-1}R_\lambda(\tilde H(\J))\chi_{_{S_+}}(\J)\tilde
f(\cdot,\,\J))(\e),
\end{eqnarray}
\begin{eqnarray}\label{dfThtpl1}
&&(\Theta_\lambda^+f)(\e)=\frac{1}{(2\pi)^{\frac{d}{2}}}\int_{\T^d\setminus
S_+}\exp(i\J\cdot\e)
(R_\lambda(\tilde H(\J))\tilde f(\cdot,\,\J))(\e)\,d\J=\nonumber\\
&&(\tilde U^{-1}(R_\lambda(\tilde H(\J))\chi_{_{\T^d\setminus
S_+}}(\J)\tilde f(\cdot,\,\J)))(\e).
\end{eqnarray}
Here we denote by $\chi_{_A}$ the characteristic function for a set
$A\subset \T^d$. Since
$F^+=\bigcup_{k=1}^{n_+}F^+_k=(\lambda^+)^{-1}(\lambda_+)$ and
$\lambda_+=\min_{\J\in\T^d}\lambda^+(\J)$, then
$\lambda^+(\J)>\lambda_+$ for any $\J\notin S_+$. Hence there exists
$\delta>0$ such that for any
$\lambda\in((\lambda_++\lambda_-)/2,\;\lambda_+)$ and
$\J\in\T^d\setminus S_+ $ $dist(\lambda,\sigma(\tilde
H(\J)))\ge\delta$, and hence $\|R_\lambda(\tilde
H(\J))\|_0\le\frac{1}{\delta}$. Denote by $|\|\cdot\||$ the norm of
elements in $L_2(\R^d/\Gamma)\otimes L_2(\T^d)$. Then using the
isometry of the operator $\tilde U$, we get from (\ref{dfThtpl1}):
\begin{eqnarray*}
&&\|\Theta_\lambda^+f\|^2=|\|R_\lambda(\tilde
H(\J))\chi_{_{\T^d\setminus S_+}}(\J)\hat
f(\cdot,\,\J)\||^2=\\
&&\int_{\T^d\setminus S_+}\|R_\lambda(\tilde
H(\J))\hat
f(\cdot,\,\J)\|_2^2\,d\J\le\\
&&\frac{1}{\delta^2}\int_{\T^d}\|\hat
f(\cdot,\,\J)\|_2^2\,d\J=\frac{1}{\delta^2}|\|\tilde U
f\||^2=\frac{1}{\delta^2}\|f\|^2.
\end{eqnarray*}
Recall that we denote by $\Vert\cdot\Vert_2$ and $(\cdot,\cdot)_2$
the norm and the inner product in the space $\B_0=L_2(\R^d/\Gamma)$
(see (\ref{dfinnprd})). Thus, we get that
\begin{equation}\label{estThtpl1}
\forall\,\lambda\in((\lambda_++\lambda_-)/2,\;\lambda_+):\quad\|\Theta_\lambda^+\|\le\frac{1}{\delta}.
\end{equation}

Now consider the case where $\J\in S_+$. As it is known, the
resolvent $R_\lambda(\tilde H(\J)$ can be represented in the form:
\begin{equation}\label{rprRlmbd1}
R_\lambda(\tilde H(\J))g=
\frac{\int_\Omega\tilde\Qc^+(\e,\bs,\J)\tilde
f(\bs,\J)\,d\bs}{\lambda^+(\J)-\lambda}+ \tilde R(\lambda,\J)g\quad
(g\in L_2(\R^d/\Gamma)),
\end{equation}
where
\begin{equation}\label{dfRk1}
\tilde
R(\lambda,\J)=\sum_{l\in\N\setminus\{j+1\}}\frac{(\,\cdot\;,\,e_l(\cdot,\,\J))_2
e_l(\cdot,\,\J)}{\lambda_l(\J)-\lambda}.
\end{equation}
Recall that $\{e_l(\e,\,\J)\}$ is the orthonormal sequence of the
eigenfunctions of the operator $\tilde H(\J)$. Then by
(\ref{dfRpl1}), (\ref{rprRlmbd1}),
$R_\lambda^+=K(\lambda)+\tilde\Theta_\lambda^+$, where
\begin{eqnarray}\label{defKlmbd1}
&&(K(\lambda)f)(\e)=\\
&&\frac{1}{(2\pi)^{\frac{d}{2}}}\sum_{k=1}^{n_+}\int_{U_k}\exp(i\J\cdot\e)
\frac{\int_\Omega\tilde\Qc^+(\e,\bs,\J)\tilde
f(\bs,\J)\,d\bs}{\lambda^+(\J)-\lambda}\,d\J\nonumber,
\end{eqnarray}
\begin{eqnarray*}
&&(\tilde\Theta_\lambda^+f)(\e)=\frac{1}{(2\pi)^{\frac{d}{2}}}\int_{S_+}\exp(i\J\cdot\e)
\tilde R(\lambda,\J)\tilde f(\cdot,\,\J)\,d\J=\nonumber\\
&&(\tilde U^{-1}\tilde R(\lambda,\J)\chi_{_{S_+}}(\J)\hat
f(\cdot,\,\J))(\e)
\end{eqnarray*}
In view of (\ref{neib1}), there exists $\delta_1>0$ such that for
any $\J\in S_+$ and $\lambda\in
((\lambda_++\lambda_-)/2,\;\lambda_+)$:
$dist(\lambda,\;\sigma(\tilde
H(\J))\setminus\{\lambda_{j+1}(\J)\}>\delta_1$, hence in view of
(\ref{dfRk1}), $\|\tilde R(\lambda,\J)\|_0\le\frac{1}{\delta_1}$.
Then in the same manner as estimate (\ref{estThtpl1}), we obtain the
following estimate:
\begin{equation}\label{esttlThtpl1}
\forall\,\lambda\in((\lambda_++\lambda_-)/2,\;\lambda_+):\quad\|\tilde\Theta_\lambda^+\|\le\frac{1}{\delta_1}.
\end{equation}
Taking $f\in L_{2,0}(\R^d)$ and using the $\Gamma$-periodicity of
$\tilde\Qc^+(\e,\bs,\J)$ by $\bs$ and the fact that $\tilde
f(\e,\J)=\exp(-i\J\cdot\e)\hat f(\e,\J)$ (with $\hat f(\e,\J)$
defined by (\ref{unit})), we obtain from (\ref{defKlmbd1}):
\begin{eqnarray*}
&&\hskip-8mm
(K(\lambda)f)(\e)=\frac{1}{(2\pi)^d}\sum_{k=1}^{n_+}\int_{U_k}\exp(i\J\cdot\e)\times
\nonumber\\
&&\hskip-8mm\left(\frac{1}{\lambda^+(\J)-\lambda}
\sum_{\lb\in\Gamma}\int_{\Omega
-\{\lb\}}\exp(-i\J\cdot\bs)\tilde\Qc^+(\e,\bs,\J)f(\bs)\,d\bs\right)\,d\J,
\end{eqnarray*}
Since the inner sum is finite, we obtain after a permutation of sums
and integrals:
$(K(\lambda)f)(\e)=\int_{\R^d}K(\e,\bs,\lambda)f(\bs)\,d\bs$, where
$K(\e,\bs,\lambda)=\sum_{k=1}^{n_+}K_k^+(\e,\bs,\lambda)$,
\begin{equation}\label{Kkxs1}
K_k^+(\e,\bs,\lambda)=
\frac{1}{(2\pi)^d}\int_{U_k}\exp(i\J\cdot(\e-\bs))
\frac{\tilde\Qc^+(\e,\bs,\J)}{\lambda^+(\J)-\lambda} \,d\J.
\end{equation}
We see that each $K_k^+(\e,\bs,\lambda)$ is continuous in
$\R^d\times\R^d\times(\lambda_-,\lambda_+)$. From (\ref{res1}) and
(\ref{dfRpl1}) we obtain the representation (\ref{represol1}) with
$\Theta(\lambda)=\Theta_\lambda^++\tilde \Theta_\lambda^+$. In view
of (\ref{estThtpl1}) and (\ref{esttlThtpl1}), $\Theta(\lambda)$
satisfies condition (\ref{bndtht2}). Claim (i) of Proposition
\ref{prrprres1} is proven.
\end{proof}

\subsection{Proof of claim (ii) of Proposition \ref{prrprres1}}
\label{sec:proofcliipr1}

\begin{proof}
For each $k\in\{1,2,\dots,n_+\}$ consider a decomposition of the
unit
\begin{equation*}
\{\phi_{k,l}(\J)\}_{l=1}^{L_k},\quad (\phi_{k,l}\in C^\infty(U_k)),
\end{equation*}
corresponding to the covering $\{U_{k,l}\}_{l=1}^{L_k}$ of $U_k$,
constructed above. Then we have the representation for the function
$K_k^+(\e,\bs,\lambda)$, defined by (\ref{Kkxs1}):
\begin{equation}\label{rprKkxs1}
K_k^+(\e,\bs,\lambda)=\sum_{l=1}^{L_k}K_{k,l}(\e,\bs,\lambda),
\end{equation}
where
\begin{eqnarray*}
&&K_{k,l}(\e,\bs,\lambda)=\\
&&\frac{1}{(2\pi)^d}\int_{U_{k,l}}\phi_{k,l}(\J)\exp(i\J\cdot(\e-\bs))\frac{\tilde\Qc^+(\e,\bs,\J)}
{\lambda_{j+1}(\J)-\lambda} \,d\J.
\end{eqnarray*}
Let $\Phi_{k,l}$ be the mapping $\Phi$ defined in Lemma
\ref{lmredchrt} and corresponding to the neighborhood $U_{k,l}$.
Like in Lemma \ref{lmdifeq}, consider the differential equation in
$U_{k,l}$:
\begin{equation}\label{difeq}
\frac{d\J}{dt}=\Big(\Phi_{k,l}^\prime(\J)\vert_{N_\J}\Big)^{-1}y.
\end{equation}
Let $\J_{k,l}(t,\J_0,y)$ be the flow of this equation. Taking into
account definition of the reducing chart and using claim (i) of
Lemma \ref{lmdifeq} and the coarea formula, we have:
\begin{eqnarray}\label{coarea}
&&\hskip-16mmK_{k,l}(\e,\bs,\lambda)=
\frac{1}{(2\pi)^d}\int_{B_{k,l}}\frac{d\by}{|\by|^2+\lambda_+-\lambda}
\int_{F_{\by}}\phi_{k,l}(\J)\exp(i\J\cdot(\e-\bs))\times\nonumber\\
&&\hskip-16mm\frac{\tilde\Qc^+(\e,\bs,\J)}{NJ\Phi_{k,l}(\J)}\,dF_{\by}(\J),
\end{eqnarray}
where $B_{k,l}=\{\by\in\R^{d-d_k^+}\,\vert\,|\by|\le
r_{k,l}\}\;(r_{k,l}>0)$, $dF_{\by}(\J)(\cdot)$ is the volume form of
the submanifold $F_{\by}=\Phi_{k,l}^{-1}(\by)$ ($\by\in B_{k,l}$),
and $NJ\Phi_{k,l}(\J)$ is the normal Jacobian of $\Phi_{k,l}$. Since
by claim (i) of Lemma \ref{lmdifeq} for any $t\in [0,r_{k,l}]$ and
$\by\in S^{d-d_k^+-1}$ the mapping $\J_{k,l}(t,\cdot,\by)$ is a
diffeomorphism between $U_{k,l}^0=U_{k,l}\cap F_k^+$ and $F_{t\by}$,
then (\ref{coarea}) can be written in the form after the change of
the variable $\J=\J_{k,l}(t,\bq,\by)\;(\bq\in U_{k,l}^0)$:
\begin{equation}\label{frmKklxs}
K_{k,l}(\e,\bs,\lambda)=\frac{1}{(2\pi)^d}\int_{U_{k,l}^0}E_{k,l}(\bq,\e,\bs,\lambda)\,dF_0(\bq),
\end{equation}
where
\begin{eqnarray}\label{dfEkl}
&&\hskip-8mmE_{k,l}(\bq,\e,\bs,\lambda)=\\
&&\hskip-8mm\int_{S^{d-d_k^+-1}}\,dS(\by)\int_0^{r_{k,l}}\frac{t^{d-d_k^+-1}}
{t^2+\lambda_+-\lambda}
\phi_{k,l}(\J_{k,l}(t,\bq,\by))\exp(i\J_{k,l}(t,\bq,\by)\cdot(\e-\bs))\times\nonumber\\
&&\hskip-8mm\frac{\tilde\Qc^+(\e,\bs,\J_{k,l}(t,\bq,\by))
TJ\J_{k,l}(t,\bq,\by)}{NJ\Phi_{k,l}(\J_{k,l}(t,\bq,\by)))}\,dt,\nonumber
\end{eqnarray}
$dS(\cdot)$ is the volume form of the unit sphere $S^{d-d_k^+-1}$,
$dF_0(\cdot)$ is the volume form of the submanifold $U_{k,l}^0$ and
$TJ\J_{k,l}(t,\bq,\by)$ is the tangential Jacobian of the mapping
$\J_{k,l}(t,\cdot,\by):\;U_{k,l}^0\rightarrow F_{t\by}$. Since for
$d-d_k^+\ge 3$ $\;\int_0^{r_{k,l}}\frac{t^{d-d_k^+-1}\,dt}
{t^2+\lambda_+-\lambda}\le\frac{r_{k,l}^{d-d_k^+-2}}{d-d_k^+-2}<\infty$
for any $\lambda\in(\lambda_-,\lambda_+)$, then by (\ref{dfEkl},
(\ref{frmKklxs}) and (\ref{rprKkxs1}) the property (\ref{cndK1}) is
valid. Hence claim (ii) of Proposition \ref{prrprres1} is proven.
\end{proof}

\subsection{Proof of claim (iii) of Proposition \ref{prrprres1}}
\label{sec:proofcliiipr1}

\begin{proof}
Assume that $1\le d-d_k^+\le 2$. Taking into account the connection
(\ref{connecteigkern}) between the eigenkernels $\Qc^+(\e,\bs,\J)$
and $\tilde\Qc^+(\e,\bs,\J)$, let us represent the function
$E_{k,l}(\bq,\e,\bs,\lambda)$, defined by (\ref{dfEkl}), in the
form:
\begin{equation}\label{rprEkl}
E_{k,l}(\bq,\e,\bs,\lambda)=E_{k,l}^{(1)}(\bq,\e,\bs,\lambda)+E_{k,l}^{(2)}(\bq,\e,\bs,\lambda)+
E_{k,l}^{(3)}(\bq,\e,\bs,\lambda),
\end{equation}
where
\begin{eqnarray}
\label{dfEkl1}
&&E_{k,l}^{(1)}(\bq,\e,\bs,\lambda)=
\phi_{k,l}(\bq)2^{(d-d_k^+)/2}\,\Qc^+(\e,\bs,\bq)\times\\
&&\sqrt{m^+(\bq)}\int_{S^{d-d_k^+-1}}\,dS(\by)\int_0^{r_{k,l}}\frac{t^{d-d_k^+-1}\,dt}
{t^2+\lambda_+-\lambda},\nonumber
\end{eqnarray}
\begin{eqnarray}\label{dfEkl2}
E_{k,l}^{(2)}(\bq,\e,\bs,\lambda)=
\phi_{k,l}(\bq)2^{(d-d_k^+)/2}\,\Qc^+(\e,\bs,\J)
\sqrt{m^+(\bq)}\tilde E_{k,l}^{(2)}(\bq,\e,\bs,\lambda),
\end{eqnarray}
\begin{eqnarray}\label{dftlEkl2}
&&\tilde E_{k,l}^{(2)}(\bq,\e,\bs,\lambda)=
\int_{S^{d-d_k^+-1}}\,dS(\by)\int_0^{r_{k,l}}\frac{t^{d-d_k^+-1}}
{t^2+\lambda_+-\lambda}\times\\
&&\big(\exp(i(\J_{k,l}(t,\bq,\by)-\bq)\cdot(\tilde\e-\tilde\bs))-1\big)\,dt,\nonumber
\end{eqnarray}
\begin{eqnarray}\label{dfEkl3}
&&\hskip-8mmE_{k,l}^{(3)}(\bq,\e,\bs,\lambda)=\int_{S^{d-d_k^+-1}}\,dS(\by)\int_0^{r_{k,l}}\frac{t^{d-d_k^+-1}}
{t^2+\lambda_+-\lambda}\times\\
&&\hskip-8mm\exp(i\J_{k,l}(t,\bq,\by)\cdot(\e-\bs))
D_{k,l}(t,\by,\bq,\e,\bs)\,dt,\nonumber
\end{eqnarray}
\begin{eqnarray}\label{dfDkl}
&&\hskip-8mmD_{k,l}(t,\by,\bq,\e,\bs)=
\phi_{k,l}(\J_{k,l}(t,\bq,\by))
\frac{\tilde\Qc^+(\e,\bs,\J_{k,l}(t,\bq,\by))
TJ\J_{k,l}(t,\bq,\by)}{NJ\Phi_{k,l}(\J_{k,l}(t,\bq,\by))}-
\nonumber\\
&&\hskip-8mm \phi_{k,l}(\bq)
\frac{2^{(d-d_k^+)/2}\,\tilde\Qc^+(\e,\bs,\bq)}
{\sqrt{NJ\,Hes_{\bq}(\lambda_{j+1})}}.
\end{eqnarray}
After simple calculations we have:
\begin{eqnarray}\label{asestint}
\int_0^{r_{k,l}}\frac{t^{d-d_k^+-1}\,dt}
{t^2+\lambda_+-\lambda}=\left\{\begin{array}{ll}
\frac{\pi}{2\sqrt{\lambda_+-\lambda}}+\theta_1(\lambda)&\rm{for}\quad d-d_k^+=1,\\
\frac{1}{2}\ln\Big(\frac{1}{\lambda_+-\lambda}\Big)+\theta_2(\lambda)&\rm{for}\quad d-d_k^+=2,
\end{array}\right.
\end{eqnarray}
where
\begin{equation}\label{esttht}
\theta_\nu(\lambda)=O(1)\quad\rm{for}\quad\lambda\uparrow\lambda_+\quad
(\nu=1,2).
\end{equation}
Then taking into account that
$\{\phi_{k,l}\vert_{F_k^+}\}_{l=1}^{L_k}$ is a decomposition of the
unit for the submanifold $F_k^+$, we get from (\ref{dfEkl1}),
(\ref{asestint}) and (\ref{esttht}):
\begin{equation}\label{rprKk1xslm}
\sum_{l=1}^{L_k}\int_{U_{k,l}^0}E_{k,l}^{(1)}(\bq,\e,\bs,\lambda)\,dF_0(\bq)=
G_k^+(\e,\bs,\lambda)+\tilde K_k^{(1)}(\e,\bs,\lambda),
\end{equation}
where $G_k^+(\e,\bs,\lambda)$ is defined by (\ref{dfFpl1}) for
$d-d_k^+=1$ and by (\ref{dfFpl2}) for $d-d_k^+=2$, and the function
$\tilde K_k^{(1)}(\e,\bs,\lambda)$ is continuous in
$\R^d\times\R^d\times(\lambda_+-\delta,\lambda_+)$ and
\begin{equation}\label{bundK1kxs}
\sup_{(\e,\bs,\lambda)\in\R^d\times\R^d\times(\lambda_+-\delta,\;\lambda_+)}|\tilde
K_k^{(1)}(\e,\bs,\lambda)|<\infty.
\end{equation}
Recall that $\delta\in(0,\,\lambda_+-\lambda_-)$.

 Let us estimate the function
$E_{k,l}^{(2)}(\bq,\e,\bs,\lambda)$, defined by (\ref{dfEkl2}). To
this end consider the Taylor representation of the flow of the
equation (\ref{difeq}) in a neighborhood of the point $t=0$:
\begin{eqnarray}\label{Tayrprp}
&&\hskip-10mm\J_{k,l}(t,\bq,\by)-\bq=\J_{k,l}(t,\bq,\by)-\J_{k,l}(0,\bq,\by)=\partial_t\J_{k,l}(0,\bq,\by)t+\\
&&\hskip-10mm\br_{k,l}(t,\bq,\by),\nonumber
\end{eqnarray}
where
$\br_{k,l}(t,\bq,\by)=\int_0^1(t-s)\partial^2_t\J_{k,l}(s,\bq,\by)\,ds$.
We have from (\ref{difeq}):
\begin{equation*}
\partial^2_t\J_{k,l}(s,\bq,\by)=\partial_\J\Big(\Big(\Phi_{k,l}^\prime(\J)\vert_{N_\J}\Big)^{-1}y\Big)
\cdot(\Big(\Phi_{k,l}^\prime(\J)\vert_{N_\J}\Big)^{-1}y\vert_{\J=\J_{k,l}(t,\bq,\by)}.
\end{equation*}
These circumstances imply that
\begin{equation}\label{bundrtqy}
\bar\br_{k,l}=\sup_{(t,\bq,\by)\in[0,r_{k,l}]\times U_{k,l}\times
S^{d-d_k^+-1}(0)}\frac{|\br_{k,l}(t,\bq,\by)|}{t^2}<\infty.
\end{equation}
Using (\ref{Tayrprp}), let us represent the function $\tilde E_{k,l}^{(2)}(\bq,\e,\bs,\lambda)$, defined by
(\ref{dftlEkl2}), in the form:
\begin{equation}\label{rprtlEkl2}
\tilde E_{k,l}^{(2)}(\bq,\e,\bs,\lambda)=\tilde E_{k,l}^{(2,1)}(\bq,\e,\bs,\lambda)+
\tilde E_{k,l}^{(2,2)}(\bq,\e,\bs,\lambda),
\end{equation}
where
\begin{eqnarray}\label{dftlEkl21}
&&\tilde E_{k,l}^{(2,1)}(\bq,\e,\bs,\lambda)=\\
&&\int_{S^{d-d_k^+-1}}\,dS(\by)\int_0^{r_{k,l}}\frac{t^{d-d_k^+-1}}
{t^2+\lambda_+-\lambda}\big(\exp(i\partial_t\J_{k,l}(0,\bq,\by)t\cdot(\e-\bs))-1\big)\,dt,\nonumber
\end{eqnarray}
and
\begin{eqnarray}\label{dftlEkl22}
&&\tilde E_{k,l}^{(2,2)}(\bq,\e,\bs,\lambda)=\\
&&\int_{S^{d-d_k^+-1}}\,dS(\by)\int_0^{r_{k,l}}\frac{t^{d-d_k^+-1}}
{t^2+\lambda_+-\lambda}\big(\exp(i\br_{k,l}(t,\bq,\by)\cdot(\e-\bs))-1\big)\times\nonumber\\
&&\exp(i\partial_t\J_{k,l}(0,\bq,\by)t\cdot(\e-\bs))\,dt,\nonumber
\end{eqnarray}
Let us estimate $\tilde E_{k,l}^{(2,1)}(\bq,\e,\bs,\lambda)$. Observe that, in view of (\ref{difeq}),
$\partial_t\J_{k,l}(0,\bq,\by)=-\partial_t\J_{k,l}(0,\bq,-\by)$. Then, since the integral of the odd part
(w.r. to $\by$) of the integrand in (\ref{dftlEkl21}) vanishes, we have:
\begin{eqnarray*}
&&\tilde E_{k,l}^{(2,1)}(\bq,\e,\bs,\lambda)=\\
&&\int_{S^{d-d_k^+-1}}\,dS(\by)\int_0^{r_{k,l}}\frac{t^{d-d_k^+-1}}
{t^2+\lambda_+-\lambda}\big(\cos(\partial_t\J_{k,l}(0,\bq,\by)t\cdot(\tilde\e-\tilde\bs))-1\big)\,dt,\nonumber
\end{eqnarray*}
Deriving the change of the variable $u=(1+|\e-\bs|)t$ in the inner integral, we get:
\begin{eqnarray*}
&&\hskip-8mm\tilde E_{k,l}^{(2,1)}(\bq,\e,\bs,\lambda)=(1+|\e-\bs|)^{2-(d-d_k^+)}\times\\
&&\hskip-8mm\int_{S^{d-d_k^+-1}}\,dS(\by)\int_0^{r_{k,l}(1+|\e-\bs|)}\frac{u^{d-d_k^+-1}}
{u^2+(\lambda_+-\lambda)(1+|\e-\bs|)^2}\times\nonumber\\
&&\hskip-8mm\big(\cos\big(\partial_t\J_{k,l}(0,\bq,\by)u\cdot\frac{\tilde\e-\tilde\bs}{1+|\e-\bs|}\big)-1\big)\,du.
\end{eqnarray*}
Using estimate (\ref{estint2}) of Lemma \ref{lmestint} with
$\rho=r_{k,l}$, $\alpha=1+|\e-\bs|$ and
$\beta=\partial_t\J_{k,l}(0,\bq,\by)\cdot\frac{\tilde\e-\tilde\bs}{1+|\e-\bs|}$,
we get:
\begin{eqnarray}\label{summary1}
&&\rm{for}\quad d-d_k^+=1:\\
&&\sup_{(\bq,\e,\bs,\lambda)\in\,
U_{k,l}^0\times\R^d\times\R^d\times(\lambda_+-\delta,\;\lambda_+)}
\frac{\big|\tilde E_{k,l}^{(2,1)}(\bq,\e,\bs,\lambda)\big|}
{1+|\e-\bs|}<\infty\nonumber
\end{eqnarray}
and
\begin{eqnarray}\label{summary2}
&&\rm{for}\quad d-d_k^+=2:\\
&&\sup_{(\bq,\e,\bs,\lambda)\in\,
U_{k,l}^0\times\R^d\times\R^d\times(\lambda_+-\delta,\;\lambda_+)}
\frac{\big|\tilde E_{k,l}^{(2,1)}(\bq,\e,\bs,\lambda)\big|}
{1+\ln(1+|\e-\bs|)}<\infty.\nonumber
\end{eqnarray}
Now let us estimate the function $\tilde E_{k,l}^{(2,2)}(\bq,\e,\bs,\lambda)$, defined by (\ref{dftlEkl22}).
To this end let us derive the change of the variable $u=\sqrt{1+|\e-\bs|}t$ in the inner integral of (\ref{dftlEkl22}):
\begin{eqnarray*}
&&\hskip-16mm\Big|\int_0^{r_{k,l}}\frac{t^{d-d_k^+-1}}
{t^2+\lambda_+-\lambda}\big(\exp(i\br_{k,l}(t,\bq,\by)\cdot(\e-\bs))-1\big)
\times\\
&&\hskip-16mm\exp(i\partial_t\J_{k,l}(0,\bq,\by)t\cdot(\e-\bs))\,dt\Big|\le\\
&&\hskip-16mm(1+|\e-\bs|)^{\frac{2-(d-d_k^+)}{2}}\int_0^{r_{k,l}\sqrt{1+|\e-\bs|}}\frac{u^{d-d_k^+-1}}
{u^2+(\lambda_+-\lambda)(1+|\e-\bs|)}\times\nonumber\\
&&\hskip-16mm\Big|\exp\Big(i\br_{k,l}\Big(\frac{u}{\sqrt{1+|\e-\bs|}},\bq,\by\Big)\cdot(\e-\bs)\Big)-1\Big|\,du
\end{eqnarray*}
Then using (\ref{bundrtqy}) and estimate (\ref{estint1}) of Lemma
\ref{lmestint} with $\rho=r_{k,l}$, $\alpha=\sqrt{1+|\e-\bs|}$ and
$\phi(u)=\br_{k,l}\Big(\frac{u}{\sqrt{1+|\e-\bs|}},\bq,\by\Big)\cdot(\tilde\e-\tilde\bs)$,
we obtain:
\begin{eqnarray}\label{summary3}
&&\rm{for}\quad d-d_k^+=1:\\
&&\sup_{(\bq,\e,\bs,\lambda)\in\,U_{k,l}^0\times\R^d\times\R^d\times(\lambda_+-\delta,\;\lambda_+)}
\frac{\big|\tilde E_{k,l}^{(2,2)}(\bq,\e,\bs,\lambda)\big|}
{\sqrt{1+|\e-\bs|}}<\infty\nonumber
\end{eqnarray}
and
\begin{eqnarray}\label{summary3a}
&&\rm{for}\quad d-d_k^+=2:\\
&&\sup_{(\bq,\e,\bs,\lambda)\in\,U_{k,l}^0\times\R^d\times\R^d\times(\lambda_+-\delta,\;\lambda_+)}
\frac{\big|\tilde E_{k,l}^{(2,2)}(\bq,\e,\bs,\lambda)\big|}
{1+\ln(1+|\e-\bs|)}<\infty.\nonumber
\end{eqnarray}

Let us estimate the function $E_{k,l}^{(3)}(\bq,\e,\bs,\lambda)$,
defined by (\ref{dfEkl3}). To this end consider the Taylor
representation of the function $D_{k,l}(t,\by,\bq,\e,\bs)$, defined
by (\ref{dfDkl}) in a neighborhood of $t=0$, taking into account
that in view of Lemma \ref{lmredchrt} and the equalities
$\J_{k,l}(0,\bq,\by)=\bq$ and $TJ(\J_{k,l}(0,\bq,\by))=1$, the
equality $D_{k,l}(0,\by,\bq,\e,\bs)=0$ is valid. We have:
$D_{k,l}(t,\by,\bq,\e,\bs)=g_{k,l}(\by,\bq,\e,\bs)t+\zeta_{k,l}(t,\by,\bq,\e,\bs)$,
where $g(\by,\bq,\e,\bs)=\partial_tD_{k,l}(0,\by,\bq,\e,\bs)$. We
see from (\ref{dfDkl}) that the functions $g_{k,l}(\by,\bq,\e,\bs)$
and $\zeta_{k,l}(t,\by,\bq,\e,\bs)$ are smooth and bounded in
$[0,r_{k,l}]\times U_{k,l}\times S^{d-d_k^+-1}\times\R^d\times\R^d$.
In the same manner as (\ref{bundrtqy}), we obtain the following
property of $\zeta_{k,l}(t,\by,\bq,\e,\bs)$:
\begin{equation}\label{bundzettqy}
\bar\zeta_{k,l}=\sup_{(t,\bq,\by,\e,\bs)\in[0,r_{k,l}]\times
U_{k,l}\times S^{d-d_k^+-1}\times\R^d\times\R^d}
\frac{|\zeta_{k,l}(t,\bq,\by,\e,\bs)|}{t^2}<\infty.
\end{equation}
Using the above representation and (\ref{Tayrprp}), let us represent
the function $E_{k,l}^{(3)}(\bq,\e,\bs,\lambda)$ in the form:
\begin{eqnarray}\label{rprEkl3}
&&E_{k,l}^{(3)}(\bq,\e,\bs,\lambda)=E_{k,l}^{(3,1)}(\bq,\e,\bs,\lambda)+\exp\big(i\bq\cdot(\e-\bs)\big)
\times\nonumber\\
&&\big(E_{k,l}^{(3,2)}(\bq,\e,\bs,\lambda)+E_{k,l}^{(3,3)}(\bq,\e,\bs,\lambda)\big),
\end{eqnarray}
where
\begin{eqnarray}\label{dfEkl31}
&&\hskip-8mmE_{k,l}^{(3,1)}(\bq,\e,\bs,\lambda)=\int_{S^{d-d_k^+-1}}\,dS(\by)\int_0^{r_{k,l}}\frac{t^{d-d_k^+-1}}
{t^2+\lambda_+-\lambda}\times\\
&&\hskip-8mm\exp(i\J_{k,l}(t,\bq,\by)\cdot(\e-\bs))
\zeta_{k,l}(t,\by,\bq,\e,\bs)\,dt,\nonumber
\end{eqnarray}
\begin{eqnarray}\label{dfEkl32}
&&\hskip-8mmE_{k,l}^{(3,2)}(\bq,\e,\bs,\lambda)=\int_{S^{d-d_k^+-1}}\,dS(\by)\int_0^{r_{k,l}}\frac{t^{d-d_k^+}}
{t^2+\lambda_+-\lambda}\times\\
&&\hskip-8mm\exp(it\partial_t\J_{k,l}(0,\bq,\by)\cdot(\e-\bs))
g_{k,l}(\by,\bq,\e,\bs)\,dt,\nonumber
\end{eqnarray}
\begin{eqnarray}\label{dfEkl33}
&&\hskip-8mmE_{k,l}^{(3,3)}(\bq,\e,\bs,\lambda)=\int_{S^{d-d_k^+-1}}\,dS(\by)\int_0^{r_{k,l}}\frac{t^{d-d_k^+}}
{t^2+\lambda_+-\lambda}\times\\
&&\hskip-8mm\exp(it\partial_t\J_{k,l}(0,\bq,\by)\cdot(\e-\bs))
g_{k,l}(\by,\bq,\e,\bs)\nonumber\times\\
&&\hskip-8mm\big(\exp(i\br_{k,l}(t,\bq,\by)\cdot(\e-\bs))-1\big)\,dt.
\nonumber
\end{eqnarray}
Using (\ref{bundzettqy}), we obtain from (\ref{dfEkl31}):
\begin{eqnarray}\label{summary3}
&&\rm{for}\quad d-d_k^+\in\{1,2\}:\\
&&\sup_{(\bq,\e,\bs,\lambda)\in\,U_{k,l}^0\times\R^d\times\R^d\times(\lambda_+-\delta,\;\lambda_+)}
\big| E_{k,l}^{(3,1)}(\bq,\e,\bs,\lambda)\big| <\infty.\nonumber
\end{eqnarray}
Now we turn to the estimation of the function
$E_{k,l}^{(3,2)}(\bq,\e,\bs,\lambda)$. Deriving the change of the
variable $u=(1+|\e-\bs|)t$ in the inner integral of (\ref{dfEkl32})
and taking into account that by claim (ii) of Lemma \ref{lmdifeq}
the functions $\partial_t\J_{k,l}(0,\bq,\by)$ and
$g_{k,l}(\by,\bq,\e,\bs)=\partial_tD_{k,l}(0,\by,\bq,\e,\bs)$ are
odd w.r.t. $\by$, we get:
\begin{eqnarray*}
&&\hskip-8mmE_{k,l}^{(3,2)}(\bq,\e,\bs,\lambda)=(1+|\e-\bs|)^{1-(d-d_k^+)}\int_{S^{d-d_k^+-1}}
g_{k,l}(\by,\bq,\e,\bs)\,dS(\by)\times\\
&&\hskip-8mm\int_0^{r_{k,l}(1+|\e-\bs|)}
\frac{u^{d-d_k^+}\sin\Big(u\partial_t\J_{k,l}(0,\bq,\by)\cdot\frac{\tilde\e-\tilde\bs}
{1+|\e-\bs|}\Big)}{u^2+(\lambda_+-\lambda)(1+|\e-\bs|)^2}\,du
\end{eqnarray*}
Using estimate (\ref{estint3}) of Lemma \ref{lmestint} with
$\rho=r_{k,l}$, $\alpha=1+|\e-\bs|$ and
$\beta=\partial_t\J_{k,l}(0,\bq,\by)\cdot\frac{\tilde\e-\tilde\bs}{1+|\e-\bs|}$,
we get:
\begin{eqnarray}\label{summary4}
&&\rm{for}\quad d-d_k^+=1:\\
&&\sup_{(\bq,\e,\bs,\lambda)\in\,U_{k,l}^0\times\R^d\times\R^d\times(\lambda_+-\delta,\;\lambda_+)}
\frac{\big|E_{k,l}^{(3,2)}(\bq,\e,\bs,\lambda)
\big|}{1+\ln(1+|\e-\bs|)}<\infty\nonumber
\end{eqnarray}
and
\begin{eqnarray}\label{summary5}
&&\rm{for}\quad d-d_k^+=2:\\
&&\sup_{(\bq,\e,\bs,\lambda)\in\,U_{k,l}^0\times\R^d\times\R^d\times(\lambda_+-\delta,\;\lambda_+)}
\big|E_{k,l}^{(3,2)}(\bq,\e,\bs,\lambda)\big|<\infty\nonumber
\end{eqnarray}

Now we turn to the estimation of the function $E_{k,l}^{(3,3)}(\bq,\e,\bs,\lambda)$. Deriving the change of the variable
$u=\sqrt{1+|\e-\bs|}t$ in the inner integral of (\ref{dfEkl33}), we get:
\begin{eqnarray*}
&&\hskip-12mm|E_{k,l}^{(3,3)}(\bq,\e,\bs,\lambda)|\le(1+|\e-\bs|)^{\frac{1-(d-d_k^+)}{2}}\int_{S^{d-d_k^+-1}}
|g_{k,l}(\by,\bq,\e,\bs)|\,dS(\by)\times\\
&&\hskip-12mm\int_0^{r_{k,l}\sqrt{1+|\e-\bs|}}\frac{u^{d-d_k^+}
\Big|\exp\Big(i\br_{k,l}\big(\frac{u}{\sqrt{1+|\e-\bs|}},\bq,\by\big)\cdot(\e-\bs)\Big)-1\Big|}
{u^2+(\lambda_+-\lambda)(1+|\e-\bs|)}\,du
\end{eqnarray*}
Using, as above, estimate (\ref{estint1}) of Lemma \ref{lmestint},
we obtain:
\begin{eqnarray}\label{summary6}
&&\rm{for}\quad d-d_k^+=1:\\
&&\sup_{(\bq,\e,\bs,\lambda)\in\,U_{k,l}^0\times\R^d\times\R^d\times(\lambda_+-\delta,\;\lambda_+)}
\frac{\big|E_{k,l}^{(3,3)}(\bq,\e,\bs,\lambda)
\big|}{1+\ln(1+|\e-\bs|)}<\infty\nonumber
\end{eqnarray}
and
\begin{eqnarray}\label{summary7}
&&\rm{for}\quad d-d_k^+=2:\\
&&\sup_{(\bq,\e,\bs,\lambda)\in\,U_{k,l}^0\times\R^d\times\R^d\times(\lambda_+-\delta,\;\lambda_+)}
\big|E_{k,l}^{(3,3)}(\bq,\e,\bs,\lambda) \big|<\infty.\nonumber
\end{eqnarray}
The representations (\ref{rprKkxs1}), (\ref{frmKklxs}),
(\ref{rprEkl}), (\ref{rprKk1xslm}), (\ref{rprtlEkl2}),
(\ref{rprEkl3}) and the properties (\ref{bundK1kxs}),
(\ref{summary1}), (\ref{summary2}), (\ref{summary3a})
(\ref{summary3}), (\ref{summary4}), (\ref{summary5}),
(\ref{summary6}) and (\ref{summary7}) imply claim (iii) of
Proposition \ref{prrprres1}.
\end{proof}

\subsection{Proof of Proposition \ref{prrprres}}
\label{sec:proofpr}
\begin{proof}
The proof is the same as the proof of Proposition \ref{prrprres1},
only we should take into account that since in our case each of
$F_k^+$ is a singleton $\{\J_k^+\}$, then the set
$U_{k,l}^0=U_{k,l}\cap F_k^+$ is or this singleton, or it is empty.
Hence in the first case in the r.h.s. of (\ref{frmKklxs}) it will be
the value of the integrand at $\J=\J_k^+$ instead of the integral
along $U_{k,l}^0$. Therefore in the r.h.s. of (\ref{dfFpl1}) and
(\ref{dfFpl2}) it will be the value of the integrand at $\J=\J_k^+$
instead of the integral along $F_k^+$.
\end{proof}

\section{\bf Appendix : $C(\Omega)$-holomorphy of Bloch functions}
\label{sec:appendix}

\setcounter{equation}{0}
\setcounter{theorem}{0}

\renewcommand{\thetheorem}{A.\arabic{theorem}}
\renewcommand{\thelemma}{A.\arabic{lemma}}
\renewcommand{\theproposition}{A.\arabic{proposition}}
\renewcommand{\theremark}{A.\arabic{remark}}
\renewcommand{\thecorollary}{A.\arabic{corollary}}
\renewcommand{\thedefinition}{A.\arabic{definition}}
\renewcommand{\theequation}{A.\arabic{equation}}
\renewcommand{\thesubsection}{A.\arabic{subsection}}

\subsection{Main claims}
\label{subsec:mainclaim}

 In this Appendix we prove that
under some assumptions the Bloch functions can be chosen to be
holomorphic w.r.t. the quasi-momentum in the $C(\Omega)$-norm.

In this section we shall denote by $\Vert\cdot\Vert_q$ and
$\Vert\cdot\Vert_{p,\,l}$ the norms in the spaces $L_q(\Omega)$ and
$W_p^l(\Omega)$ respectively. Observe that for $q=2$ the notation
$\Vert\cdot\Vert_q$ is compatible with the notation given by
(\ref{dfinnprd}).

The main results of this section are following:

\begin{theorem}\label{thmainApp}
Assume that the periodic potential $V(\e)$ satisfies the condition
(\ref{condperpotent}) Then\vskip2mm

(i) if $\lambda(\J)$ is an eigenvalue of the operator $H(\J)$, then
any eigenfunction $b(\e,\J)$ of $H(\J)$ corresponding to
$\lambda(\J)$ belongs to $C(\Omega)$;\vskip2mm

(ii) if the family of operators $H(\J)$ has a holomorphic branch of
eigenvalues $\lambda(\J)$ in a connected neighborhood
$\Oc(\J_0)\subset\C^d$ of a point $\J_0\in\R^d$, and for each
$\J\in\Oc(\J_0)$ it is possible to choose an eigenfunction
$b(\e,\J)\neq 0$ of $H(\J)$ corresponding to $\lambda(\J)$ such that
the mapping $\J\rightarrow b(\e,\J)\in L_2(\Omega)$ is holomorphic
in $\Oc(\J_0)$, then the mapping $\J\rightarrow b(\e,\J)\in
C(\Omega)$ is holomorphic in $\Oc(\J_0)$.
\end{theorem}

\begin{corollary}\label{cormainApp}
If $V(\e)$ is as in Theorem \ref{thmainApp} and for some
$\J_0\in\R^d$ an eigenvalue $\lambda_0$ of $H(\J_0)$ is simple,
then\vskip2mm

(i) for some neighborhood $\Oc(\J_0)\subset\C^d$ of $\J_0$ there
exists a branch $\lambda(\J)$ of eigenvalues of the family
$H(\J)\;(\J\in\Oc(\J_0))$ such that $\lambda(\J_0)=\lambda_0$,
$\lambda(\J)$ is simple for any $\J\in\Oc(\J_0)$, the function
$\lambda(\J)$ is holomorphic in $\Oc(\J_0)$ and for any
$\J\in\Oc(\J_0)$ it is possible to choose an eigenfunction
$b(\e,\J)$ of $H(\J)$ corresponding to $\lambda(\J)$ such that
$\Vert b(\e,\J)\Vert_2=1$, the function $b(\e,\J)$ is continuous in
$\Omega$, and the mapping $\J\rightarrow b(\cdot,\J)\in C(\Omega)$
is real-analytic in $\Oc(\J_0)\cap\R^d$; vskip2mm

(ii) for any $\J\in\Oc(\J_0)\cap\R^d$ the eigenkernel, corresponding
to $\lambda(\J)$, has the form
$\Qc(\e,\bs,\J)=b(\e,\J)\overline{b(\bs,\J)}$ and the mapping
$\J\rightarrow\Qc(\cdot,\cdot,\J)\in C(\Omega\times\Omega)$ is real
analytic in $\Oc(\J_0)\cap\R^d$; vskip2mm

(iii) for any $\J\in\Oc(\J_0)\cap\R^d$ the eigenkernel
$\Qc(\e,\bs,\J)$ does not depend on the choice of the branch
$b(\e,\J)$ of eigenfunctions of the family
$H(\J)\;(\J\in\Oc(\J_0))\cap\R^d$ satisfying the conditions imposed
in claim (i).
\end{corollary}

Notice that for $d=3$  the claim of Theorem \ref{thmainApp} follows
from the results of the paper \cite{Wil}. It have been shown there
(\cite{Wil}, Lemmas 3.7, 3.8) that the square $R^2(\J)$ of the
resolvent $R(\J)$ of the operator $H(\J)$ is an integral operator
with an integral kernel $K(\e,\bs,\J)$, such that for some
$\gamma>0$ for any fixed $\J$ with $|\Im(\J)|\le\gamma$
$K(\e,\bs,\J)\in C(\Omega\times\Omega)$ and the mapping
$\J\rightarrow K(\e,\bs,\J)\in C(\Omega\times\Omega)$ is holomorphic
for $|\Im(\J)|\le\gamma$. This fact implies easily the claim of
Theorem \ref{thmainApp}. Observe that the arguments of \cite{Wil}
are true also in the case where $d<3$. But in the case $d\ge 4$ the
integral kernel of the resolvent $R(\J)$ has a stronger singularity
at its diagonal, hence we need to deal with a higher power $R^l(\J)$
of it in order to get an integral kernel having the property
mentioned above.

\subsection{Domains and self-adjointness of the operators $H(\J)$ and $H_0$}
\label{subsec:mainclaim}

In this section we generalize to the case of an arbitrary dimension
$d$ the results on domains and self-adjointness of the operators
$H(\J)$ and $H_0$, obtained in \cite{Wil} for $d=3$ (Lemmas 1.2 and
1.4). The arguments used in \cite{Wil} are true also for $d<3$.
These arguments are based on the fact that for $d\le 3$ the
continuous embedding $W_2^2(\Omega)\hookrightarrow C(\Omega)$ holds.
For $d\ge 4$ this embedding is not true, but we use in this case the
Sobolev's theorem on embedding of $W_p^2(\Omega)$ into
$L_q(\Omega)$. First of all, let us prove the lemma, whose first
claim is an analog of Lemma 1.3 from \cite{Wil}:

\begin{lemma}\label{lmestVuOm}
(i) If $d\ge 4$ and  $V\in L_s(\Omega)$ with $s>\frac{d}{2}$, then
there exists $C>0$ such that for any $u\in W_2^2(\Omega)$ and any
$\epsilon>0$
\begin{equation}\label{estnrmVu1}
\Vert Vu\Vert_2^2\le C|\Vert V\Vert_s^2(\epsilon\Vert
u\Vert_{2,\,2}^2+\epsilon^{-\mu}\Vert u\Vert_2^2),
\end{equation}
where
\begin{equation}\label{dfmu}
\mu=\mu(s)=\frac{(p(s))^{-1}-(q(s))^{-1}}{(q(s))^{-1}-(q(\tilde
s))^{-1}},
\end{equation}
\begin{equation}\label{dfqs}
q(s)=\frac{2s}{s-2},
\end{equation}
\begin{equation}\label{dfps}
p(s)=\frac{2ds}{(d+4)s-2d}
\end{equation}
and
\begin{equation}\label{dftls}
\tilde s=\frac{1}{2}\Big(\frac{d}{2}+s\Big);
\end{equation}

(ii) If $d\ge 2$ and  $V\in L_s(\Omega)$ with $s>\frac{d}{2}$, then
there exists $\tilde C>0$ such that for any $u\in W_2^1(\Omega)$ and
any $\epsilon>0$ $\Big|\int_\Omega V(\e)|u(\e)|^2\,d\e|\Big|\le
\tilde C|\Vert V\Vert_s(\epsilon\Vert
u\Vert_{2,\,1}^2+\epsilon^{-\tilde\mu}\Vert u\Vert_2^2)$, where
\begin{equation}\label{dftlmu}
\tilde\mu=\tilde\mu(s)=\frac{(\tilde p(s))^{-1}-(\tilde
q(s))^{-1}}{(\tilde q(s))^{-1}-(\tilde q(\tilde s))^{-1}},\quad
\tilde p(s)=\frac{2ds}{(d+2)s-d},
\end{equation}
\begin{equation}\label{dftlqs}
\tilde q(s)=\frac{2s}{s-1}
\end{equation}
and $\tilde s$ is defined by (\ref{dftls});\vskip2mm

(iii) If $d=1$ and  $V\in L_1(\Omega)$, then for any $u\in
W_2^1(\Omega)$ and any $\epsilon>0$ $\Big|\int_\Omega
V(\e)|u(\e)|^2\,d\e|\Big|\le|\Vert V\Vert_1(\epsilon\Vert
u\Vert_{2,\,1}^2+(\epsilon^{-1}+T^{-1})\Vert u\Vert_2^2)$, where
$T=lenth(\Omega)$.
\end{lemma}
\begin{proof}
(i) Using H\"older's inequality, we have for any $u\in
W_2^2(\Omega)$:
\begin{eqnarray}\label{Holdest}
&&\Vert Vu\Vert_2^2=\int_\Omega (V(\e))^2|u(\e)|^2\,d\e\le\\
&&\Big(\int_\Omega(V(\e))^s\,d\e\Big)^{\frac{2}{s}}\Big(\int_\Omega|u(\e)|^{q(s)}\,d\e\Big)^{\frac{2}{q(s)}}=
\Vert V\Vert_s^2\,\Vert u\Vert_{q(s)}^2,\nonumber
\end{eqnarray}
where $q(s)$ is defined by (\ref{dfqs}). We see from the
representation $q(s)=2+\frac{4}{s-2}$ that $q(s)$ is decreasing and
if $s$ runs over $\Big(\frac{d}{2},\,\infty\Big)$, $q(s)$ runs over
$\Big(2,\,\frac{2d}{d-4}\Big)$ for $d>4$ and over $(2,\,\infty)$ for
$d=4$. By Sobolev's embedding theorem, if
\begin{equation}\label{frmp}
\frac{1}{p}=\frac{1}{q(s)}+\frac{2}{d},
\end{equation}
then $W_p^2(\Omega)\hookrightarrow L_{q(s)}(\Omega)$. Taking into
account (\ref{dfqs}), we get that the number $p=p(s)$, for which
(\ref{frmp}) holds, has the form (\ref{dfps}). Observe that the
representation $p(s)=\frac{2d}{d+4}\Big(1+\frac{2d}{(d+4)s-2d}\Big)$
implies that $p(s)$ is decreasing and when $s$ runs over
$\Big(\frac{d}{2},\,\infty\Big)$, $p(s)$ runs over
$\Big(\frac{2d}{d+4},\,2\Big)$. Observe that, in view of
(\ref{dftls}), $p(s)<p(\tilde s)<2<q(s)<q(\tilde s)$, hence the
interpolation inequality $\Vert u\Vert_{q(s)}\le\Vert
u\Vert_{p(s)}^\lambda\Vert u\Vert_{q(\tilde s)}^{1-\lambda}$ is
valid with $\lambda=\frac{(q(s))^{-1}-(q(\tilde
s))^{-1}}{(p(s))^{-1}-(q(\tilde s))^{-1}}$. Then the Young's
inequality implies that for any $\epsilon>0$
\begin{equation}\label{interpestuq}
\Vert u\Vert_{q(s)}^2\le\epsilon \Vert u\Vert_{q(\tilde
s)}^2+\epsilon^{-\mu}\Vert u\Vert_{p(s)}^2,
\end{equation}
where $\mu=\mu(s)$ is defined by (\ref{dfmu}) (see \cite{Gil-Tr},
Chapt. 7). By the Sobolev's embedding theorem, $W_{p(\tilde
s)}^2(\Omega)\hookrightarrow L_{q(\tilde s)}(\Omega)$, hence since
$W_2^2(\Omega)\hookrightarrow W_{p(\tilde s)}^2(\Omega)$, we get:
$W_2^2(\Omega)\hookrightarrow L_{q(\tilde s)}(\Omega)$. This fact,
the embedding $L_2(\Omega)\hookrightarrow L_{p(s)}(\Omega)$ and the
inequalities (\ref{Holdest}) and (\ref{interpestuq}) imply that for
some $C>0$ and for any $u\in W_2^2(\Omega)$, $\epsilon>0$ the
desired inequality (\ref{estnrmVu1}) is valid. Claim (i) is proven.

(ii) We have for $u\in W_2^1(\Omega)$: $\Big|\int_\Omega
V(\e)|u(\e)|^2\,d\e|\Big|\le|\Vert V\Vert_s\Vert u\Vert_{\tilde
q(s)}^2$, where $\tilde q(s)$ is defined by (\ref{dftlqs}). Further
we continue the proof like the proof of claim (i), using the fact
that $W_p^1(\Omega)\hookrightarrow L_{\tilde q(s)}(\Omega)$, if
$\frac{1}{p}=\frac{1}{\tilde q(s)}+\frac{1}{d}$.

(iii) Using the Newton-Leibnitz formula, we get easily that for any
$u\in W_2^1(\Omega)$, $\epsilon>0$ and $x\in\Omega$ the inequality
holds $|u(x)|^2\le \epsilon\Vert
u\Vert_{2,\,1}^2+(\epsilon^{-1}+T^{-1})\Vert u\Vert_2^2$, from which
the claim follows immediately.
\end{proof}

 Using Lemma \ref{lmestVuOm} and the arguments of the proof of
 Lemma 1.1 from \cite{Wil}, it is not difficult to prove the
 following claim:

\begin{lemma}\label{lmestVuRd}
(i) If $d\ge 4$ and $V\in L_s(\Omega)$ with $s>\frac{d}{2}$, then
there exists $C>0$ such that for any $u\in W_2^2(\R^d)$ and any
$\epsilon>0$
\begin{equation*}
\Vert Vu\Vert_{L_2(\R^d)}\le C\Vert V\Vert_s(\epsilon\Vert
u\Vert_{W_2^2(\R^d)}+\epsilon^{-\mu}\Vert u\Vert_{L_2(\R^d)}),
\end{equation*}
where $\mu=\mu(s)$ is defined by
(\ref{dfmu})-(\ref{dftls});\vskip2mm

(ii) If $d\ge 2$ and  $V\in L_s(\Omega)$ with $s>\frac{d}{2}$, then
there exists $\tilde C>0$ such that for any $u\in W_2^1(\R^d)$ and
any $\epsilon>0$
\begin{equation*}
\Big|\int_{\R^d} V(\e)|u(\e)|^2\,d\e|\Big|\le \tilde C|\Vert
V\Vert_s(\epsilon\Vert
u\Vert_{W_2^1(\R^d)}^2+\epsilon^{-\tilde\mu}\Vert
u\Vert_{L_2(\R^d)}^2),
\end{equation*}
where $\tilde\mu=\tilde\mu(s)$ is defined by
(\ref{dftlmu});\vskip2mm

(iii) If $d=1$ and  $V\in L_1(\Omega)$, then for any $u\in
W_2^1(\Omega)$ and any $\epsilon>0$
\begin{equation*}
\Big|\int_{\R^d} V(\e)|u(\e)|^2\,d\e|\Big|\le\Vert
V\Vert_1(\epsilon\Vert
u\Vert_{W_2^1(\R^d)}^2+(\epsilon^{-1}+T^{-1})\Vert
u\Vert_{L_2(\R^d)}^2),
\end{equation*}
where $T=lenth(\Omega)$.
\end{lemma}

We now turn to the main claim of this section.

\begin{proposition}\label{propselfadj}
If $V(\e)$ satisfies the condition (\ref{condperpotent}),
then\vskip2mm

(i) for any $\J\in \T^d$ the operator $H(\J)$, generated in the
space $\B_{\J}$ by the operation $h=-\Delta+V(\e)\cdot$ and having
the domain $\Dc_\J=\B_{\J}\cap W_{2,loc}^2(\R^d)$, is self-adjoint
and bounded below uniformly w.r.t. $\J\in \T^d$;\vskip2mm

(ii) the operator $H_0$, generated in the space $L_2(\R^d)$ by the
operation $h=-\Delta+V(\e)\cdot$ and having the domain
$W_2^2(\R^d)$, is self-adjoint and bounded below.\vskip2mm

Recall that $\B_{\J}$ is the Hilbert space of functions $u\in
L_{2,loc}(\R^d)$ satisfying the condition (\ref{Htau}) with the
inner product, defined by (\ref{dfinnprd}).
\end{proposition}
\begin{proof}
For $d=3$ the claims are proved in \cite{Wil} (Lemmas 1.2 and 1.4)
and the arguments used there are true also in the case $d<3$. Let us
prove claim (i) for the case $d\ge 4$. Assume that $V\in
L_s(\Omega)$ with $s>\frac{d}{2}$. Then using claim (i) of Lemma
\ref{lmestVuOm} and the arguments from the proof of Lemma 1.3 from
\cite{Wil}, we get that for some $\tilde C>0$ and for any
$u\in\Dc_\J$, $\epsilon>0$ $\;\Vert Vu\Vert_2\le \tilde C|\Vert
V\Vert_s(\epsilon\Vert\Delta u\Vert_2+\epsilon^{-\mu}\Vert
u\Vert_2)$ and the operator $-\Delta$ with the domain $\Dc_\J$ is
self-adjoint and non-negative. Hence claim (i) follows from the
Kato's theorem (\cite{Kat}, p. 287, Theorem 4.3). In the analogous
manner claim (ii) for $d\ge 4$ follows from claim (i) of Lemma
\ref{lmestVuRd}, arguments from the proof of Lemma 1.1 of \cite{Wil}
and the Kato's theorem mentioned above.
\end{proof}

 In the same manner as in \cite{Wil} (Lemma 1.5) the following claim
 is proved:
\begin{proposition}\label{prcompres}
If $V(\e)$ satisfies the condition (\ref{condperpotent}), then the
resolvent operator $R_\zeta(H(\J))=(H(\J)-\zeta I)^{-1}$ is compact
for every $\zeta$ in the resolvent set of $H(\J)$. Hence, in
particular, $H(\J)$ has a discrete spectrum $\sigma(H(\J))$ for
every $\J\in\T^d$.
\end{proposition}

\subsection{Fundamental solution of the Helmholz's equation in $\R^d$}
\label{subsec:fundsolHelm}

First of all, let us study the structure of the fundamental solution
$\Ec(\e,\gamma_0)\;(\gamma_0>0)$ of the Helmholz equation in the
space $\R^d$, that is the generalized solution of the equation
\begin{equation}\label{fndsl}
-\Delta\Ec+\gamma_0^2\Ec=\delta(\e),
\end{equation}
belonging to the space $S^\prime(\R)$ of slowly growing
distributions.

\begin{proposition}\label{prfndsl}
(i) Equation (\ref{fndsl})has a unique solution $\Ec_d\in
S^\prime(\R^d)$;
\vskip2mm

(ii) it is spherically symmetric, that is
$\Ec_d(\e,\gamma_0)=\Ec_d(|\e|,\gamma_0)$, $\Ec(\e,\gamma_0)>0$ for
$\e\neq 0$, and it has the form:

(iii) for $d=1$
\begin{equation}\label{fndslone}
\Ec_d(x,\gamma_0)=\frac{1}{2\gamma_0}e^{-\gamma_0|x|};
\end{equation}
\vskip2mm

(iv) for
$d=3$$\;\Ec_d(\e,\gamma_0)=\frac{1}{4\pi|\e|}e^{-\gamma_0|\e|}$;
\vskip2mm

(v) for $d=2m+1\;(m=2,3,\dots)$
\begin{eqnarray*}
&&\Ec_d(\e,\gamma_0)=\frac{s_{d-2}}{2(2\pi)^{d-1}}\sum_{k=0}^{m-1}
\left(\begin{array}{l}
m-1\\
\quad k
\end{array}\right)
(-1)^{m-1-k}\sum_{j=0}^{2k} \left(\begin{array}{l}
2k\\
\;j
\end{array}\right)
\gamma_0^{2(m-1-k)+j}\times \nonumber\\
&&\frac{(2k-j)!\;e^{-\gamma_0|\e|}}{|\e|^{2k-j+1}};
\end{eqnarray*}
\vskip2mm

(vi) for $d=2$
$\;\Ec_d(\e,\gamma_0)=\frac{1}{2\pi}K_0(\gamma_0|\e|)$, where
\begin{equation}\label{mcdn}
K_\nu(x):=\int_1^\infty\frac{e^{-xt}}{t^\nu\sqrt{t^2-1}}\,dt
\end{equation}
$(x>0,\;\nu\ge 0)$ is the MacDonald's function; \vskip2mm

(vii) for $d=2m\;(m=2,3,\dots)$
\begin{eqnarray*}
&&\Ec_d(\e,\gamma_0)=\frac{s_{d-1}}{(2\pi)^{d}}\sum_{k=0}^{m-1}
\left(\begin{array}{l}
m-1\\
\quad k
\end{array}\right)
(-1)^{m-1-k}\sum_{j=0}^{2k} \left(\begin{array}{l}
2k\\
\;j
\end{array}\right)
\gamma_0^{2(m-1-k)+j}\times \nonumber\\
&&\frac{(2k-j)!}{|\e|^{2k-j}}K_{2k-j}(\gamma_0|\e|).
\end{eqnarray*}
\end{proposition}
\begin{proof}
We see from (\ref{fndsl}) that the Fourier transform $\hat\Ec_d$ of
$\Ec_d$ satisfies the equation
\begin{equation}\label{eqfrtrfnsl}
(|\J|^2+\gamma_0^2)\hat\Ec=\frac{1}{(2\pi)^{d/2}},
\end{equation}
hence
$\hat\Ec_d(\J,\gamma_0)=\frac{1}{(2\pi)^{d/2}(|\J|^2+\gamma_0^2)}$
and, as it is clear, this is a unique solution of (\ref{eqfrtrfnsl})
belonging to $S^\prime(\R^d)$. This proves claim (i) of the
proposition. Let us reconstruct $\Ec_d$ from $\hat\Ec_d$:
\begin{equation}\label{reconstr}
\Ec_d(\e,\gamma_0)=\frac{1}{(2\pi)^d}\int_{\R^d}\frac{e^{i\J\cdot\e}}{|\J|^2+\gamma_0^2}\,d\J.
\end{equation}
for d=1 we have using Jordan's lemma:
\begin{eqnarray}
\Ec_d(x,\gamma_0)=\frac{1}{2\pi}\int_{-\infty}^\infty\frac{e^{ipx}}{p^2+\gamma_0^2}\,dp=
\left\{\begin{array}{ll} i
Res_{p=i\gamma_0}\frac{e^{ipx}}{p^2+\gamma_0^2}& for x>0,\\
-i Res_{p=-i\gamma_0}\frac{e^{ipx}}{p^2+\gamma_0^2}& for x<0
\end{array}\right.
=\frac{1}{2\gamma_0}e^{-\gamma_0|\e|}. \nonumber
\end{eqnarray}
So, we have proved claim (iii). Rotating the space $\R^d$ such that
the direction of the vector $\e$ transforms to the direction of the
axis $p_d$, we have from (\ref{reconstr}) for $d>1$:
\begin{equation*}
\Ec_d(\e,\gamma_0)=\frac{1}{(2\pi)^d}\int_{\R^{d-1}}d\tilde\J\int_{-\infty}^\infty
\frac{e^{ip_d|\e|}}{|\tilde\J|^2+p_d^2+\gamma_0^2}\,dp_d,
\end{equation*}
where $\tilde\J=(p_1,p_2,\dots,p_{d-1})$. Applying Jordan's lemma to
the inner integral, we have:
\begin{eqnarray}\label{descent}
&&\Ec_d(\e,\gamma_0)=\frac{1}{2(2\pi)^{d-1}}\int_{\R^{d-1}}\frac{\exp(-\sqrt{|\tilde\J|^2+\gamma_0^2}|\e|)}
{\sqrt{|\tilde\J|^2+\gamma_0^2}}\,d\tilde\J=\nonumber\\
&&\frac{s_{d-2}}{2(2\pi)^{d-1}}\int_0^\infty
\frac{r^{d-2}\exp(-\sqrt{r^2+\gamma_0^2}|\e|)}
{\sqrt{r^2+\gamma_0^2}}\,dr
\end{eqnarray}
We see from the last equality and (\ref{fndslone}) that
$\Ec_d(\e,\gamma_0)$ is spherically symmetric, that is
$\Ec_d(\e,\gamma_0)=\Ec(|\e|,\gamma_0)$, it is finite for $\e\neq 0$
and $\Ec_d(\e,\gamma_0)>0$. We have proved claim (ii). Taking in
(\ref{descent}) $d=3$, we get easily claim (iv).

For $d\ge 4$ let us write (\ref{descent}) in the form:
\begin{eqnarray*}
&&\Ec_d(\e,\gamma_0)=\frac{s_{d-2}}{2(2\pi)^{d-1}}\left(\int_0^\infty
\frac{r^{d-4}(r^2+\gamma_0^2)\exp(-\sqrt{r^2+\gamma_0^2}|\e|)}{\sqrt{r^2+\gamma_0^2}}\,dr
-\right.\\
&&\left.\gamma_0^2 \int_0^\infty
\frac{r^{d-4}\exp(-\sqrt{r^2+\gamma_0^2}|\e|)}
{\sqrt{r^2+\gamma_0^2}}\,dr\right)=\nonumber\\
&&\frac{s_{d-2}}{s_{d-4}(2\pi)^2}\left(\frac{d^2}{d\rho^2}-\gamma_0^2\right)
\Ec_{d-2}(\rho,\gamma_0)|_{\rho=|\e|},
\end{eqnarray*}
From this recursive formula and claim (iv) we get easily claim (v).

Using the Hadamard's descent principle:
\begin{equation*}
\Ec_{2m}(\e,\gamma_0)=\int_{-\infty}^\infty\Ec_{2m+1}((\e,\xi),\gamma_0)\,d\xi
\end{equation*}
(see \cite{Wl}), it is not difficult to prove claims (vi) and (vii)
with the help of claims (iv) and (v).
\end{proof}

\begin{corollary}\label{corest}
For any natural $d$ there exists $M=M(d,\gamma_0)>0$ such
that\vskip2mm

(i) if $d=1$, then for any $x\in\R$ $\;\Ec_d(x,\gamma_0)\le
Me^{-\gamma_0|x|}$; vskip2mm

 (ii) if $d\ge 3$ is odd, then
\begin{equation*}
\Ec_d(\e,\gamma_0)\le \left\{\begin{array}{ll} \frac{M}{|\e|^{d-2}}&
for\quad
|\e|\le 1\\
Me^{-\gamma_0|\e|}& for\quad |\e|> 1;
\end{array}\right.
\end{equation*}

(iii) if $d$ is even, then
\begin{equation*}
\Ec_d(\e,\gamma_0)\le \left\{\begin{array}{ll}
\frac{M}{|\e|^{d-2}}\ln\left(\frac{1}{|\e|}\right)& for\quad
|\e|\le 1\\
Me^{-\gamma_0|\e|}& for\quad |\e|> 1;
\end{array}\right.
\end{equation*}

(iv) in particular, $\Ec_d(\e,\gamma_0)\in L_q(\R^d)$ for $d=1,2$,
$q\in[1,\infty)$ and for $d>2$, $q\in[1,\,\frac{d}{d-2})$.
\end{corollary}

The following claims are proved easily:

\begin{lemma}\label{lmdil}
The following equality is valid:
$\Ec_d(\e,\gamma_0)=\gamma_0^{d-2}\Ec_d(\gamma_0\e,1)$.
\end{lemma}

\begin{lemma}\label{rsintop}
The restriction of the operator
$R_{-\gamma_0^2}(-\Delta)=(-\Delta+\gamma_0^2)^{-1}$ on the set
$L_1(\R^d)\cap L_\infty(\R^d)$ is represented as the integral
operator:
\begin{equation*}
R_{-\gamma_0^2}(-\Delta)f=\int_{\R^d}\Ec_d(\e-\bs,\gamma_0)f(\bs)\,d\bs\quad
(f\in L_1(\R^d)\cap L_\infty(\R^d)).
\end{equation*}
\end{lemma}

\subsection{Green's function of the operator $H(\J)$}
\label{subsec:GreenHp}

Let $H_0(\J)$ be the operator generated  by the operation $-\Delta$
in the space $\B_\J$ (defined in Section \ref{sec:preliminaries})
with the domain $\Dc_\J(\Gamma)=W_{2,loc}^2(\R^d)\cap\B_{\J}$.
Consider the function
\begin{equation}\label{dfG0}
G_0(\e,\J,\gamma_0):=\sum_{\mb\in\Z^d}\Ec_d(\e-\mb,\gamma_0)\exp(i\J\cdot\mb).
\end{equation}
Furthermore, we shall assume in  what follows that $d\ge 4$, because
the case $d\le 3$ have been studied in \cite{Wil}. Using the
function $G_0(\e,\J,\gamma_0)$, we shall construct below the
integral kernel of the operator
$R_{-\gamma_0^2}(H(\J))=(H(\J)+\gamma_0^2)^{-1}$, where the operator
$H(\J)$ have been defined in Section \ref{sec:preliminaries}. But
first of all we shall study some properties of the function
$G_0(\e,\J,\gamma_0)$.

Denote by $L_{c,q}(\Omega\times\Omega)\,(q\ge 1)$ the set of all
measurable functions $f:\,\Omega\times\Omega\rightarrow\C$ such that
for any fixed $\e\in\Omega$ the function $f(\e,\cdot)$ belongs to
$L_q(\Omega)$ and the function $\e\rightarrow f(\e,\cdot)\in
L_q(\Omega)$ is continuous. This is a Banach space w. r. to the norm
\begin{equation*}
\|f\|_{c,q}:=\|\|f(\e,\cdot)\|_{L_q(\Omega)}\|_{C(\Omega)}=
\max_{\e\in\Omega}\Big(\int_\Omega|f(\e,\bs)|^q\,d\bs\Big)^{\frac{1}{q}}.
\end{equation*}
In the analogous manner the space $L_{q,c}(\Omega\times\Omega)$,
having the norm
\begin{equation*}
\|f\|_{q,c}:=\|\|f(\cdot,\bs)\|_{L_q(\Omega)}\|_{C(\Omega)}=
\max_{\bs\in\Omega}\Big(\int_\Omega|f(\e,\bs)|\,d\e\Big)^{\frac{1}{q}},
\end{equation*}
is defined.

If $q\ge 1$, we shall denote by $q^\prime$ the number conjugate to
$q$, that is $1/q+1/q^\prime=1$. As it is easy to check,
$q\in(1,\frac{d}{d-2})$ if and only if $q^\prime>\frac{d}{2}$.

Denote $\B_{\J,q^\prime}:=\B_\J\cap\L_{q^\prime,loc}(\R^d)$,
$\B_{\J,\infty}:=\B_\J\cap\L_{\infty,loc}(\R^d)$,
$\Z^d_+=\{\mb\in\Z^d\;|\;m_k\ge 0\;(k=1,2,\dots,d)\}$. For
$\mb\in\Z^d$ we denote $|\mb|_\infty:=\max_{1\le k\le d}|m_k|$,
$|\mb|_1:=\sum_{k=1}^d|m_k|$ ($\mb=(m_1,\dots,m_d)$).

Applying Corollary \ref{corest} of Proposition \ref{prfndsl} and
using the same arguments that have been used for $d=3$ in \cite{Wil}
(Lemmas 3.1, 3.2, 3.5), we can prove the following claims, using the
$L_{c,q}(\Omega\times\Omega)$-norm for integral kernels of operators
instead of the $L_2(\Omega\times\Omega)$-norm  used in \cite{Wil}:

\begin{lemma}\label{lmcnvser}
(i) The series in (\ref{dfG0}) converges for $(\e,\J)\in
(\R^d/\Z^d)\times\{\J\in\C^d\;|\;|\Im\J|<\gamma_0\}$ and the
convergence is uniform on compact subsets. Moreover, if
$\,G_0^\prime(\e,\J,\gamma_0)$ is defined by
\begin{equation}\label{dfG0pr}
G_0(\e,\J,\gamma_0)=\sum_{\mb\in\Nc}\Ec_d(\e-\mb,\gamma_0)\exp(i\J\cdot\mb)+G_0^\prime(\e,\J,\gamma_0),
\end{equation}
where $\Nc:=\{\mb\in\Z^d\;|\;|\mb|\le\sqrt{d}\}$, then
$G_0^\prime(\cdot,\J,\gamma_0)\in C(2\Omega)$ and the mapping
$\J\rightarrow G_0^\prime(\cdot,\J,\gamma_0)\in C(2\Omega)$ is
holomorphic for $|\Im\J|<\gamma_0$.

Furthermore, if $g\ge 1$ for $d\le 2$ and $q\in[1,\frac{d}{d-2})$
for $d>2$, then \vskip2mm

(ii) each term of the series
$\sum_{\mb\in\Z^d}\Ec_d(\e-\bs-\mb,\gamma_0)\exp(i\J\cdot\mb)$
belongs to the space $L_{c,q}(\Omega\times\Omega)$ for any fixed
$\J\in C^d$ and it is holomorphic in $\C^d$ w.r. to $\J$ in the
$L_{c,q}(\Omega\times\Omega)$-norm; \vskip2mm

(iii) this series converges in the
$L_{c,q}(\Omega\times\Omega)$-norm to a kernel
$G_0(\e,\bs,\J,\gamma_0)=G_0(\e-\bs,\J,\gamma_0)\in
L_{c,q}(\Omega\times\Omega)$ uniformly w.r. to
$\J\in\Pi_\gamma=\{\J\in\C^d\;|\;|\Im p|<\gamma\}\;(\gamma\in
(0,\gamma_0))$, hence the mapping $\J\rightarrow
G_0(\e,\bs,\J,\gamma_0)\in L_{c,q}(\Omega\times\Omega)$ is
holomorphic in the strip $\Pi_{\gamma_0}$; \vskip2mm

(iv) for any $\J\in\T^d$ the restriction of the operator
$R_0(\J)=(-\Delta+\gamma_0^2I)^{-1}$ on the set $\B_{\J,q^\prime}$
has the form: $R_0(\J) f=\int_\Omega
G_0(\e,\bs,p,\gamma_0)f(\bs)\,d\bs$;
\vskip2mm

(v) the operator $R_0(\J)$ maps the set $\B_{\J,q^\prime}$ into
$\B_{\J,\infty}$.
\end{lemma}
\begin{lemma}\label{lmholcomp}
Let $\Pi$ be a domain in $\C^d$. If $V\in
L_{q^\prime}(\Omega)\;(q^\prime\in(1,\infty]$ and
$G_k(\e,\bs,\J)\;(k=0,1)$ be functions defined in
$\Omega\times\Omega\times\Pi$ such that for any fixed $\J\in\Pi$
$G_k(\cdot,\cdot,\J)\in L_{c,q}(\Omega\times\Omega)\;(k=0,1)$ and
the mappings $\J\rightarrow G_k(\cdot,\cdot,\J)\in
L_{c,q}(\Omega\times\Omega)\;(k=0,1)$ are holomorphic in $\Pi$. Then
the function
$G_2(\e,\bs,\J)=\int_{\Omega}G_0(\e,\xi,\J)V(\xi)G_1(\xi,\bs,\J)$
has the same properties, i.e. for any fixed $\J\in\Pi$
$G_2(\cdot,\cdot,\J)\in L_{c,q}(\Omega\times\Omega)$ and the mapping
$\J\rightarrow G_2(\cdot,\cdot,\J)\in L_{c,q}(\Omega\times\Omega)$
is holomorphic in $\Pi$.
\end{lemma}
\begin{proof}
The claim follows from the estimates:
\begin{eqnarray*}
 &&\hskip-8mm\|G_2(\e,\cdot,\J)\|_q\le
\Big(\int_\Omega\Big(\int_\Omega|G_0(\e,\xi,\J)||V(\xi)|
 |G_1(\xi,\bs,\J)|\,d\xi\Big)^q\,d\bs\Big)^{\frac{1}{q}}\le\\
&&\hskip-8mm\|V\|_{q^\prime}\|G_0(\e,\cdot,\J)\|_q
\|G_1(\cdot,\cdot,\J)\|_{c,q},
\end{eqnarray*}
\begin{eqnarray*}
 &&\hskip-8mm\|G_2(\e,\cdot,\J)-G_2(\e_0,\cdot,\J)\|_q\le\\
&&\hskip-8mm\Big(\int_\Omega\Big(\int_\Omega|G_0(\e,\xi,\J)-G_0(\e_0,\xi,\J)||V(\xi)|
 |G_1(\xi,\bs,\J)|\,d\xi\Big)^q\,d\bs\Big)^{\frac{1}{q}}\le\\
&&\hskip-8mm\|V\|_{q^\prime}\|G_0(\e,\cdot,\J)-G_0(\e_0,\cdot,\J)\|_q
\|G_1(\cdot,\cdot,\J)\|_{c,q},
\end{eqnarray*}
\begin{eqnarray*}
&&\hskip-7mm \|(G_2(\cdot,\cdot,\J_0+\hb)-G_2(\cdot,\cdot,\J_0))-
\int_\Omega\left(\nabla_\J
G_0(\cdot,\xi,\J)|_{\J=\J_0}\cdot\hb\times\right.\\
&&\hskip-7mm
\left.V(\xi)G_1(\xi,\cdot,\J_0)+G_0(\cdot,\xi,\J_0)V(\xi)\nabla_\J
G_1(\xi,\cdot,\J)_{\J=\J_0}\right)\,d\xi\|_{c,q}\le\\
&&\hskip-7mm\|V\|_{q^\prime}\left(\|(G_0(\cdot,\cdot,\J_0+\hb)-G_0(\cdot,\cdot,\J_0))-
\nabla_\J
G_0(\cdot,\cdot,\J)|_{\J=\J_0}\cdot\hb\|_{c,q}\right.\\
&&\hskip-7mm \times\|G_1(\cdot,\cdot,\J_0)\|_{c,q}+
\|G_0(\cdot,\cdot,\J_0+\hb)-G_0(\cdot,\cdot,\J_0)\|_{c,q}\times\\
&&\hskip-7mm\|G_1(\cdot,\cdot,\J_0+\hb)-G_1(\cdot,\cdot,\J_0)\|_{c,q}+
\|G_0(\cdot,\cdot,\J_0+\hb,\gamma_0)\|_{c,q}\times\\
&&\hskip-7mm\left.\|(G_1(\cdot,\cdot,\J_0+\hb)-G_1(\cdot,\cdot,\J_0))-\nabla_\J
G_1(\cdot,\cdot,\J)|_{\J=\J_0}\cdot\hb\|_{c,q}\right),
\end{eqnarray*}
where $\e,\e_0\in\Omega$ and $\J,\J_0,\J_0+\hb\in\Pi$.
\end{proof}

The proof of the following proposition is based on the previous
claims and it is analogous to the proof of Lemma 3.5 from
\cite{Wil}, only again instead of the $L_2(\Omega\times\Omega)$-norm
one should use the $L_{c,q}(\Omega\times\Omega)$-norm:

\begin{proposition}\label{prNeuser}
If $d\ge 4$ and $V\in L_{q^\prime}(\Omega)$ with
$q^\prime\in(\frac{d}{2},\infty]$, then for a large enough
$\gamma_0$:
\vskip2mm

(i) $-\gamma_0^2\in\Rs(H(\J))$ for any $\J\in\R^d$ and the resolvent
$R(\J)=(H(\J)+\gamma_0^2I)^{-1}$ is represented by the Neumann's
series
\begin{equation*}
R(\J)=(I+L_0(\J))^{-1}R_0(\J)=\sum_{n=0}^\infty(-1)^n(L_0(\J))^n
R_0(\J),
\end{equation*}
where $R_0(\J)=(-\Delta+\gamma_0^2I)^{-1}$, $L_0(\J):=R_0(\J)V$ and
this series converges to $R(\J)$ in $\f(L_2(\Omega))$-norm;
\vskip2mm

(ii) for any $\J\in\R^d$ the restriction of $R(\J)$ on
$\B_{\J,\infty}$ is the integral operator of the form:
$R(\J)f=\int_\Omega G(\e,\bs,\J,\gamma_0)f(\bs)\,d\bs\quad
(f\in\B_{\J,\infty})$, where $G(\cdot,\cdot,\J,\gamma_0)\in
L_{c,q}(\Omega\times\Omega)$ $(q^{-1}+(q^\prime)^{-1}=1)$,
\begin{equation}\label{NeuserG}
G(\e,\bs,\J,\gamma_0)=G_0(\e,\bs,\J,\gamma_0)+\sum_{n=1}^\infty
G_n(\e,\bs,\J,\gamma_0),
\end{equation}
\begin{eqnarray*}
&&G_{n+1}(\e,\bs,\J,\gamma_0)=
\nonumber\\
&&-\int_\Omega
G_0(\e,\bxi,\J,\gamma_0)V(\bxi)G_n(\bxi,\bs,\J,\gamma_0)\,d\bxi\quad
(n=0,1.\dots),
\end{eqnarray*}
$G_n(\e,\bs,\J,\gamma_0)\in L_{c,q}(\Omega\times\Omega)$ and the
series in (\ref{NeuserG}) converges to $G(\e,\bs,\J,\gamma_0)$ in
the $L_{c,q}(\Omega\times\Omega)$-norm;
\vskip2mm

(iii) each mapping $\J\rightarrow G_n(\cdot,\cdot,\J,\gamma_0)\in
L_{c,q}(\Omega\times\Omega)$ admits a holomorphic continuation from
$\R^d$ to the strip
$\Pi_{\frac{\gamma_0}{2}}:=\{\J\in\C^d\;|\;|\Im\J|<\frac{\gamma_0}{2}\}$
and the series in (\ref{NeuserG}) converges in the
$L_{c,q}(\Omega\times\Omega)$-norm uniformly in
$\Pi_{\frac{\gamma_0}{2}}$. Hence the mapping $\J\rightarrow
G(\cdot,\cdot,\J,\gamma_0)\in L_{c,q}(\Omega\times\Omega)$ admits a
holomorphic continuation from $\R^d$ to $\Pi_{\frac{\gamma_0}{2}}$;
\vskip2mm

(iv) for any $\J\in\Pi_{\frac{\gamma_0}{2}}$ the operator $R(\J)$
with the integral kernel $G(\e,\bs,\J,\gamma_0)$ maps the set
$\B_{\J,\infty}$ into itself; \vskip2mm

(v) for any $\J\in\Pi_{\frac{\gamma_0}{2}}$ and $\e,\bs\in\Omega$
the equalities are valid:
$G_0(\e,\bs,\J,\gamma_0)=G_0(\bs,\e,-\J,\gamma_0)$,
$G(\e,\bs,\J,\gamma_0)=G(\bs,\e,-\J,\gamma_0)$.
\end{proposition}

\begin{corollary}\label{corsymm}
If $d\ge 4$ and $\,V\in L_{q^\prime}(\Omega)$ with
$q^\prime>\frac{d}{2}$, then for a large enough $\gamma_0$ and any
$\J\in\Pi_{\frac{\gamma_0}{2}}\;$ $G_0(\cdot,\cdot,\J,\gamma_0),\,
G(\cdot,\cdot,\J,\gamma_0)\in L_{c,q}(\Omega\times\Omega)\cap
L_{q,c}(\Omega\times\Omega)$ and the mappings $\J\rightarrow
G_0(\cdot,\cdot,\J,\gamma_0)\in L_{c,q}(\Omega\times\Omega)$,
$\J\rightarrow G_0(\cdot,\cdot,\J,\gamma_0)\in
L_{q,c}(\Omega\times\Omega)$, $\J\rightarrow
G(\cdot,\cdot,\J,\gamma_0)\in L_{c,q}(\Omega\times\Omega)$ and
$\J\rightarrow G(\cdot,\cdot,\J,\gamma_0)\in
L_{q,c}(\Omega\times\Omega)$ are holomorphic in
$\Pi_{\frac{\gamma_0}{2}}$.
\end{corollary}

\subsection{$C(\Omega\times\Omega)$-holomorphy of a compositional power of the Green's function of $H(\J)$}
\label{subsec:contcomppow}

We shall use the following notation. By $K_1(\cdot,\cdot)\circ
K_2(\cdot,\cdot)$ we denote the composition of integral kernels
$K_1(\e,\bf)$ and $K_2(\e,\bf)$: $(K_1(\cdot,\cdot)\circ
K_2(\cdot,\cdot))(\e,\bs):=\int_\Omega
K_1(\e,\xi)K_2(\xi,\bs)\,d\xi$. The composition of several integral
kernels $K_1(\cdot,\cdot)\circ K_2(\cdot,\cdot)\circ\cdots\circ
K_N(\cdot,\cdot)$ will be denoted briefly by $\circ_{j=1}^N
K_j(\cdot,\cdot)$; if
$K_1(\cdot,\cdot)=K_2(\cdot,\cdot)=\cdots=K_N(\cdot,\cdot)=K(\cdot,\cdot)$,
we shall denote it as a compositional power $K(\cdot,\cdot)^{\circ
N}$. We shall denote by $Hol(D,B)$ the set of abstract holomorphic
functions $\phi:D\rightarrow B$, where $D$ is an open domain in
$\C^d$ and $B$ is a complex Banach space. Like in \cite{Wil}, let us
define the following equivalence relation between the functions
$\J\rightarrow K(\e,\bs,\J)\;(\e,\bs\in\Omega)$: we shall write that
$K(\e,\bs,\J)\sim 0$, if $(\J\rightarrow K(\cdot,\cdot,\J))\in
Hol(\Pi_{\frac{\gamma_0}{2}}, C(\Omega\times\Omega))$ and
$K_1(\e,\bs,\J)\sim K_2(\e,\bs,\J)$, if
$K_1(\e,\bs,\J)-K_2(\e,\bs,\J)\sim 0$.

\begin{lemma}\label{lmcomp}
Assume that $d\ge 4$ and $V\in\L_{q^\prime}(\Omega)$ with
$q^\prime>\frac{d}{2}$. Consider the composition of integral kernels
\begin{equation}\label{comp}
K(\e,\bs,\J)=\circ_{j=1}^N K_j(\cdot,\cdot,\J)(\e,\bs)\quad (\J\in
\Pi_{\frac{\gamma_0}{2}}),
\end{equation}
where or $K_j(\e,\bs,\J)\sim 0$ and the composition contains at
least one term of this kind, or $K_j(\e,\bs,\J)=L_j(\e,\bs,\J)$, or
for $j<N$ $K_j(\e,\bs,\J)=L_j(\e,\bs,\J)V(\bs)$  with
\begin{equation}\label{mmbship}
(\J\rightarrow L_j(\cdot,\cdot,\J))\in Hol(\Pi_{\frac{\gamma_0}{2}},
L_{c,q}(\Omega\times\Omega))\cap Hol(\Pi_{\frac{\gamma_0}{2}},
L_{q,c}(\Omega\times\Omega)),
\end{equation}
where $q^{-1}+(q^\prime)^{-1}=1$. Then $K(\e,\bs,\J)\sim 0$.
\end{lemma}
\begin{proof}
We shall prove the lemma in two steps.

Step 1. Consider the kernel of the form
$K_1(\cdot,\cdot,\J)=L(\cdot,\cdot,\J)V(\cdot)\circ
K(\cdot,\cdot,\J)$, where the mapping $\J\rightarrow
L(\cdot,\cdot,\J)$ satisfies the condition (\ref{mmbship}) and
$K(\e,\bs,\J)\sim 0$. Let us prove that $K_1(\e,\bs,\J)\sim 0$. We
have for $\e,\e+\hb,\bs\in\Omega$:
\begin{eqnarray*}
&&|K_1(\e+\hb,\bs,\J)-K_1(\e,\bs,\J)|=\\
&&\left|\int_\Omega(L(\e+\hb,\xi,\J)-L(\e,\xi,\J))V(\xi)K(\xi,\bs,\J)\,d\xi\right|\le\\
&&\|V\|_{q^\prime}\|K(\cdot,\cdot,\J)\|_{C(\Omega\times\Omega)}
\|L(\e+\hb,\cdot,\J)-L(\e,\cdot,\J)\|_q.
\end{eqnarray*}
This estimate and the inclusion $L(\cdot,\cdot,\J)\in
L_{c,q}(\Omega\times\Omega)$ imply that for any
$\J\in\Pi_{\frac{\gamma_0}{2}}$ the function $K_1(\e,\bs,\J)$ is
continuous w.r. to $\e$ at each point $\e\in\Omega$ uniformly w.r.
to $\bs\in\Omega$. Let us estimate for $\e,\bs,\bs+\hb\in\Omega$:
\begin{eqnarray*}
&&|K_1(\e,\bs+\hb,\J)-K_1(\e,\bs,\J)|\le\\
&&\|V\|_{q^\prime}\|L(\cdot,\cdot,\J)\|_{c,q}\max_{\xi\in\Omega}|K(\xi,\bs+\hb,\J)-K(\xi,\bs,\J)|.
\end{eqnarray*}
This estimate and the continuity of $K(\xi,\bs,\J)$ in
$\Omega\times\Omega$ imply that for any fixed $\e\in\Omega$
$K_1(\xi,\bs,\J)$ is continuous at each point $\bs\in\Omega$. So, we
have proved that for any $\J\in\Pi_{\frac{\gamma_0}{2}}$
$K_1(\cdot,\cdot,\J)\in C(\Omega\times\Omega)$. The estimate
\begin{eqnarray*}
\hskip-8mm &&\|K_1(\cdot,\cdot,\J+\tb)-K_1(\cdot,\cdot,\J)-\nabla_\J
L(\cdot,\cdot,\J)\cdot\tb\,V(\cdot)\circ
K(\cdot,\cdot,\J)\|_{C(\Omega\times\Omega)}\le\|V\|_{q^\prime}\times\\
\hskip-8mm&&\left(\|(L(\cdot,\cdot,\J+\tb)-L(\cdot,\cdot,\J))-\nabla_\J
L(\cdot,\cdot,\J)\cdot\tb\|_{c,q}\,\|K(\cdot,\cdot,\J)\|_{C(\Omega\times\Omega)}+
\right.\\
\hskip-8mm&&\|L(\cdot,\cdot,\J+\tb)-L(\cdot,\cdot,\J)\|_{c,q}\,\|K(\cdot,\cdot,\J+\tb)-
K(\cdot,\cdot,\J)\|_{C(\Omega\times\Omega)}+
\|L(\cdot,\cdot,\J)\|_{c,q}\times\\
\hskip-8mm&&\left.\|(K(\cdot,\cdot,\J+\tb)-K(\cdot,\cdot,\J))-\nabla_\J
K(\cdot,\cdot,\J)\cdot\tb\|_{C(\Omega\times\Omega)}\right)\quad
(\J,\J+\tb\in\Pi_{\frac{\gamma_0}{2}})
\end{eqnarray*}
and the holomorphy of the mappings $\J\rightarrow
L(\cdot,\cdot,\J)\in L_{c,q}(\Omega\times\Omega)$ and $\J\rightarrow
K(\cdot,\cdot,\J)\in C(\Omega\times\Omega)$ in the strip
$\Pi_{\frac{\gamma_0}{2}}$ imply that the mapping $\J\rightarrow
K_1(\cdot,\cdot,\J)\in C(\Omega\times\Omega)$ is holomorphic there,
that is $K_1(\e,\bs,\J)\sim 0$. In the analogous manner we can prove
that if $L$ satisfies the condition (\ref{mmbship}) and
$K(\e,\bs,\J)\sim 0$, then $L(\cdot,\cdot,\J)\circ
K(\cdot,\cdot,\J)\sim 0$, $K(\cdot,\cdot,\J)\circ
V(\cdot)L(\cdot,\cdot,\J)\sim 0$ and $K(\cdot,\cdot,\J)\circ
L(\cdot,\cdot,\J)\sim 0$.

Step 2. By the assumption of the lemma, in the composition
(\ref{comp}) for some $j_0\in\{1,2,\dots,N\}$
$\;K_{j_0}(\e,\bs,\J)\sim 0$. Observe that if a term of (\ref{comp})
is equivalent to zero, it belongs to the set from the right hand
side of (\ref{mmbship}). Then using inductively the results of Step
1, we obtain that the composition $T(\e,\bs,\J)$ of kernels from
(\ref{comp}) which are before $K_{j_0}(\e,\bs,\J)$ and of
$K_{j_0}(\e,\bs,\J)$ itself has the property $T(\e,\bs,\J)\sim 0$.
On the other hand, the composition $\tilde T(\e,\bs,\J)$ of terms
from (\ref{comp}) which are after $K_{j_0}(\e,\bs,\J)$ can be
represented in the form $\tilde T(\e,\bs,\J)=(\circ_{j=j_0+1}^N
\tilde K_j(\cdot,\cdot,\J))(\e,\bs)$, where the kernels $\tilde
K_j(\cdot,\cdot,\J))$ have one of the forms: or $\tilde
K_j(\cdot,\cdot,\J))=V(\cdot)\tilde L_j(\cdot,\cdot,\J)$, or $\tilde
K_j(\cdot,\cdot,\J))=\tilde L_j(\cdot,\cdot,\J)$ and each $\tilde
L_j(\cdot,\cdot,\J)$ belongs to the set from the right hand side of
(\ref{mmbship}). Again using inductively the results of Step 1, we
obtain that $K(\e,\bs,\J)=(T(\cdot,\cdot,\J)\circ\tilde
T(\cdot,\cdot,\J))(\e,\bs)\sim 0$.
\end{proof}

In  what follows we need the following result about Fourier's
transform of a function with a polar singularity:

\begin{lemma}\label{lmestFour}
For the Fourier's transform $\hat\Fc(\omega)$ on $\R^d$of the
function
$\Fc(\e)=\frac{\exp(-\gamma|\e|)}{|\e|^{\beta}}\;(\beta\in[0,d),\,\gamma>0)$
the property is valid:
\begin{equation}\label{boundhatE1}
\sup_{\omega\in\R^d}\psi(\omega)|\hat\Fc(\omega)|<\infty,
\end{equation}
where
\begin{eqnarray}\label{defpsi}
\psi(\omega)=\left\{
\begin{array}{ll}
|\omega|,&\rm{if}\quad \beta\in[0,d-1),\\
\frac{\omega|}{1+\ln(1+|\omega|)},&\rm{if}\quad \beta=d-1,\\
|\omega|^{d-\beta},&\rm{if}\quad \beta\in(d-1,d).
\end{array}\right.
\end{eqnarray}
\end{lemma}

\begin{proof}
Since $\beta<d$, $\Fc\in L_1(\R^d)$, hence $\hat\Fc(\omega)$ is
bounded on $\R^d$. Hence it is sufficient to prove
(\ref{boundhatE1}) with $\{\e\in\R^d\,|\;|\e|\ge 1\}$ instead of
$\R^d$. Along with the Euclidean norm $|\omega|$ on $\R^d$ consider
the equivalent norm $|\omega|_\infty=\max_{1\le j\le
d}|\omega_j|\;(\omega=(\omega_1,\omega_2,\dots,\omega_d))$. Let us
take $\omega\neq 0$ and $j_0\in\{1,2,\dots,d\}$ such that
$|\omega_{j_0}|=|\omega|_\infty$. If the vector $\hbf(\omega)$ is
defined by
\begin{eqnarray*}
(\hbf(\omega))_j=\left\{\begin{array}{ll}
\frac{\pi}{|\omega|_\infty}&\rm{for}\quad j=j_0,\\
0&\rm{for}\quad j\neq j_0
\end{array}\right.
\end{eqnarray*}
$(j\in\{1,2,\dots,d\})$, then the equality
$\exp(-i\omega\cdot(\e+\hbf(\omega)))=-\exp(-i\omega\cdot\e)$ is
valid, hence we have the representation
\begin{eqnarray*}
&&\hat\Fc(\omega)=\frac{1}{2(2\pi)^{\frac{d}{2}}}\Big(\int_{\R^d}\Fc(\e)\exp(-i\omega\cdot\e)\,d\e+\\
&&\int_{\R^d}\Fc(\e+\hbf(\omega))\exp(-i\omega\cdot(\e+\hbf(\omega)))\,d\e\Big)=\\
&&\frac{1}{2(2\pi)^{\frac{d}{2}}}\int_{\R^d}\big(\Fc(\e)-\Fc(\e+\hbf(\omega))\big)\exp(-i\omega\cdot\e)\,d\e,
\end{eqnarray*}
which implies the estimate:
\begin{equation}\label{Fourmodcont}
|\Fc(\omega)|\le\frac{1}{2(2\pi)^{\frac{d}{2}}}\|\Fc(\e)-\Fc(\e+\hbf(\omega))\|_1.
\end{equation}
Here we denote by $\|\cdot\|_1$ the norm in the space $L_1(\R^d)$.
We have:
\begin{equation}\label{estmodcont}
\|\Fc(\e)-\Fc(\e+\hbf(\omega))\|_1\le I_1(\omega)+I_2(\omega),
\end{equation}
where
\begin{equation*}
I_1(\omega)=\int_{|\e|\le 1}|\Fc(\e)-\Fc(\e+\hbf(\omega))|\,d\e,
\end{equation*}
\begin{equation}\label{dfI2om}
I_2(\omega)=\int_{|\e|\ge 1}|\Fc(\e)-\Fc(\e+\hbf(\omega))|\,d\e.
\end{equation}
Let us derive the change of the variable in the integral
$I_1(\omega)$: $\by=\frac{\e}{|\hbf(\omega)|}$, denote
$\tilde\hbf=\frac{\hbf(\omega)}{|\hbf(\omega)|}$ and rotate the
space $\R^d$ such that $\tilde\hbf=(1,0,\dots,0)$. Then we obtain:
\begin{eqnarray}\label{rprI1om}
&&\hskip-8mmI_1(\omega)=|\hbf(\omega)|^{d-\beta}\int_{|\by|\le\frac{1}{|\hbf(\omega)|}}
\Big|\frac{\exp(-\gamma|\hbf(\omega)||\by+\tilde\hbf|)}{|\by+\tilde\hbf|^\beta}-
\frac{\exp(-\gamma|\hbf(\omega)||\by|)}{|\by|^\beta}\Big|\,d\by\le\nonumber\\
&&\hskip-8mm\pi|\omega|_\infty\big(I_{1,1}+I_{1,2}(\omega)\big),
\end{eqnarray}
where
\begin{equation}\label{dfI11}
I_{1,1}=\int_{|\by|\le 2} \Big(\frac{1}{|\by+\tilde\hbf|^\beta}+
\frac{1}{|\by|^\beta}\Big)\,d\by,
\end{equation}
\begin{equation}\label{dfI12om}
I_{1,2}(\omega)=\int_{2\le|\by|\le\frac{1}{|\hbf(\omega)|}}
\Big|\frac{\exp(-\gamma|\hbf(\omega)||\by+\tilde\hbf|)}{|\by+\tilde\hbf|^\beta}-
\frac{\exp(-\gamma|\hbf(\omega)||\by|)}{|\by|^\beta}\Big|\,d\by.
\end{equation}
Let us represent:
\begin{eqnarray}\label{rprdiff}
&&\hskip-15mm\frac{\exp(-\gamma|\hbf(\omega)||\by+\tilde\hbf|)}{|\by+\tilde\hbf|^\beta}-
\frac{\exp(-\gamma|\hbf(\omega)||\by|)}{|\by|^\beta}=\\
&&\hskip-15mm\int_0^1 \frac{\partial}{\partial
t}\Big(\frac{\exp(-\gamma|\hbf(\omega)||\by+t\tilde\hbf|)}{|\by+t\tilde\hbf|^\beta}\Big)\,dt=
\nonumber\\
&&\hskip-15mm-\int_0^1\Big(\beta\frac{\exp(-\gamma|\hbf(\omega)||\by+t\tilde\hbf|)}{|\by+t\tilde\hbf|^{\beta+1}}+
\gamma|\hbf(\omega)|\frac{\exp(-\gamma|\hbf(\omega)||\by+t\tilde\hbf|)}{|\by+t\tilde\hbf|^\beta}\Big
)\times\nonumber\\
&&\hskip-15mm\frac{(\by+t\tilde\hbf)\cdot\tilde\hbf}{|\by+t\tilde\hbf|}\,dt.\nonumber
\end{eqnarray}
Then we have from (\ref{dfI12om}):
\begin{eqnarray*}
&&I_{1,2}(\omega)\le\int_{2\le|\by|\le\frac{1}{|\hbf(\omega)|}}\Big(\frac{\beta}{(|\by|-1)^{\beta+1}}+
\frac{\gamma|\hbf(\omega)|}{(|\by|-1)^\beta}\Big)\,d\by=\\
&&s_{d-1}\Big(\beta\int_2^{\frac{1}{|\hbf(\omega)|}}\frac{r^{d-1}\,dr}{(r-1)^{\beta+1}}+
\gamma|\hbf(\omega)|\int_2^{\frac{1}{|\hbf(\omega)|}}\frac{r^{d-1}\,dr}{(r-1)^\beta}\Big),
\end{eqnarray*}
hence for $\beta\in[0,d)\setminus\{d-1\}$
\begin{equation}\label{estI12om}
I_{1,2}(\omega)\le s_{d-1}
2^{d-1}\Big(|\hbf(\omega)|^{\beta-d+1}+1\Big)\Big(\frac{\beta}{\beta-d+1}+\frac{\gamma}{d-\beta}\Big)
\end{equation}
and for $\beta=d-1$
\begin{equation}\label{estI12om1}
I_{1,2}(\omega)\le s_{d-1}
2^{d-1}\Big(\beta\ln\Big(\frac{1}{|\hbf(\omega)|}\Big)+\gamma\Big).
\end{equation}
Let us estimate the integral $I_2(\omega)$, defined by
(\ref{dfI2om}), using for the difference under it the
representation, analogous to (\ref{rprdiff}):
\begin{eqnarray*}
&&I_2(\omega)\le|\hbf(\omega)|\int_{|\e|\ge 1}
\Big(\frac{\beta}{(|\e|-|\hbf(\omega)|)^{\beta+1}}+\frac{\gamma}{(|\e|-|\hbf(\omega)|)^\beta}\Big)\times\\
&&\exp\big(-\gamma(|\e|-|\hbf(\omega)|)\big)\,d\e,
\end{eqnarray*}
if $|\hbf(\omega)|<1$. From this estimate and from
(\ref{Fourmodcont}), (\ref{estmodcont}), (\ref{rprI1om}),
(\ref{dfI11}), (\ref{estI12om}) and (\ref{estI12om1}) we obtain the
desired claim.
\end{proof}

In  what follows we need also the following claim:

\begin{lemma}\label{lmcmpEc}
Assume that $V\in L_{q^\prime}(\Omega)$ with $q^\prime>\frac{d}{2}$
and denote
\begin{equation}\label{defl}
l:=\left[\frac{d}{\theta}\right]+2,
\end{equation}
where
\begin{equation}\label{deftht}
\theta=\min\{1, d-q(d-2)\}.
\end{equation}
Then the composition of kernels
\begin{equation}\label{cmpEc}
\Theta(\e,\bs)=\left(\circ_{j=1}^l\Ec_j(\cdot,\cdot)\right)(\e,\bs),
\end{equation}
where $\tilde\Ec_l(\e,\bs)=\Ec_d(\e-\bs-\mb_l,\gamma_0)$ and for
$j<l$
\begin{displaymath}
\tilde\Ec_j(\e,\bs)=\left\{\begin{array}{ll} or&\Ec_d(\e-\bs-\mb_j,\gamma_0),\\
or& \Ec_d(\e-\bs-\mb_j,\gamma_0)V(\bs)
\end{array}\right.
\end{displaymath}
($\mb_j\in\Z^d$), is continuous in $\Omega\times\Omega$.
\end{lemma}
\begin{proof}
Denote
$\tilde\Theta(\e,\bs)=\left(\circ_{j=1}^{l-2}\tilde\Ec_j(\cdot,\cdot)\right)\circ\Ec_d(\e-\bs-\mb_{l-1},\gamma_0)(\e,\bs)$
and
\begin{equation}\label{dfbrTht}
\bar\Theta(\e,\bs)=\left(\circ_{j=2}^l\tilde\Ec_j(\cdot,\cdot)\right)(\e,\bs).
\end{equation}
Then we have from (\ref{cmpEc}):
\begin{equation}\label{rprTht1}
\Theta(\e,\bs)=(\tilde\Theta(\cdot,\cdot)\circ
W_{l-1}(\cdot)\circ\Ec_d(\cdot-\cdot-\mb_l,\gamma_0))(\e,\bs),
\end{equation}
\begin{equation}\label{rprTht2}
\Theta(\e,\bs)=(\Ec_d(\cdot-\cdot-\mb_1,\gamma_0)\circ
W_1(\cdot)\circ\bar\Theta(\cdot,\cdot))(\e,\bs),
\end{equation}
where
\begin{displaymath}
W_j(\xi_j)=\left\{\begin{array}{ll} or& V(\xi_j),\\
or& 1,
\end{array}\right.
\end{displaymath}
Assume that $d$ is odd. From the explicit formula
\begin{eqnarray*}
&&\tilde\Theta(\e,\bs)=\int_\Omega\Ec_d(\e-\xi_1-\mb_1,\gamma_0)W_1(\xi_1)\,d\xi_1
\int_\Omega\Ec_d(\xi_1-\xi_2-\mb_2,\gamma_0)W_2(\xi_2)\,d\xi_2\dots\\
&&W_{l-2}(\xi_{l-2})\int_\Omega\Ec_d(\xi_{l-2}-\bs-\mb_{l-1},\gamma_0)\,d\xi_{l-2}
\end{eqnarray*}
and claims (ii)-(v) of Proposition \ref{prfndsl} we obtain that
there exists $A>0$ such that the
 estimate is valid:
\begin{eqnarray}\label{estcmp}
&&\hskip-8mm|\tilde\Theta(\e,\bs)|\le\max\{\|1\|_{q^\prime},\|V\|_{q^\prime}\}
\Big(\int_\Omega|\Ec_d(\e-\xi_1-\mb_1,\gamma_0)|^q \,d\xi_1\times\nonumber\\
&&\hskip-8mm\Big|\int_\Omega\Ec_d(\xi_1-\xi_2-\mb_2,\gamma_0)W_2(\xi_2)\,d\xi_2
\dots
W_{l-2}(\xi_{l-2})\times\nonumber\\
&&\hskip-8mm\int_\Omega\Ec_d(\xi_{l-2}-\bs-\mb_{l-1},\gamma_0)\,d\xi_{l-2}\Big|^q\Big)^{\frac{1}{q}}\le
(\max\{\|1\|_{q^\prime},\|V\|_{q^\prime}\})^2\times\nonumber\\
&&\hskip-8mm\Big(\int_\Omega|\Ec_d(\e-\xi_1-\mb_1,\gamma_0)|^q \,d\xi_1\times\nonumber\\
&&\hskip-8mm\int_\Omega|\Ec_d(\xi_1-\xi_2-\mb_2,\gamma_0)|^q\,d\xi_2\Big|\int_\Omega
d\xi_3 \dots
W_{l-2}(\xi_{l-2})\times\nonumber\\
&&\hskip-8mm\int_\Omega\Ec_d(\xi_{l-2}-\bs-\mb_{l-1},\gamma_0)\,d\xi_{l-2}\Big|^q\Big)^{\frac{1}{q}}\le\dots\le\nonumber\\
&&\hskip-8mm
A\cdot(\max\{\|1\|_{q^\prime},\|V\|_{q^\prime}\})^{l-2}(\phi(\e-\bs)^{\frac{1}{q}},
\end{eqnarray}
where
\begin{equation}\label{conv}
\phi(\e)=(\Fc(\cdot-\mb_1)\star\Fc(\cdot-\mb_2)\star\dots\star
\Fc(\cdot-\mb_{l-1}))(\e)
\end{equation}
and $\Fc(\e)=\frac{\exp(-\frac{\gamma_0 q}{2}|\e|)}{|\e|^{q(d-2)}}$
(here $\star$ is the sign of the convolution of functions defined on
$\R^d$). Let us show that the function $\phi(\cdot,\gamma)$ is
bounded on $\R^d$. To this end let us calculate the Fourier
transform on $\R^d$ of the function $\phi(\e)$. We have from
(\ref{conv}):
\begin{eqnarray*}
\hat\phi(\omega)=\exp\left(-i\omega\cdot\sum_{j=1}^{l-1}\mb_j\right)
(\hat\Fc(\omega))^{l-1}\quad (\omega\in\R^d),
\end{eqnarray*}
hence by Lemma \ref{lmestFour},
$\sup_{\omega\in\R^d}(\psi(\omega))^{l-1}|\hat\phi(\omega)|<\infty$,
where $\psi(\omega)$ is defined by (\ref{defpsi}) with
$\beta=q(d-2)$. Since in view of (\ref{defl}) and (\ref{deftht}),
$(l-1)\theta>d$, these circumstances imply that $\hat\phi\in
L_1(\R^d)$. Hence the function $\phi(\e)$ is bounded in $\R^d$ and
we obtain from (\ref{estcmp}) that for an odd $d$  the function
$\tilde\Theta(\e,\bs)$ is bounded in $\Omega\times\Omega$. Consider
the case of an even $d$. From the formula (\ref{mcdn}) for the
MacDonald's function, we get that for any $\epsilon>0$
\begin{equation*}
\frac{K_\nu(\gamma_0|\e|)}{|\e|^\nu}\le
C(\epsilon)\frac{\exp(-\frac{\gamma_0|\e|}{2})}{|\e|^{\nu+\epsilon}}\quad
(\nu\ge 0),
\end{equation*}
where
$C(\epsilon)=\sup_{\tau\in[0,\infty)}\big(\tau^\epsilon\exp(-\gamma_0\tau/2)\big)\int_1^\infty\frac{dt}
{t^{\nu+\epsilon}\sqrt{t^2-1}}$. Using claims (vi)-(vii) of
Proposition \ref{prfndsl} and choosing a small enough $\epsilon>0$,
we can show in the same manner as above that also for an even $d$
the function $\tilde\Theta(\e,\bs)$ is bounded in
$\Omega\times\Omega$. In the analogous manner we obtain that the
function $\bar\Theta(\e,\bs)$, defined by (\ref{dfbrTht}), is
bounded in $\Omega\times\Omega$.

Let us prove that the function $\Theta(\e,\bs)$ is continuous in
$\Omega\times\Omega$. We have from (\ref{rprTht1}) for
$\e,\bs,\bs+\hb\in\Omega$:
\begin{eqnarray*}
&&\hskip-5mm|\Theta(\e,\bs+\hb)-\Theta(\e,\bs)|=\\
&&\hskip-5mm\left|\int_\Omega\tilde\Theta(\e,\xi)
(\Ec_d(\xi-\bs-\hb-\mb_l,\gamma_0)-\Ec_d(\xi-\bs-\mb_l,\gamma_0))\,d\xi\right|\le
\|\tilde\Theta(\cdot,\cdot)\|_{L_\infty(\Omega\times\Omega)}\\
&&\hskip-5mm\times\max\{\|1\|_{q^\prime},\,\|V\|_{q^\prime}\}
\|\Ec_d(\cdot-\bs-\hb-\mb_l,\gamma_0)-\Ec_d(\cdot-\bs-\mb_l,\gamma_0)\|_q.
\end{eqnarray*}
Since by claim (ii) of Lemma \ref{lmcnvser} the function
$\Ec_\mb(\cdot,\cdot)=\Ec_d(\cdot-\cdot-\mb,\gamma_0)$ belongs to
the class $L_{q,c}(\Omega\times\Omega)$, then the latter estimate
imply that the function $\Theta(\e,\bs)$ is continuous w.r. to $\bs$
at each point $\bs\in\Omega$ uniformly w.r. to $\e\in\Omega$. In the
analogous manner we prove, using (\ref{rprTht2}), that for any fixed
$\bs\in\Omega$ the function $\Theta(\e,\bs)$ is continuous in
$\Omega$ w.r. to $\e$. These circumstances mean that the function
$\Theta(\e,\bs)$ is continuous in $\Omega\times\Omega$. The lemma is
proven.
\end{proof}

We now turn to the main result of this section.

\begin{proposition}\label{prpwkrn}
If $d\ge 4$, $V\in L_{q^\prime}(\Omega)$ with
$q^\prime>\frac{d}{2}$, $R(\J)=(H(\J)+\gamma_0^2)^{-1}$ and the
natural $l$ is defined by (\ref{defl}), (\ref{deftht}), then for a
large enough $\gamma_0$ the operator $K(\J)=(R(\J))^l$ has the
integral kernel $K(\e,\bs,\J)$ such that $K(\e,\bs,\J)\sim 0$.
\end{proposition}
\begin{proof}
By Proposition \ref{prNeuser}, for a large enough $\gamma_0$ and
$\J\in\Pi_{\frac{\gamma_0}{2}}$ the operator maps the set
$\B_{\J,\infty}$ into itself and for any $f\in\B_{\J,\infty}$
$R(\J)f=\int_\Omega G(\e,\bs,\J)f(\bs)\,d\bs$, where
\begin{equation}\label{serofcmp}
G(\e,\bs,\J)=G_0(\e,\bs,\J)+\left(\sum_{k=1}^\infty
(G_0(\cdot,\cdot,\J)V(\cdot))^{\circ k}\circ
G_0(\cdot,\cdot,\J)\right)(\e,\bs)
\end{equation}
and the latter series converges in the
$L_{c,q}(\Omega\times\Omega)$-norm uniformly w.r. to
$\J\in\Pi_{\frac{\gamma_0}{2}}$. Then for any $f\in\B_{\J,\infty}$
\begin{equation}\label{intpwofR}
K(\J)f=\int_\Omega K(\e,\bs,\J)f(\bs)\,d\bs,
\end{equation}
where
\begin{equation}\label{pwkrn}
K(\e,\bs,\J)=\left(G(\cdot,\cdot,\J)^{\circ l}\right)(\e,\bs).
\end{equation}
Let us represent (\ref{serofcmp}) in the form:
\begin{eqnarray}\label{rpserofcmp}
&&G(\e,\bs,\J)=G_0(\e,\bs,\J)+\left(\sum_{k=1}^{l-1}
(G_0(\cdot,\cdot,\J)V(\cdot))^{\circ k}\circ
G_0(\cdot,\cdot,\J)\right)(\e,\bs)+\nonumber\\
&&((G_0(\cdot,\cdot,\J)V(\cdot))^{\circ (l-1)}\circ
G_0(\cdot,\cdot,\J)\circ V(\cdot)G(\cdot,\cdot,\J))(\e,\bs).
\end{eqnarray}
On the other hand, by claim (i) of Lemma \ref{lmcnvser} the
representation (\ref{dfG0pr}) is valid, in which
$G_0^\prime(\e,\bs,\J)\sim 0$ and
$\Nc:=\{\mb\in\Z^d\;|\;|\mb|\le\sqrt{d}\}$. Substituting
(\ref{dfG0pr}) into (\ref{rpserofcmp}) and using Corollary
\ref{corsymm} and Lemma \ref{lmcomp}, we obtain:
\begin{eqnarray*}
&&\hskip-8mm G(\e,\bs,\J)\sim
\sum_{\mb\in\Nc}\Ec_d(\e-\bs-\mb,\gamma_0)\exp(i2\pi\J\cdot\mb)+\\
&&\hskip-8mm\sum_{k=1}^{l-2}\;\sum_{\mb_1,\dots,\mb_{k+1}\in\Nc}
\exp\left(i2\pi\J\cdot\sum_{j=1}^{k+1}\mb_j\right)\left(\circ_{j=1}^k
\Ec_d(\cdot-\cdot-\mb_j,\gamma_0)V(\cdot)\right)\circ\\
&&\hskip-8mm\Ec_d(\cdot-\cdot-\mb_{k+1},\gamma_0)(\e,\bs)+\sum_{\mb_1,\dots,\mb_l\in\Nc}
\exp\left(i2\pi\J\cdot\sum_{j=1}^l\mb_j\right)\Theta_{\mb_1,\dots,\mb_l}(\e,\bs)+\\
&&\hskip-8mm\sum_{\mb_1,\dots,\mb_l\in\Nc}
\exp\left(i2\pi\J\cdot\sum_{j=1}^l\mb_j\right)\tilde\Theta_{\mb_1,\dots,\mb_l}(\cdot,\cdot)\circ
V(\cdot)G(\cdot,\cdot,\J)(\e,\bs),
\end{eqnarray*}
where $\Theta_{\mb_1,\dots,\mb_l}(\e,\bs)$ and
$\tilde\Theta_{\mb_1,\dots,\mb_l}(\e,\bs)$ are functions of the form
described in the formulation of Lemma \ref{lmcmpEc}, hence they are
continuous in $\Omega\times\Omega$. Then by Lemma \ref{lmcomp} and
Corollary \ref{corsymm},
$\Theta_{\mb_1,\dots,\mb_l}(\cdot,\cdot)\circ
V(\cdot)G(\cdot,\cdot,\J)(\e,\bs)\sim 0$. Hence
\begin{eqnarray*}
&&\hskip-8mm G(\e,\bs,\J)\sim
\sum_{\mb\in\Nc}\Ec_d(\e-\bs-\mb,\gamma_0)\exp(i2\pi\J\cdot\mb)+\\
&&\hskip-8mm\sum_{k=1}^{l-2}\;\sum_{\mb_1,\dots,\mb_{k+1}\in\Nc}
\exp\left(i2\pi\J\cdot\sum_{j=1}^{k+1}\mb_j\right)\left(\circ_{j=1}^k
\Ec_d(\cdot-\cdot-\mb_j,\gamma_0)V(\cdot)\right)\circ\\
&&\hskip-8mm\Ec_d(\cdot-\cdot-\mb_{k+1},\gamma_0)(\e,\bs).
\end{eqnarray*}
Then we obtain from (\ref{pwkrn}), claim (i) of Lemma \ref{lmcnvser}
and Lemma \ref{lmcomp}:
\begin{eqnarray*}
&&\hskip-8mm K(\e,\bs,\J)\sim\\
&&\hskip-8mm
\left(\sum_{\mb\in\Nc}\Ec_d(\cdot-\cdot-\mb,\gamma_0)\exp(i2\pi\J\cdot\mb)+
\sum_{k=1}^{l-2}\;\sum_{\mb_1,\dots,\mb_{k+1}\in\Nc}
\exp\left(i2\pi\J\cdot\sum_{j=1}^{k+1}\mb_j\right)\times\right.\\
&&\hskip-8mm \left.\left(\circ_{j=1}^k
\Ec_d(\cdot-\cdot-\mb_j,\gamma_0)V(\cdot)\right)\circ
\Ec_d(\cdot-\cdot-\mb_{k+1},\gamma_0)\right)^{\circ l}(\e,\bs).
\end{eqnarray*}
The latter expression is a trigonometric polynomial w.r.t. $\J$,
whose coefficients are kernels satisfying the condition of Lemma
\ref{lmcmpEc}, hence they are continuous in $\Omega\times\Omega$.
Hence $K(\e,\bs,\J)\sim 0$. Since for any fixed
$\J\in\Pi_{\frac{\gamma_0}{2}}$ the operator $K(\J)$ is bounded in
$L_2(\Omega)$, the operator with the continuous integral kernel
$K(\e,\bs,\J)$ in the right hand side of (\ref{intpwofR}) is bounded
in $L_2(\Omega)$ too and the set $\B_{\J,\infty}$ is dense in
$L_2(\Omega)$, then the representation (\ref{intpwofR}) is valid for
any $f\in L_2(\Omega)$. The proposition is proven.
\end{proof}

\subsection{Proof of Theorem \ref{thmainApp}}
\label{subsec:proofmainclaim}

\begin{proof}
(i) Let us choose $\gamma_0>0$ and consider the domain
$\Pi_{\gamma_0/2}=\{\J\in \C^d:\;|\Im(\J)|\le\gamma_0/2\}$.
Furthermore, let us take $l=2$ for $d\le 3$ and assume that $l$ is
defined by (\ref{defl}), (\ref{deftht}) for $d\ge 4$.
 Like in \cite{Wil} (Lemmas 3.6 and
3.7), it is easy to show that if $\gamma_0$ is large enough, for
$\J\in\R^d$ the number $\nu(\J)=(\lambda(\J)+\gamma_0^2)^l$ is a
singular value of the operator
$K(\J)=(R(\J))^l\;(R(\J)=(H(\J)+\gamma_0^2)^{-1})$ with the
eigenfunction $b(\e,\J)$. This means that for $\J\in\R^d$ the
equality is valid:
\begin{equation}\label{eqsingvl}
b(\cdot,\J)=\nu(\J)K(\J)b(\cdot,\J),
\end{equation}
By Lemma 3.7 from \cite{Wil} (for $d=3$) and Proposition
\ref{prpwkrn} (for $d\ge 4$), if $\gamma_0$ is large enough, for
each $\J\in\Pi_{\gamma_0/2}$ $K(\J)$ is an operator with a
continuous integral kernel $K(\e,\bs,\J)$ such that the mapping
$\J\rightarrow K(\cdot,\cdot,\J)\in C(\Omega\times\Omega)$ is
holomorphic in $\Pi_{\gamma_0/2}$. Observe that this property is
valid also for $d\le 2$, because all the arguments used in
\cite{Wil} are true also in this case. Then for $\J\in\R^d$ the
equality (\ref{eqsingvl}) acquires the form:
\begin{equation}\label{rprex}
b(\e,\J)=\int_\Omega\nu(\J)K(\e,\bs,\J)b(\bs,\J)\,d\bs.
\end{equation}
This equality and the inclusion $b(\cdot,\J)\in L_2(\Omega)$ imply
that $b(\cdot,\J)\in C(\Omega)$. Claim (i) is proven.

(ii) Let us choose $\gamma_0$ such that
$\Oc(\J_0)\subset\Pi_{\frac{\gamma_0}{2}}$. In view of the
properties of $K(\e,\bs,\J)$ mentioned above, the mapping
$\J\rightarrow K(\J)\in\f(L_2(\Omega))$ is holomorphic in
$\Oc(\J_0)$. Taking into account that the function $\nu(\J)$ and the
mapping $\J\rightarrow b(\in,\J)\in L_2(\Omega)$ are holomorphic in
$\Oc(\J_0)$, we obtain by the principle of analytic continuation
that the equality (\ref{eqsingvl}) is valid for any $\J\in\Oc(\J_0)$
Furthermore, the holomorphy in $\Oc(\J_0)$ of $\nu$, of
$\J\rightarrow K(\e,\bs,\J)\in C(\Omega\times\Omega)$ and of
$\J\rightarrow b(\e,\J)\in L_2(\Omega)$ imply that for any
$\J\in\Oc(\J_0)$ the function $\nu(\J)K(\e,\cdot,\J)b(\cdot,\J)$
belongs to the class $C(\Omega,L_2(\Omega))$ and the mapping
$\J\rightarrow\nu(\J)K(\e,\cdot,\J)b(\cdot,\J)\in
C(\Omega,L_2(\Omega))$ is holomorphic in $\Oc(\J_0)$. These
circumstances and equality (\ref{rprex}) imply claim (ii) of the
theorem.
\end{proof}

\subsection{Proof of Corollary \ref{cormainApp}}
\label{subsec:proofcormainclaim}
\begin{proof}
(i) Consider the family of operators $K(\J)$ defined in the proof of
Theorem \ref{thmainApp}. Since $\lambda_0$ is a simple eigenvalue of
the operator $H(\J_0)$, then $\nu_0=(\lambda(\J_0)+\gamma_0^2)^l$ is
a simple singular number of the operator $K(\J_0)$. Let $b_0(\e)\neq
0$ be an eigenfunction of $K(\J_0)$, corresponding to $\nu_0$ (hence
it is an eigenfunction of $H(\J_0)$, corresponding to $\lambda_0$).
Then, as it has been proved in \cite{Bau} (Corollary 1, p. 376),
there exists a neighborhood $\Oc(\J_0)\subset\C^d$ of $\J_0$ such
that in it there is a branch $\nu(\J)$ of singular numbers of the
family $K(\J)\;(\J\in\Oc(\J_0))$ having the properties:
$\nu(\J_0)=\nu_0$, $\nu(\J)$ is simple for any $\J\in\Oc(\J_0)$, the
function $\nu(\J)$ is holomorphic in $\Oc(\J_0)$ and the
corresponding eigenprojections $Q(\J)$ of $K(\J)$ have the property:
the mapping $\J\rightarrow Q(\J)\in\f(L_2(\Omega))$ is holomorphic
in $\Oc(\J_0)$. Hence the eigenfunctions $\tilde
b(\cdot,\J)=Q(\J)b_0$ of $K(\J)$, corresponding to $\nu(\J)$, have
the property: the mapping $\J\rightarrow \tilde b(\e,\J)\in
L_2(\Omega)$ is holomorphic in $\Oc(\J_0)$. Furthermore, we can
restrict the neighborhood $\Oc(\J_0)$ such that $\tilde
b(\cdot,\J)\neq 0$ for any $\J\in\Oc(\J_0)$.  Hence by Theorem
\ref{thmainApp}, $\tilde b(\e,\J)\in C(\Omega)$ for any
$\J\in\Oc(\J_0)$ and the mapping $\J\rightarrow \tilde b(\e,\J)\in
C(\Omega)$ is holomorphic in $\Oc(\J_0)$. Observe that $\nu(\J)$ is
real for any $\J\in\Oc(\J_0)\cap\R^d$, because in this case $K(\J)$
is self-adjoint. Recall that for any $\J\in\R^d$ $JH(\J)J=H(-\J)$,
where $J$ is the conjugation operator $(Jf)(f)(\e):=\overline
{f(\e)}$. Hence $JK(\J)J=K(-\J)$ for any $\J\in\R^d$, therefore in
this case $\sigma(K(\J))=\sigma(K(-\J))$ and for any
$\mu\in\sigma(K(\J))$ the corresponding eigenprojections $Q_\mu$ of
$K(\J)$ and $\tilde Q_\mu$ of $K(-\J)$ are connected in the
following manner: $\tilde Q_\mu=JQ_\mu J$. In particular, for any
$\J\in\Oc(\J_0)$ the number $\nu(\J)$ is a simple characteristic
number of $K(-\J)$ and the operator $JQ(\J)J$ is the eigenprojection
of $K(-\J)$ corresponding to $\nu(\J)$, hence $\overline{\tilde
b(\e,\J)}$ is an eigenfunction of $K(-\J)$ corresponding to
$\nu(\J)$.  These circumstances and the arguments used above imply
that the mapping $\J\rightarrow\overline{\tilde b(\e,\J)}\in
C(\Omega)$ is real analytic in $\Oc(\J_0)\cap\R^d$. Therefore the
function $\Vert b(\cdot,\J)\Vert_2=\Big(\int_{\Omega}\tilde
b(\bs,\J)\overline{\tilde b(\bs,\J)}\,d\e\Big)^{1/2}$ is real
analytic in $\Oc(\J_0)\cap\R^d$. Hence the mapping $\J\rightarrow
b(\e,\J)=\frac{\tilde b(\e,\J)}{\Vert\tilde b(\cdot,\J)\Vert_2}\in
C(\Omega)$ is real-analytic in $\Oc(\J_0)\cap\R^d$. Returning from
the family $K(\J)$ to the family $H(\J)$, we obtain the desired
claim (i) of the corollary.\vskip2mm

(ii) The claim follows from the representation
\begin{equation*}
(Q(\J)f)(\e)=(f,\,b(\cdot,\J))_2b(\e,\J)=\int_\Omega
b(\e,\J)\overline{b(\bs,\J)}f(\bs)\,d\bs\quad (f\in L_2(\Omega))
\end{equation*}
and the results obtained in the process of the proof of claim (i).
\vskip2mm

(iii) Since for any $\J\in\Oc(\J_0)\cap\R^d$, the eigenvalue
$\lambda(\J)$ is simple, then another choice $b_1(\e,\J)$ of the
branch of the eigenfunctions corresponding to $\lambda(\J)$ (with
$\Vert b_1(\cdot,\J))\Vert_2=1$) is connected with the previous one
in the following manner: $b_1(\e,\J)=e^{i\,\theta(\J)}b(\e,\J)$ with
a real-valued function $\theta(\J)$. Then
$b_1(\e,\J)\overline{b_1(\bs,\J)}=b(\e,\J)\overline{b(\bs,\J)}$.
This proves claim (iii).
\end{proof}

\bigskip
\bigskip

Department of Mathematics\hskip 3.5truecm

University of Haifa\hskip 4.9truecm

31905 Haifa, Israel\hskip 5truecm

{\it E-mail address:}\hskip 5.5truecm

zelenko@math.haifa.ac.il

\end{document}